\documentclass[matprg]{svjour3}
\usepackage{dsfont}
\usepackage{booktabs}
\usepackage{makecell}
\usepackage{comment}
\usepackage{endnotes}
\let\footnote=\endnote

\usepackage{tikz}

\usetikzlibrary{positioning,fit,backgrounds}
\usepackage{natbib}
 \bibpunct[, ]{(}{)}{,}{a}{}{,}%
 \def\bibhang{24pt}%

 \usepackage{amsmath,amsfonts,amssymb,graphicx, hyperref, xcolor, float}
\usepackage{mathtools}
\usepackage{algorithm} 
\usepackage{algorithmic}  
\usepackage[algo2e,vlined,linesnumberedhidden,ruled,resetcount]{algorithm2e}
\usepackage{booktabs}
\usepackage[normalem]{ulem}
\usepackage{multirow}
\usepackage{soul}
\definecolor{amber}{rgb}{1.0, 0.49, 0.0}
\DeclareMathAlphabet{\mathpzc}{OT1}{pzc}{m}{it}
\usepackage{enumitem}
\newtheorem{assumption}{Assumption}
\usepackage{setspace}
\usepackage{xr}     % then this
\usepackage{longtable}

\newcommand*{\xbf}{{\mathbf x}}

\newcommand*{\ubf}{{\mathbf u}}
\newcommand*{\ybf}{{\mathbf y}}

\newcommand*{\zbf}{{\mathbf z}}
\newcommand*{\gbf}{{\mathbf g}}
\newcommand*{\vbf}{{\mathbf v}}

\newcommand*{\R}{{\mathbb R}}
\newcommand*{\X}{{\mathcal X}}
\newcommand*{\xhb}{\widehat{\mathbf x}}

\newcommand*{\xhbj}{\widehat{\mathbf x}^{(j)}}

\newcommand*{\E}{{\mathbb E}}

\newcommand*{\Prob}{\Pr}

\usepackage{ulem}
\usepackage{appendix}

\usepackage{booktabs,makecell,array}
 \renewcommand{\arraystretch}{1.15}

\usepackage{xparse}
\makeatletter
\newwrite\stored@out
\immediate\openout\stored@out=stored_sentences.tex
\newcommand{\useStored}[1]{\csname stored@#1\endcsname}
\NewDocumentCommand{\sref}{m g}{%
  \IfNoValueTF{#2}
    {\useStored{#1}}%
    {%
      \expandafter\gdef\csname stored@#1\endcsname{#2}%
      \begingroup
        \protected@edef\stored@tmp{%
          \noexpand\gdef\expandafter\noexpand\csname stored@#1\endcsname
          {\unexpanded\expandafter{#2}}%
        }%
        \immediate\write\stored@out{\stored@tmp}%
      \endgroup
      #2%
    }%
}
\AtEndDocument{\immediate\closeout\stored@out}
\makeatother
    
\begin{document}

\title{Metric Entropy-Free Sample Complexity Bounds for Sample Average Approximation in Convex Stochastic Programming\thanks{Results (Theorems \ref{thm: suboptimality} and \ref{thm: second main theorem convex optimal rate} and their related discussions) in Section \ref{sec: SA-comparable} of this paper  extend  the conference paper  {\it Liu, H., Tong, J. (2024) New Sample Complexity Bounds for Sample Average Approximation in Heavy-Tailed Stochastic Programming}, in {\it Proceedings of the 41st International Conference on Machine Learning (ICML)}, Vienna, Austria. PMLR 235, 2024. {\color{black} In that paper, two key assumptions are imposed:
(i) $f(\cdot,\xi)$ is differentiable on $\X$ for almost every $\xi\in\Xi$, and
(ii) an exact solution to the SAA problem can be computed.
The present work relaxes both assumptions: everywhere differentiability is not always required, and our main results apply to inexact SAA solutions.
Moreover, Sections \ref{sec: Metric entropy-free}, \ref{sec: non-Lipschitzian}, and Appendix \ref{sec: lower bound} develop, respectively, light-tailed results, non-Lipschitz bounds, and lower sample complexity analysis that go beyond the conference version; the corresponding statements, analyses, and proofs  are entirely new. 
Finally, the numerical experiments reported here are also   not included in the ICML paper.} }
}

\author{
    Hongcheng Liu
    \thanks{Center for Applied Optimization \and Department of Industrial and Systems Engineering,
    University of Florida, Gainesville, FL 32611,
       (email: {\tt liu.h@ufl.edu}).}
\and
    Jindong Tong
    \thanks{Center for Applied Optimization \and Department of Industrial and Systems Engineering,
    University of Florida, Gainesville, FL 32611,
       (email: {\tt jindongtong@ufl.edu}). }
}

\maketitle


\begin{abstract}
{This paper studies     sample average approximation (SAA)     in  solving  convex or strongly convex stochastic programming (SP) problems.  In estimating   SAA's sample efficiency, the state-of-the-art sample complexity bounds entail metric entropy terms (such as the logarithm of the feasible region's covering number), which often grow polynomially with  problem dimensionality.  While it has been shown that metric entropy-free complexity rates are attainable  under a  uniform Lipschitz  condition,   such an assumption can be overly critical for many important SP problem settings. In  response, \textcolor{black}{\label{abs rev 1}{this paper  presents   metric entropy-free  sample complexity bounds for the SAA under  standard SP assumptions}}---in the absence of the uniform Lipschitz condition. \textcolor{black}{\label{abs rev 2}{For a $d$-dimensional   problem,}} the new results often lead to an $O(d)$-improvement in the complexity rate compared with the state-of-the-art. From the newly established complexity bounds, an important revelation is that SAA and the canonical stochastic mirror descent (SMD) method, two mainstream solution approaches to SP, entail almost identical rates of sample efficiency,  lifting a theoretical discrepancy of  SAA from  SMD  also by a factor of $O(d)$. Furthermore, this paper explores  non-Lipschitzian scenarios where  SAA maintains provable efficacy but the corresponding results for  SMD remain mostly unexplored, indicating  the potential of  SAA's better applicability   in some irregular settings. The results of our  numerical experiments  align with our theoretical findings. }
\end{abstract}
\keywords{Sample average approximation, heavy tails, stochastic programming, dimensionality} 

\section{Introduction}\label{sec: intro}
This paper is focused on a convex or strongly convex stochastic programming (SP) problem of the following form:
\begin{align}
\min_{\xbf\in\X}\,F(\xbf):=\E[f(\xbf,\xi)],\label{Eq: SP problem statement}
\end{align}
where we let  $(\Omega, \mathcal{F}, \Prob)$ be a complete probability space, $\Xi$ is a Borel subset of $\R^m$, $\xi:(\Omega,\mathcal F)\to(\Xi,\mathcal B(\Xi))$ is an $\Xi$-valued random vector of problem parameters, and $P$ denotes the  induced probability distribution of $\xi$. The cost function $f: B \times \Xi \rightarrow \R$ is deterministic and   $\big(\mathcal B(B)\otimes\mathcal B(\Xi)\big)$-measurable   for some open and convex set $B \subseteq \R^d$, where $d$ is a positive integer denoting the problem dimensionality. The feasible region $\X \subseteq B$ is non-empty, closed, and convex.  Throughout this paper, we suppose that the expectation $\E[f(\xbf,\xi)] = \int_{\Xi} f(\xbf,\zeta)\,\text{d}P(\zeta)$ is well-defined and finite-valued for all $\xbf \in \X$. Furthermore, we assume the population-level cost function $F$ admits a finite minimizer $\xbf^*$ on $\X$. To be formalized subsequently, we also assume that $f(\cdot, \xi)$ is   convex (and in certain settings, strongly convex) on $B$.
 For all $\xbf\in\X$, we also assume the existence of a deterministic and  universally measurable function $G:\X\times\Xi\rightarrow\R^d$   such that  $G(\xbf,\xi)\in \partial_\xbf f(\xbf,\xi)$ for almost every $\xi\in\Xi$. Here, $\partial_{\xbf}f(\xbf,\xi)$ is the subdifferential of $f(\,\cdot\,,\xi)$ at $\xbf\in\X$ for a given $\xi\in\Xi$.   % and we also denote by  $H:\,\X\rightarrow\R^d$  a deterministic function such that  $H(\xbf)\in \partial F(\xbf)$    for all $\xbf\in\X$.%Our results will further impose convexity or strongly convexity to be formalized subsequently.  
% The SP problems like the above are discussed  extensively in  literature \citep[e.g., by][to name only a few]{shapiro2021lectures,birge1997state,birge2011introduction,ruszczynski2003stochastic,lan2020first}.

  The SP problem  above has been widely applied and  much discussed   \citep[e.g., by][to name only a few]{shapiro2021lectures,birge1997state,birge2011introduction,ruszczynski2003stochastic,lan2020first}.      Particularly, it has extensive connections with many statistical and machine learning problems \citep[as per, e.g.,][]{bartlett2006convexity,liu2019sample}. Indeed, the suboptimality gap in solving \eqref{Eq: SP problem statement} can often be interpreted as the excess risk, an important metric of generalizability in statistical and machine learning. %Due to the (increasingly) frequent need to perform data-driven  modeling or decision-making in the presence of  extreme values or outliers in data, studying solution techniques for \eqref{Eq: SAA} under heavy tails becomes growingly more important \citep{oliveira2023sample}.

This paper revisits one of the most traditional but popular solution methods for  SP called   {\it sample average approximation} (SAA). Following its literature \citep[][among many others]{dupacova1988asymptotic,ruszczynski2003stochastic,kleywegt2002sample,shapiro2021lectures,oliveira2023sample,king1991epi}, we particularly focus on  both the canonical formulation of   SAA and one of its simple, regularized variations: 
\begin{itemize}
\item[(i)] {\bf Formulation 1:} In particular, the canonical SAA is as below:
\begin{align}
\min_{\xbf\in\X}\,F_N(\xbf):=N^{-1}\sum_{j=1}^Nf(\xbf,\xi_j),\label{Eq: SAA}
\end{align}
where $\boldsymbol\xi_{1,N}:=(\xi_j:\,j=1,...,N)$ is an i.i.d. random sample of $\xi$. % \textcolor{black}{We deleted this}.  
Our analysis of this formulation is  centered around   its effectiveness for strongly convex SP problems. Here,   positive integer $N$ is the sample size.
\item[(ii)]  {\bf Formulation 2:} On top of \eqref{Eq: SAA}, we also consider   the SAA variation that incorporates a Tikhonov-like regularization    in the following:
\begin{align}
\min_{\xbf\in\X}\,F_{\lambda_0,N}(\xbf):=F_N(\xbf)+\lambda_0V_{q'}(\xbf),\label{Eq: SAA-ell2}
\end{align}
where  $\lambda_0\geq 0$ is a tuning parameter, and $V_{q'}:\X\rightarrow\R_+$,     given a user's choice of $q'$-norm (with $q'\in(1,2]$), is defined as
\begin{align}
V_{q'}(\xbf)=\frac{1}{2}\Vert\xbf-\xbf^0\Vert_{q'}^2,\label{penalty term}
\end{align}
for any initial guess $\xbf^0\in\X$ (though many of our results can be easily extended to any  choice of  $\xbf^0\in\R^d$ \textcolor{black}{\label{possibly feasible x0} and possibly $\xbf^0\notin \X$}).   Particularly in the case of $q'=2$ and $\xbf^0=\mathbf 0$,   we have  $V_{q'}(\xbf)=0.5\Vert\xbf\Vert_2^2$, which becomes the  canonical Tikhonov regularization \citep{tikhonov1995numerical,golub1999tikhonov} commonly employed in ridge regression \citep{hoerl1970ridge}. The same type of regularization approach has been discussed in the SAA theories for (non-strongly) convex SP, among others,  by \citet{hu2020sample},  \citet{feldman2019high},  \citet{shalev2010learnability,shalev2009stochastic} and \citet{lei2020sharper}. Similarly in this paper, we also study SAA \eqref{Eq: SAA-ell2} in  convex SP problems.  Because  $V_{q'}$ is  $(q'-1)$-strongly convex w.r.t. the $q'$-norm, additionally incorporating this term as in \eqref{Eq: SAA-ell2}
would not add significant computational burden to solving the \textcolor{black}{SAA problem.} \textcolor{black}{\label{revision Bregman divergence} Due to   strong convexity and differentiability of $V_{q'}$, it can serve as a distance generating function for Bregman divergence \citep{dhillon2008matrix}. While we conjecture that a wider class of distance generating functions can be employed to construct regularization terms in \eqref{Eq: SAA-ell2}, we would like to leave the detailed analysis to future research.}
\end{itemize}
 Throughout this paper, we   refer to both formulations \eqref{Eq: SAA} and \eqref{Eq: SAA-ell2} as  SAA; or, if  there is ambiguity,  ``SAA  \eqref{Eq: SAA}'' and ``SAA  \eqref{Eq: SAA-ell2}'', respectively. Both versions of SAA   avoid the multi-dimensional integral involved in \eqref{Eq: SP problem statement} and thus render the   SP problem to be solvable as a ``deterministic'' nonlinear program \citep{shapiro2021lectures}, often leading to a substantial  improvement in tractability. 

%Relative to the literature, our choice of regularization $V_q$ in \eqref{Eq: SAA-ell2} is comparably more general and our results can provide  new insights   in   non-Lipschitzian, smooth, or high-dimensional cases.%In contrast, this paper presents results generalizes these previous results to non-Lipchitzian settings and %Recently, a special form  $V(\cdot)=0.5\cdot \Vert\cdot\Vert_2^2$ has been studied in SAA for convex SP by \cite{}. %We show in this paper that such regularization  promotes the efficacy of SAA in (general) convex SP problems. 
 
Hereafter, for a   given  sample $\boldsymbol\xi_{1,N}$  and given scalars $q\geq 1$ and $\delta>0$,  we refer to the random vector  $\xhb:=\widetilde{\xbf}(\boldsymbol\xi_{1,N})\in\X$  as a {\it $(\delta,q)$-approximate solution} to SAA \eqref{Eq: SAA} (or SAA \eqref{Eq: SAA-ell2}) if  $\widetilde{\xbf}:\,\Xi^N\rightarrow\X$ is a deterministic and $\mathcal B(\Xi^N)$-measurable function
such that
\begin{align}
 \left\langle \mathbf g,\,\xbf-\widetilde{\xbf}(\boldsymbol\xi_{1,N})\right\rangle\geq -\delta\cdot \left\Vert\xbf-\widetilde \xbf(\boldsymbol\xi_{1,N})\right\Vert_q,~~\forall \xbf\in\X,\label{define solution}
\end{align}
for some $\gbf\in \partial F_N(\widetilde{\xbf}(\boldsymbol\xi_{1,N}))$ (or $\gbf\in \partial F_{\lambda_0,N}(\widetilde{\xbf}(\boldsymbol\xi_{1,N}))$, respectively).   Here $\partial F_N(\xbf)$ denotes the subdifferential of $F_N$ at  $\xbf\in\X$. Clearly, under convexity, $\widetilde{\xbf}(\boldsymbol\xi_{1,N})$ must also satisfy the following:
  \begin{align}
  \begin{cases}
  ~~~~~~\,F_N(\xbf)-F_N(\widetilde{\xbf}(\boldsymbol\xi_{1,N}))\geq -\delta\cdot \Vert \xbf-\widetilde{\xbf}(\boldsymbol\xi_{1,N})\Vert_q,~\forall\xbf\in\X,&\text{for SAA  \eqref{Eq: SAA}};
  \\[5pt]
  F_{\lambda_0,N}(\xbf)-F_{\lambda_0,N}(\widetilde{\xbf}(\boldsymbol\xi_{1,N}))\geq -\delta\cdot \Vert \xbf-\widetilde{\xbf}(\boldsymbol\xi_{1,N})\Vert_q,~\forall\xbf\in\X,&\text{for SAA \eqref{Eq: SAA-ell2}}.
\end{cases}\label{suboptimality in computing SAA}
  \end{align}
  By \eqref{suboptimality in computing SAA}, if  $\delta=0$,   then $\xhb:=\widetilde{\xbf}(\boldsymbol\xi_{1,N})$ becomes an exact optimal solution to SAA, denoted by $\xhb^{\textnormal{opt}}$ hereafter.
  
 The measurability of    optimal solutions to SAA has been previously studied \citep[e.g., by][]{shapiro2021lectures,rockafellar2009variational,kratschmer2023first}. Meanwhile, the quality of solution $\xhb$ in approximating a solution to the genuine SP problem in \eqref{Eq: SP problem statement} has also been much studied  \citep[e.g., by][]{artstein1995consistency, dupacova1988asymptotic,king1993asymptotic,king1991epi,pflug1995asymptotic,pflug1999stochastic,pflug2003stochastic,shapiro1989asymptotic,shapiro1993asymptotic,shapiro2003monte,shapiro2021lectures,liu2016dimensionality,liu2022high}.  \textcolor{black}{\label{SAA nonconvex R1} Many of these results are general enough to apply to inexact solutions of SAA problems in solving nonconvex stochastic programs.} % The properties of $\xhb^{\textnormal{opt}}$  can be immediately implied by our results on $\xhb$.
 Following many works in this body of literature, this  current paper  is focused on SAA's (non-asymptotic) sample complexity; namely, how large the sample size $N$ should be in order to ensure that   $\xhb$ approximates an optimal solution $\xbf^*$ to \eqref{Eq: SP problem statement}, meeting the user-prescribed thresholds for  accuracy and probability. Although this is a well-visited topic,  existing sample complexity bounds  commonly carry an undesirable term of metric entropy, such as the example below. 

   % in terms of its sample complexity; how large the sample size $N$ should be in order to ensure that   $\xhb$ approximates the optimal solution $\xbf^*$ to \eqref{Eq: SP problem statement} with accuracy $\epsilon$ and probability at least $1-\beta$, for a user-specified accuracy threshold $\epsilon>0$ and a significance level $\beta\in(0,1)$.
   % These include the seminal result below:

 \smallskip
 
 {\bf Typical non-asymptotic sample complexity under light tails \citep[when results, e.g., by][are applied to our settings]{shapiro2021lectures,shapiro2003monte,shapiro2005complexity}:}  {\it 
Under the 
 Lipschitz assumption that,  for all $\xbf,\ybf\in\X$ and   every $\xi\in\Xi$,
 \begin{align}
\vert f(\xbf,\xi)-f(\ybf,\xi)\vert\leq M(\xi)\cdot \Vert\xbf-\ybf\Vert_q,\label{Lipschitz condition}
\end{align}
where $M:\Xi\rightarrow\R_+$ is some deterministic and measurable function and $\Vert\cdot\Vert_q$ is the $q$-norm ($q\geq 1$), the optimal solution $\xhb^{\textnormal{opt}}$ to   SAA   \eqref{Eq: SAA} satisfies the following: For any given $\epsilon>0,\,\,\beta\in(0,1)$:
\begin{align}
\Prob\left[F(\xhb^{\textnormal{opt}})-F(\xbf^*)\leq \epsilon\right]\geq 1-\beta, ~~\text{if }N\geq O\left(\frac{\varphi^2 \mathcal D_q^2}{\epsilon^2} \left[    \Gamma_\epsilon(\X)+\ln\frac{1}{\beta}\right] \right).\label{summary typical results}
\end{align}
Here, $ \Gamma_{\epsilon}(\X)$ \textcolor{black}{may be} the logarithm of the covering number of $\X$---one common form of the said metric entropy, $\mathcal D_q$ is the $q$-norm diameter of $\X$, and $\varphi$  is the parameter related to the sub-Gaussian   distribution assumed for    $M(\xi)$.} 

 \smallskip

 Other than the logarithm of covering number above, in some other existing works,   metric entropy  may be calculated in different ways, including the ``generic chaining''-based complexity measures as in the analysis by  \cite{oliveira2023sample}. Regardless of how the metric entropy is quantified, it generally causes the sample complexity bounds to exhibit an elevated  dependence on problem dimensionality $d$.  For example,  because the covering number grows exponentially with $d$  in general, the logarithm of covering number $\Gamma_{\epsilon}(\X)$ is then polynomial in $d$, leading to the following more explicit representation of the complexity rate than \eqref{summary typical results}:
\begin{align}
O\left(\frac{\varphi^2 \mathcal D_q^2}{\epsilon^2} \left[   d\cdot  \ln\left(\frac{\mathcal D_{q}\varphi}{\epsilon}\right)+\ln\frac{1}{\beta}\right] \right).\label{reduced rate}
\end{align}
Changing ways to account for the metric entropy, such as via the said ``generic chaining''-based argument, does not avoid the aforementioned polynomial growth rate with $d$, except when some special   structural assumptions hold (e.g.,  the feasible region is a simplex). While  uniform stability  arguments for SAA have led to sample bounds that are free of  metric entropy terms,  these arguments rely on a uniform Lipschitz condition, which is often restrictive for many SP problems. (See Section \ref{sec: related results} for more detailed discussions.)

The influence of  metric entropy such as in \eqref{summary typical results} (and thus in \eqref{reduced rate}) seems particularly overt when  SAA is compared with the stochastic mirror descent (SMD) methods (also known as the \textcolor{black}{robust stochastic approximation} or the stochastic first-order methods) as discussed, e.g., by \cite{nemirovski2009robust,ghadimi2013stochastic}, and \cite{lan2020first}.  SMD is a  mainstream alternative  to   SAA in solving an SP problem. For a convex or strongly convex SP problem under comparable, and sometimes  weaker, assumptions than those by the existing   SAA theories,   SMD can often achieve metric entropy-free sample complexity rates.  As a consequence, comparisons between  the current theories for both methods would suggest a significant performance gap  in sample efficiency---SMD would be substantially more efficient than    SAA by a margin of order $O(d)$.  With such a stark contrast, one would  expect SMD  to exhibit a significantly better solution quality   than  SAA in practice, when both methods are operating under the same sample size constraints. Yet, this theoretical discrepancy, though being persistent in the literature, seems to have never been confirmed in empirical studies. On the contrary, in many reported experiments \citep[such as  those  by][]{nemirovski2009robust},  SAA achieves comparable, if not better, solution accuracy than  SMD, given the same sample size. The aforementioned inconsistency between the theoretical predictions and numerical results indicates a critical   literature gap, to which this paper seeks to respond by studying a directly related    question:

\smallskip

{{\bf The open question of our focus:} \it   Under typical assumptions in the SP literature, does SAA, or its simple variations, admit sample complexity bounds that are completely free from any  quantification of   metric entropy?}

\smallskip

\noindent  As mentioned, while this   question has been studied under a uniform Lipschitz condition \citep[e.g., by][]{bousquet2002stability,shalev2009stochastic,shalev2010learnability,feldman2018generalization,feldman2019high}, this assumption tends to be restrictive for many important SP formulations such as stochastic linear programming with Gaussian cost coefficients. (See Section \ref{sec: related results} for more discussions).  In a standard SP setting---which is typically in the absence of the uniform Lipschitz condition or its comparable variations---the said research question remains largely open. In response, this current paper presents perhaps the first  affirmative answer  through the following three sets of results on $\xhb$ with some   small enough  $\delta\geq 0$ as defined in \eqref{define solution}:
\begin{itemize}
\item[(a)] We show, perhaps for the first time, that   SAA's sample complexity matches     that of the canonical SMD \citep{nemirovski2009robust,ghadimi2013stochastic,lan2020first}, under comparable assumptions commonly imposed for the latter.  More specifically, when the subgradient 
 of $f(\cdot,\xi)$  admits a bounded variance of no more than  $\sigma_p^2$---with the underlying randomness being heavy-tailed---and the population-level objective function $F$ consists of an $\mathcal L$-smooth term and an $\mathcal M$-Lipschitz term, we prove in Theorems \ref{thm: suboptimality}  and \ref{thm: second main theorem convex optimal rate} of Section \ref{sec: SA-comparable} that, for any given $\epsilon\in(0,\,1]$, 
\begin{multline}
\mathbb E[F(\xhb)-F(\xbf^*)]\leq \epsilon,
\\\text{if~}N\geq \begin{cases}
O\left( \max\left\{\frac{\mathcal L}{\mu},~\frac{\sigma_p^2+\mathcal M^2}{ \mu\cdot \epsilon}\right\}\right)&\text{for SAA \eqref{Eq: SAA} in  $\mu$-strongly convex SP};
\\[8pt]
O\left(\frac{V_{q'}(\xbf^*)}{q'-1}\cdot \max\left\{  \frac{\mathcal L }{\epsilon},~\frac{\sigma_p^2+\mathcal M^2}{\epsilon^2}   
\right\}\right)&\text{for SAA \eqref{Eq: SAA-ell2} in (non-strongly) convex SP},
\end{cases}\label{summary first result}
\end{multline}
where we recall the definition of $V_{q'}$ in \eqref{penalty term}. 
This set of results is identical to the best-known sample complexity bounds for the conventional SMD, providing perhaps the first theoretical explanation on the comparable empirical performance between   SAA and SMD. A summary of their comparisons is presented in Table \ref{table summary results compare with SMD}. Furthermore,  \eqref{summary first result} also exhibits advantages when compared with the  existing SAA's sample complexity benchmark   by \cite{oliveira2023sample}. In particular,   with the  aforementioned assumption on the boundedness of variance, our results    apply to the ``most heavy-tailed'' scenario considered by \cite{oliveira2023sample}. In this scenario, as explained in Remark \ref{remark: compare to OT p=2} later, \eqref{summary first result} significantly improves over the said benchmark   in terms of the rate with   dimensionality $d$ and, for some cases, additionally in terms of the rate   with      $\mathcal M$.

%The rates in \eqref{summary first result} are also more advantageous than the concurrent results on the SAA in comparable settings.
%The results are reported in Theorems \ref{} and \ref{} in Section ???

\item[(b)] In Theorem \ref{Most explicit result here}  of Section \ref{sec: Metric entropy-free}, under comparable conditions as in the typical non-asymptotic results for  SAA, we prove what seems to be the first, large deviations-type sample complexity bound  completely free from any metric entropy terms in the light-tailed settings: 
Suppose that \eqref{Lipschitz condition} holds and $M(\xi)$ therein is  sub-Gaussian  with a parameter $\varphi>0$. For any $\epsilon\in(0,1]$ and $\beta\in(0,1)$, the solution $\xhb$ to SAA \eqref{Eq: SAA-ell2} satisfies that
\begin{align}
\Prob\left[F(\xhb)-F(\xbf^*)\leq \epsilon\right]\geq 1-\beta,~~
\text{if~}N\geq O\left(\frac{  \mathcal D_{q'}^2\cdot\varphi^2}{(q'-1)\cdot \epsilon^2} \cdot  \ln \left(\Phi\ln\frac{e}{\beta}\right)\cdot\ln^2\frac{\Phi}{\beta} \right),\label{metric entropy free bounds intro}
\end{align}
where $\Phi:=  {\varphi \mathcal D_{q'}\cdot e}/{[ (q'-1)\cdot \epsilon}] $ and $\mathcal D_{q'}$ is the $q'$-norm 
 diameter of the feasible region for  the same $q':\,q'\leq q$ (with $q$ defined    in \eqref{Lipschitz condition}) as a hyperparameter of SAA   \eqref{Eq: SAA-ell2}. Compared to the typical SAA results, such as those in \eqref{summary typical results} and \eqref{reduced rate}, the avoidance of  metric entropy in \eqref{metric entropy free bounds intro} leads to a significantly better growth rate with problem dimensionality $d$.  %Furthermore,  one can observe that the said independence from metric entropy in \eqref{metric entropy free bounds intro} is achieved with only some small trade-offs:  there is a poly-logarithmic  increase in the dependence on $1/\beta$ and a few other problem quantities.  
 Also, in comparison with the existing large deviations bounds for  canonical SMD as discussed by \cite{nemirovski2009robust}, Eq.\,\eqref{metric entropy free bounds intro} presents an almost identical rate up to some (poly-)logarithmic terms. (See  Table \ref{table summary results compare with SMD} for a summary of comparisons with SMD and see further discussions    in Remark \ref{comparison results to be added} later). Our results in Proposition \ref{thm: large deviation} and Theorem \ref{Most explicit result here} (both in Section \ref{sec: Metric entropy-free}) further extend \eqref{metric entropy free bounds intro} to metric entropy-free bounds for both the sub-exponential settings (where the tails of the underlying distribution vanish at least as fast as an exponential random variable) and the heavy-tailed settings (namely, when the $p$th central moment of $M(\xi)$ is bounded for some given $p>2$). %These results complement our findings for the case of  bounded variance (and thus $p=2$) in (a) above.

\item[(c)] In Theorems \ref{thm: first main theorem} and \ref{thm: first main theorem convex} of Section \ref{sec: non-Lipschitzian}, we additionally identify cases where   SAA's theoretical efficacy may even outperform those of SMD. In particular, we consider a non-Lipschitzian scenario where neither $F$ nor its (sub)gradient admits a known upper bound on the Lipschitz constant. In this case, except for some recent SAA results  by \cite{milz2023sample} that apply to a special case of our discussion, whether SMD or SAA can still be effective   seems largely unknown from the literature thus far. In response, we show that, for any $\vartheta>0$ and $\beta\in(0,1)$, when the $\mu$-strong convexity holds w.r.t.\ the $q$-norm ($q> 1$) and  the $p$th central moment of $G(\cdot,\xi)$ is bounded by $\psi_p$ (for some $p<\infty$ such that $1\leq p\leq q/(q-1)$), the   solution $\xhb$  to SAA \eqref{Eq: SAA} satisfies that
\begin{align}
\Prob\left[\Vert \xhb -\xbf^*\Vert_q^2\leq \vartheta\right]\geq 1-\beta,~~~\text{if~}N\geq O \left(\frac{p\cdot \psi_p^2}{\mu^2\cdot \vartheta}\cdot \beta^{-2/p}\right).\label{strong convex case nonlip}
\end{align}
Such a complexity bound does not depend on any Lipschitz constant.
Meanwhile,   in the  (non-strongly) convex case, for any given $\epsilon>0$, $\vartheta>0$ and $\beta\in(0,1)$, the  solution $\xhb$   to SAA \eqref{Eq: SAA-ell2}  with   $q'\in(1,\,2]$ satisfies that
\begin{align}
\Prob\left[\Vert \xhb-\xbf^*_{\epsilon}\Vert_{q'}^2\leq \vartheta\right]\geq 1-\beta,~~~\text{if~}N\geq O \left(\frac{ p\cdot\psi_{p}^2\cdot [\textcolor{black}{V_{q'}(\xbf^*)}]^2}{ (q'-1)^2\cdot   \epsilon^2\cdot \vartheta}\cdot \beta^{-\frac{2}{p}}\right),\label{general convex case nonlip}
\end{align}
where $\xbf_\epsilon^*$ is some $\epsilon$-suboptimal solution to the genuine SP problem  \eqref{Eq: SP problem statement}. Here, \eqref{general convex case nonlip} provides the sample requirement for $\xhb$ to reside in some $\vartheta$-neighborhood of an $\epsilon$-suboptimal solution with probability at least $1-\beta$ and is, again, invariant to any Lipschitz constant.
\end{itemize}

We would like to reiterate that all our theoretical results share the advantage of being independent from any form of metric entropy, implying the SAA's innate, SMD-comparable  dimension-insensitivity.    Under the common assumptions of SP literature (which typically are beyond the uniform Lipschitz condition), the same level of dimension-insensitivity, to our knowledge, has not been uncovered thus far.

While \cite{guigues2017non} and   \cite{shapiro2021lectures} have shown that  SAA's polynomial dependence on $d$   is unavoidable in general, their lower sample complexity bounds are not at odds with our findings. (See more detailed discussions   in Appendix \ref{sec: lower bound}).
{\color{black}\label{front-load} Indeed, a closer examination would reveal that quantities --- such as $\sigma_p$,   $\varphi$, $\mathcal M$, and $\psi_p$ in our bounds --- may all depend on $d$ (implicitly) at some polynomial growth rate in less favorable scenarios.} Nonetheless, our results also point to important special cases, such as  illustrated in both Remark \ref{rk: variance} and Appendix \ref{sec: lower bound}, where $d$ is allowed to be  larger  than $N$ (sometimes substantially).  

Our observations in numerical experiments with synthetic data in Section \ref{sec: simulated data} are consistent with our theoretical predictions. These observations also point to the connections between the performance of SAA and the configuration of the hyperparameters introduced through the Tikhonov-like regularization in \eqref{penalty term}.

\subsection{Organization}\label{sec: organizations} The rest of this paper is organized as follows: Section \ref{sec: related results} summarizes related works. Our main theorems   are presented in Section \ref{sec: main result} and our numerical results are provided in Section \ref{sec: simulated data} with  additional details on parameters discussed in Appendix \ref{details SMD stepsize}. % shows some additional details on parameter selections in the numerical experiments. % and their proofs are provided in Section \ref{sec: main result} and Appendix \ref{sec: add proofs}. 
 Section \ref{sec: ccln} concludes the paper.  Some useful lemmata, propositions, and the proof of a corollary are presented in Appendices \ref{sec: useful lemma 2} and \ref{proof of corollary 1}. Appendix \ref{sec: preliminary}   discusses some preliminary properties of $V_{q'}$ as a component of SAA   \eqref{Eq: SAA-ell2}. Appendix \ref{sec: lower bound} presents comparisons of our results with  lower sample complexity bounds.
  
\subsection{Notations}\label{sec: notations}
Denote by $\R$   the collection of all real   numbers, and by $\R_+$ and $\R_{++}$ those of the non-negative and strictly positive real numbers, respectively.  
$\mathbf 0$ is the all-zero vector of  some proper dimension.  We at times use $(x_i)$ or $(x_i:\,i=1,...,d)$ to denote a vector $
\xbf=(x_1, \cdots, x_d)^\top\in\R^d$ for convenience. 
  For a function $g$, denote by $\nabla  g$  the gradient and by $\nabla_i g$  its $i$th  element. We also denote by  $\partial g$ (or $\partial_\xbf g$)  the subdifferential of $g$ (w.r.t. its argument $\xbf$).  We let $e$ be the base of natural logarithm.  For any vector $\vbf=(v_i:\,i=1,...,d)\in\R^d$, denote  by $\Vert  {\vbf}\Vert_p:= (\sum_{i=1}^d \vert v_i\vert^p)^{1/p}$ the $p$-norm ($p\geq 1$), while  $\Vert \vbf\Vert_{\infty} = \max_{i=1,\cdots,d}\vert v_i\vert$. Meanwhile, we define the  $L^p$-norm of a random vector $\boldsymbol\zeta=(\zeta_i)\in\mathbb R^d$ to be $\Vert \boldsymbol\zeta\Vert_{L^p}:= (\sum_{i=1}^d\E_{\zeta_i} [ \vert \zeta_i\vphantom{V^{V}} \vert^p] )^{1/p}$. %For any two random variables $a$ and $b$, we use the notation $a\perp b$ to indicate their independence. 
  For any random variable/vector $y$, we also denote by $\E_{y}[\,\cdot\,]$ the expectation of ``$\cdot$'' over $y$, except that $\E[\,\cdot\,]$ denotes the expectation over all the randomness in ``$\cdot$''.      Finally, ``w.r.t.''  and ``a.s.'' are short-hands for ``with respect to'' and  ``almost surely'', respectively. \textcolor{black}{\label{note dual}  For $q\geq 1$, we often denote by $\varrho$-norm the dual of $q$-norm with $\varrho=q/(q-1)$, following the usual convention that  \(\varrho=\infty\) when \(q=1\).}
\begin{table}[H]
{\color{black}
\centering
\scriptsize
\setlength{\tabcolsep}{2.5pt}
\caption{A summary of key results and their comparison with   SMD's  sample complexity results \citep[as per][]{nemirovski2009robust,lan2020first}  in  the same settings. All the presented bounds use the accuracy measure of the suboptimality gap $\epsilon$ or the distance $\vartheta$ from a target solution. This target solution is an optimal solution for most results and an $\epsilon$-suboptimal solution for Theorem \ref{thm: first main theorem convex}. Results labeled ``{\it E:}'' are bounds for controlling expected accuracy, while results labeled ``{\it P:}'' are probabilistic bounds for controlling accuracy with significance level $\beta$. In the table,  $\X^{*,\epsilon}$ is the set of $\epsilon$-suboptimal solutions,  $q'$ is a hyperparameter of SAA \eqref{Eq: SAA-ell2}, problem quantity $\Phi:= {e\varphi \mathcal D_{q'}}/{[ (q'-1)\cdot \epsilon}] $, and $\mathcal D_{q'}\approx \mathcal D_{q}$ with $\mathcal D_{q'}$ (or $\mathcal D_{q}$) being the $q'$-norm (or $q$-norm, resp.) diameter of the feasible region $\X$. Meanwhile, $V_{q'}(\xbf^*)$ is also comparable to $\mathcal D_{q}^2$ up to constants.}\label{table summary results compare with SMD}
\begin{center}
\begin{tabular}{@{}l
  >{\raggedright\arraybackslash}p{0.32\linewidth}   % Assumptions (+ A-refs)
  >{\raggedright\arraybackslash}p{0.30\linewidth}   % This paper
  >{\raggedright\arraybackslash}p{0.16\linewidth}   % SMD
  >{\raggedright\arraybackslash}p{0.18\linewidth}   % AC-SA
@{}}
\toprule
\textbf{Result} & \textbf{Assumptions} & \textbf{This paper} & \textbf{SMD} \\
\midrule

Thm.~\ref{thm: suboptimality}
& $\mathcal L+\mathcal M$ (A\ref{L-smoothness});\; Var.\,$\le \sigma_p^2$ (A\ref{assumption: Variance everywhere});\; SC (A\ref{SC condition constant all})
& \makecell[l]{\emph{E: }$\underbrace{\tfrac{L}{\mu} \vee \tfrac{\sigma_p^2{+}\mathcal M^2}{\mu\,\epsilon}}_{\text{SMD-matching}}$\\
\vspace{-1.5mm}
\\
               \emph{P: } $\underbrace{\tfrac{L}{\mu} \vee \tfrac{\sigma_p^2{+}\mathcal M^2}{\mu\,\epsilon\,\beta}}_{\text{SMD-matching}}$}
& \makecell[l]{\emph{E: }$ \tfrac{L}{\mu} \vee \tfrac{\sigma_p^2{+}\mathcal M^2}{\mu\,\epsilon} $\\
\vspace{1.5mm}
\\
               \emph{P: } $ \tfrac{L}{\mu} \vee \tfrac{\sigma_p^2{+}\mathcal M^2}{\mu\,\epsilon\,\beta}$}
 
\\[2mm]\hline\addlinespace[3pt]

Thm.~\ref{thm: second main theorem convex optimal rate}
& $\mathcal L+\mathcal M$ (A\ref{L-smoothness});\; Var.\,$\le \sigma_p^2$ (A\ref{assumption: Variance everywhere});\; Conv.\ (A\ref{GC condition constant all})
& \makecell[l]{\emph{E: }$ \dfrac{1}{q'-1}\cdot \underbrace{ V_{q'}(\xbf^*)\!\cdot\!\big(\tfrac{L}{\epsilon} \vee \tfrac{\sigma_p^2{+}\mathcal M^2}{\epsilon^2}\big)}_{\text{SMD-matching}}$\\
\vspace{-1.5mm}
\\
               \emph{P: } $ \dfrac{1}{q'-1}\cdot \underbrace{V_{q'}(\xbf^*)\!\cdot\!\big(\tfrac{L}{\epsilon} \vee \tfrac{\sigma_p^2{+}\mathcal M^2}{\epsilon^2\beta}\big)}_{\text{SMD-matching}}$}
& \makecell[l]{\emph{E: }$\mathcal D_{q}^2\!\cdot\!\big(\tfrac{L}{\epsilon} \vee \tfrac{\sigma_p^2{+}\mathcal M^2}{\epsilon^2}\big)$\\
\vspace{1.5mm}
\\
               \emph{P: } $ \mathcal D_{q}^2\!\cdot\!\big(\tfrac{L}{\epsilon} \vee \tfrac{\sigma_p^2{+}\mathcal M^2}{\epsilon^2\beta}\big)$}
\\[3mm]\hline\addlinespace[3pt]

Thm.~\ref{Most explicit result here}(b)
& Conv.;\; $M(\xi)$-Lip.\ (A\ref{assumption f lipschitz v3}(a));\; subGauss.\,$M(\xi)$ (A\ref{assumption f lipschitz v3}(b))
& \makecell[l]{\emph{P: }$\dfrac{1}{q'-1} \cdot \underbrace{\dfrac{\mathcal D^2_{q'}\,\varphi^2}{\epsilon^2} \ln^2\!\tfrac{\Phi}{\beta}}_{\substack{\text{SMD-matching}\\\text{up to $\ln^2 \Phi$}}}\cdot  \ln \!\left(\Phi\ln\tfrac{e}{\beta}\right)$}
& \makecell[l]{\emph{P: }$ \dfrac{\mathcal D_{q}^2\varphi^2}{\epsilon^{2}}\cdot \ln^2\!\tfrac{1}{\beta} $}
 
\\[2.5mm]\hline\addlinespace[3pt]

Thm.~\ref{thm: first main theorem}
& loc-SC and $p$-mom.\ at $\xbf^*$ (A\ref{SC condition constant})
& \makecell[l]{\emph{E: }$ \tfrac{p\,\psi_p^2}{\mu^2\,\vartheta}$\\
\vspace{-1.5mm}
\\
               \emph{P: } $ \tfrac{p\,\psi_p^2}{\mu^2\,\vartheta}\,\beta^{-2/p}$}& \makecell[l]{Not known}

\\[5mm]\hline\addlinespace[3pt]

Thm.~\ref{thm: first main theorem convex}
& Conv.\ (A\ref{GC condition constant all});\; $p$-mom.\ on $\X^{*,\epsilon}$ (A\ref{assumption: Variance suboptimality})
& \makecell[l]{\emph{E: }$ \dfrac{p\,\psi_p^2 V^2_{q'}(\xbf^*)}{(q'-1)^2\,\epsilon^2\,\vartheta}$\\
\vspace{-1.5mm}
\\
               \emph{P: } $ \dfrac{p\,\psi_p^2 V^2_{q'}(\xbf^*)}{(q'-1)^2\,\epsilon^2\,\vartheta}\cdot \beta^{-2/p}$}
 
 & \makecell[l]{Not known}
 
\\[6mm]
\bottomrule\addlinespace[3pt]
\end{tabular}
\vspace{1mm}

\begin{minipage}{0.96\linewidth}\footnotesize
\noindent\textbf{Abbreviations in the table:} ``A\#'' stands for Assumption \#; ``$\mathcal L+\mathcal M$'', a composite objective with an $\mathcal L$-smooth term and an $\mathcal M$-Lipschitz term; ``SC'', strong convexity; ``Var'', variance; ``Conv.'', (non-strong) convexity; ``$M(\xi)$\!-Lip.'', $M(\xi)$-Lipschitz continuity; ``loc-SC'', local strong convexity at $\xbf^*$; ``subGauss.'', sub-Gaussian; and ``$p$-mom.'', finite $p$-th moment (at $\xbf^*$ or on $\X^{*,\epsilon}$ when specified).
\end{minipage}
\end{center}

}
\end{table}

 \section{Related work}\label{sec: related results}
There is a rich body of literature on \eqref{Eq: SAA} and \eqref{Eq: SAA-ell2}. Asymptotic analysis of  SAA has been provided by, e.g., \cite{artstein1995consistency, dupacova1988asymptotic,king1993asymptotic,king1991epi,pflug1995asymptotic,pflug1999stochastic,pflug2003stochastic}, and \cite{shapiro1989asymptotic}.   {\label{additional reference covering number}\color{black}Stability estimates of SAA and the covering number approach have been studied by \cite{kavnkova1978approximative}.} Non-asymptotic (finite-sample) complexity bounds are also made available by   works of, e.g., \cite{shapiro2003monte,shapiro2005complexity} and \cite{shapiro2021lectures} in light-tailed settings and the  work of, e.g., \cite{oliveira2023sample} in heavy-tailed settings.   However, to our knowledge,  the state-of-the-art sample complexity results in both light-tailed or heavy-tailed scenarios  carry metric entropy terms that measure the complexity of the feasible region, such as $\Gamma_\epsilon(\X)$ in \eqref{summary typical results}. These terms  typically grow rapidly with $d$, elevating the dependence of  SAA's predicted sample requirement on the problem dimensionality.

Sample complexities free from the said metric entropy terms have been made available for machine learning algorithms, e.g., by \cite{shalev2010learnability,shalev2009stochastic,bousquet2002stability,feldman2018generalization,feldman2019high} and \cite{klochkov2021stability}. Their results can also provide metric entropy-free complexity bounds for   SAA  under the stipulation of   more critical conditions. More specifically, through the argument of uniform stability or its variations, 
it has been proven \citep[e.g., by][]{bousquet2002stability,shalev2010learnability,shalev2009stochastic,hu2020sample} that an optimal solution $\xhb^{\textnormal{opt}}$ to  SAA satisfies the below:
\begin{align*}
\E\left[F(\xhb^{\textnormal{opt}})-F(\xbf^*)\right]\leq \epsilon , ~~
\text{if }~
N\geq \begin{cases} O(\frac{M}{\mu\epsilon}) &\text{for SAA \eqref{Eq: SAA} in $\mu$-strongly convex SP};
\\ 
\\
O(\frac{M V_{q'}(\xbf^*)}{\epsilon^2})&\text{for SAA \eqref{Eq: SAA-ell2} in  convex SP,}
\end{cases}
\end{align*}
where $V_{q'}(\cdot)$ is defined as in \eqref{penalty term} with  $\xbf^0=\mathbf 0$,
when the following, what we call, {\it uniform Lipschitz condition}  holds: 
\label{original clarification of ULC}For all $\xbf,\,\ybf\in\X$ and $\xi\in\Xi$, there exist $q\geq q'$ and some $\xi$-invariant constant $M>0$ such that
\begin{align}
\vert f(\xbf,\xi)-f(\ybf,\xi)\vert\leq M\cdot    \Vert\xbf-\ybf\Vert_{q}.\label{lipschitz global}
\end{align}
Furthermore, high probability bounds   that are logarithmic in $1/\beta$ are also obtained  under the same   condition  \citep[e.g., by][when their results are applied to the analysis of \eqref{Eq: SAA} or \eqref{Eq: SAA-ell2}]{feldman2018generalization,feldman2019high,bousquet2002stability,feldman2018generalization,klochkov2021stability}.  
 Nonetheless,     the uniform Lipschitz condition  in \eqref{lipschitz global}   can be overly critical for many applications of the SP;  because $M$   is $\xi$-invariant, this quantity  can  be undesirably large and even unbounded under the more common conditions, such as   \eqref{Lipschitz condition}, in the SAA literature. To see this, one may consider a simple stochastic linear program of  $\min\{\E[\alpha^\top\xbf]:\,\xbf\in[-1,1]^d\}$, where $\alpha\in\R^d$ is some Gaussian random vector. While many applications can be subsumed by simple variations of this SP problem, it hardly satisfies \eqref{lipschitz global} for   a finite $M$. One may easily extend this linear program to identify typical SP problems where \eqref{lipschitz global} fails to apply. Indeed, a closer examination can reveal that \eqref{lipschitz global} is an arguably restrictive special case of \eqref{Lipschitz condition} when $M(\xi)$ in the latter is assumed with a bounded support set.  In contrast, our results are based on more flexible conditions common to the SP literature and subsume the said stochastic linear programming problem. Moreover, our results do not assume a bounded support set   for $M(\xi)$.%  in the most comparable (and actually more adversarial) settings  of our results (i.e., when $\mathcal L=0$), to obtain \eqref{second error bound} and \eqref{second error bound probability} only requires   
 % $F(\cdot)=\E[f(\cdot,\xi)]$ --- 
 %the population-level objective function --- to be Lipschitz continuous. This can sometimes be a non-trivially weaker condition relative to  those imposed by the aforementioned literature. 

Some recent works on SAA for high-dimensional SP study the implications of special problem structures in reducing the growth rate of the metric entropy terms  w.r.t. their dependence on $d$. Along this direction, \cite{liu2022high,liu2019sample} and \cite{lee2023regularized} study the implications of sparsity and low-rankness. \cite{lam2022general} consider the influence of low Vapnik-Chervonenkis (VC) dimensions. \cite{bugg2021logarithmic} investigate the dimension-independent budget constraints. As is also discussed by \cite{oliveira2023sample}, when the feasible region is (representable by) a simplex, the generic chaining-based metric entropy exhibits a logarithmic growth rate with $d$. Nonetheless, to our knowledge, those results may not apply beyond the corresponding special   structural assumptions.

\cite{birge2023uses} shows that careful designs with sub-sample estimates can promote sample efficiency particularly in terms of the dependence on the number of random parameters. Yet, the current theories therein are mostly focused on some more special (but still widely applicable) SP problems. \textcolor{black}{\label{intro rev 1}\cite{guigues2017non} show that SAA's optimal cost can lead to metric entropy-free confidence bounds for the SP's optimal value. Yet, their results do not ensure the quality of an SAA  solution in minimizing the expected cost in \eqref{Eq: SP problem statement}. Metric entropy-free bounds have also been made available by \cite{rigollet2025sample} for optimal transport problems. Yet, it is currently unknown how these results can be transferred to the  problem settings of our discussion.}

This current paper frequently refers to the existing complexity bounds, e.g., by \cite{shapiro2021lectures,shapiro2003monte,shapiro2005complexity}, and \cite{oliveira2023sample} as benchmarks  in order to explain the claimed advantages of our results. Yet, it is worth noting that those existing works apply to more general settings than this paper. For instance, the SAA theories by \cite{shapiro2021lectures} and \cite{oliveira2023sample} can handle nonconvex problems. The findings by \cite{oliveira2023sample} further admit stochasticity in the feasible region.  Nonetheless, when applied to the settings of our consideration---SAA for   convex SP problems with deterministic constraints---the results by \cite{shapiro2021lectures,shapiro2003monte,shapiro2005complexity}, and \cite{oliveira2023sample} are known to be the best available benchmarks. We would like to also argue that the SP problems considered herein are still flexible enough to cover a very wide spectrum of applications, and our proof arguments, which seem to differ   from most SAA literature, may   be further extended to   nonconvex problems  and scenarios with uncertain constraints.

 \section{Main Results}\label{sec: main result}
 This section presents the formal statements of  our three sets of results. First, Section \ref{sec: SA-comparable} provides the sample complexity bounds that match   those of the canonical SMD under comparable assumptions. Second, Section \ref{sec: Metric entropy-free} shows our   large deviations-type, metric entropy-free bounds under the standard Lipschitz condition as in \eqref{Lipschitz condition}. Lastly, Section \ref{sec: non-Lipschitzian} discusses our findings in non-Lipschitzian settings. \textcolor{black}{For ease of reference, Table~\ref{table: result vs assumption} summarizes the assumptions underlying each key result.}
 
 \begin{table}[H]
 \color{black}
\centering
\scriptsize
\setlength{\tabcolsep}{3pt}
\caption{A summary of different combinations of assumptions imposed for the key results.  ``A\#'' stands for ``Assumption \#''. Each assumption appears either in the ``\emph{Regularity}'' row or the ``\emph{Randomness}'' row, characterizing the type of condition it represents. ``$\checkmark$'' indicates the corresponding assumption (in the column) is required for the result (in the row).  Convexity  everywhere on $\Xi$ (which implies A\ref{GC condition constant all}) is assumed  in Prop.~\ref{thm: large deviation}. \newline \noindent\textbf{Abbreviations of assumptions:}  ``subexp.'' stands for subexponential; ``HT'', heavy-tailed;    Others follow Table \ref{table summary results compare with SMD}. }\label{table: result vs assumption}
\begin{tabular}{@{}l*{9}{c}@{}}
\toprule
& \multicolumn{9}{c}{\textbf{Assumptions}} \\
\cmidrule(lr){2-10}
\textbf{Regularity} &
\makecell[t]{A\ref{L-smoothness}\\($\mathcal L{+}\mathcal M$)} & % A1
& % A2 (randomness)
\makecell[t]{A\ref{SC condition constant all}\\(SC)} &          % A3
\makecell[t]{A\ref{GC condition constant all}\\(Conv.)} &        % A4
\makecell[t]{A\ref{assumption f lipschitz}(a)\\($M(\xi)$-Lip.)} &% A5(a)
\makecell[t]{A\ref{assumption f lipschitz v2}(a)\\($M(\xi)$-Lip.)} & % A6(a)
\makecell[t]{A\ref{assumption f lipschitz v3}(a)\\($M(\xi)$-Lip.)} & % A7(a)
\makecell[t]{A\ref{SC condition constant}(a)\\(loc-SC)}          % A8(a)
\\[10pt]
\textbf{Randomness} &
& \makecell[t]{A\ref{assumption: Variance everywhere}\\(Var$\,\leq\sigma_p^2$)} & % A2
& & \makecell[t]{A\ref{assumption f lipschitz}(b)\\(HT $M(\xi)$)} & % A5(b)
\makecell[t]{A\ref{assumption f lipschitz v2}(b)\\(subexp\,$M(\xi)$)} & % A6(b)
\makecell[t]{A\ref{assumption f lipschitz v3}(b)\\(subGauss\,$M(\xi)$)} & % A7(b)
\makecell[t]{A\ref{SC condition constant}(b)\\($p$-mom.\,at $\xbf^*$)} & % A8(b)
\makecell[t]{A\ref{assumption: Variance suboptimality}\\($p$-mom.\,on\,$\X^{*,\epsilon}$)} % A9
\\
\midrule
% ---------------- rows: results with checkmarks ----------------
Thm.~\ref{thm: suboptimality}  & $\checkmark$ & $\checkmark$ & $\checkmark$ & -- & -- & -- & -- & -- & -- \\
Thm.~\ref{thm: second main theorem convex optimal rate}  & $\checkmark$ & $\checkmark$ & -- & $\checkmark$ & -- & -- & -- & -- & -- \\
Thm.~\ref{Most explicit result here}(a) & -- & -- & -- & $\checkmark$ & -- & $\checkmark$ & -- & -- & -- \\
Thm.~\ref{Most explicit result here}(b) & -- & -- & -- & $\checkmark$ & -- & -- & $\checkmark$ & -- & -- \\
Thm.~\ref{thm: first main theorem}  & -- & -- & -- & -- & -- & -- & -- & $\checkmark$ & -- \\
Thm.~\ref{thm: first main theorem convex}  & -- & -- & -- & $\checkmark$ & -- & -- & -- & -- & $\checkmark$ \\
\addlinespace[2pt]
Prop.~\ref{thm: large deviation}(a) & -- & -- & -- & $\checkmark$ & $\checkmark$ & -- & -- & -- & -- \\
Prop.~\ref{thm: large deviation}(b) & -- & -- & -- & $\checkmark$ & -- & $\checkmark$ & -- & -- & -- \\
Prop.~\ref{thm: large deviation}(c) & -- & -- & -- & $\checkmark$ & -- & -- & $\checkmark$ & -- & -- \\
\bottomrule
\end{tabular}
\end{table}
 
 %Subsection \ref{sec: assumptions} discusses our assumptions, and then Subsection \ref{sec: detailed results strongly convex} provides our theorems, whose proofs are in the appendices.

\subsection{SMD-comparable sample complexity of  SAA}\label{sec: SA-comparable}
 
We   start by introducing our assumptions. First, we formalize the previously mentioned  structure of a composite objective function in the genuine SP problem \eqref{Eq: SP problem statement} as below:
\begin{assumption}\label{L-smoothness}For a given $q\geq 1$, let $\varrho=q/(q-1)$ (with the usual convention that \(\varrho=\infty\) when \(q=1\)). There exist two deterministic    functions, denoted by $F_1:\,\X\rightarrow\R$ and $F_2:\,\X\rightarrow\R$ such that
\begin{align*}
F(\xbf)=F_1(\xbf)+F_2(\xbf),
\end{align*}
where $F_1$ and $F_2$ satisfy  the below:
\begin{itemize}
\item[(a)]
Function $F_1$ is everywhere differentiable. Meanwhile, for   some $\mathcal L\geq 1$,  
\begin{align}
\Vert\nabla F_1(\xbf_1)-\nabla F_1(\xbf_2)\Vert_{\varrho}\leq \mathcal L\cdot \Vert\xbf_1-\xbf_2\Vert_{q},~~\forall (\xbf_1,\,\xbf_2)\in\X^2;\label{smooth inequality}
\end{align}
\item[(b)]  \textcolor{black}{There exists some  $\mathcal M\geq 1$ such that, for any given $\xbf\in\X$ and every $\gbf_{F_2}\in\{\mathbf g_F-\nabla F_1(\xbf):\, \mathbf g_F\in\partial F(\xbf)\}$,}
\begin{align}
\textcolor{black}{\Vert \gbf_{F_2}\Vert_\varrho \leq \mathcal M.}\label{lipschitz inequality}
\end{align}
%Here,  $p=q/(q-1)$ in both \eqref{smooth inequality} and \eqref{lipschitz inequality}.
\end{itemize}
\end{assumption}

We sometimes refer to this condition as ``Assumption \ref{L-smoothness} w.r.t. the $q$-norm''. Here, \eqref{smooth inequality}  means that the first component of the population-level objective function $F_1$  is $\mathcal L$-smooth; that is, it admits an $\mathcal L$-Lipschitz continuous gradient. Meanwhile, under the common assumption  that  $F_2$ is convex, \eqref{lipschitz inequality} essentially imposes that   $F_2$ is $\mathcal M$-Lipschitz continuous. Results that apply to such a composite objective function subsume the special cases of   $F$ being smooth (with $F_2\equiv0$) and   $F$ being Lipschitz (with $F_1\equiv0$). Conditions  similar to, if not more critical than, Assumption \ref{L-smoothness} have   been considered in much SP literature  \cite[such as][]{ghadimi2012optimal,ghadimi2013stochastic,nemirovski2009robust,rakhlin2011making,lan2020first,nesterov2017random}. \textcolor{black}{\label{M clarification}Compared to the uniform Lipschitz condition in \eqref{lipschitz global}---which requires the random cost function $f(\cdot,\xi)$ to be Lipschitz continuous for a $\xi$-invariant Lipschitz constant---Assumption \ref{L-smoothness} is a weaker condition. This is because the latter imposes regularities only on   the expected function $F$ instead.}

{\color{black}
\begin{remark}\label{remark smooth inequality}
A natural variant of \eqref{smooth inequality} is
\begin{equation}\label{q to q paper}
\text{for some }L'\ge 0,\qquad 
\|\nabla F_1(\xbf_1)-\nabla F_1(\xbf_2)\|_{\mathbf q}\;\le\; L'\,\|\xbf_1-\xbf_2\|_{\mathbf q},
\quad \forall\,(\xbf_1,\xbf_2)\in\X^2, 
\end{equation}
which employs the same \(\mathbf q\)-norm (\(\mathbf q\ge 1\)) on both sides, in contrast to \eqref{smooth inequality} where \(\varrho\)- and \(q\)-norms appear on different sides. 
It is straightforward to verify that \eqref{q to q paper}, for any \(\mathbf q\ge 1\), implies the existence of a H\"older-conjugate pair \((\varrho,q)\) with \(\varrho = q/(q-1)\) and $q\in[1,2]$ such that \eqref{smooth inequality} holds with \(\mathcal L \le L'\).
\end{remark}}

 %%%%%%%%%%%%%

% \begin{remark}\label{remark: differentiability} The assumption of  everywhere differentiability of $F_2$  is non-critical and can be dropped with some further analysis. Specifically,  for any choice of $\delta>0$, we may always associate $F(\xbf)$ (when it has a non-differentiable component of $F_2$)
%   with its ``smoothed'' approximation of 
%   $F_{\delta}(\xbf):=\E_{\ubf}[F_{\delta}(\xbf+\delta\ubf)]$, where the expectation in $\E_{\ubf}$
%   is over $\ubf$, a standard Gaussian random vector on $\R^d$.  Then, the desired results regarding the original SP problem with objective function $F$ can be viewed as the limiting case of  the SP problem with objective function $F_\delta$, 
%    as $\delta\rightarrow 0^+$. This allows our results under differentiability of $F_2$
%   to be  extendable to the scenarios where $F_2$
%   is  not necessarily everywhere differentiable. A similar argument is discussed by \cite{guigues2017non}.
% \end{remark}% to justify a related differentiability condition for an analysis of the confidence bounds on SAA's optimal cost value.

Our assumption on the   underlying randomness is the everywhere boundedness of the variance of some stochastic subgradient. We formalize this condition in the following, where we recall that $G$ is a deterministic and universally measurable function such that  $G(\xbf,\xi)\in \partial_\xbf f(\xbf,\xi)$ for all $\xbf\in\X$ and  almost every $\xi\in\Xi$:% as formalized below:
\begin{assumption}\label{assumption: Variance everywhere}
\color{black}For some $p\geq 1$ and   $\sigma_p\geq 1$, the expectation   $\E[G(\xbf,\xi)]$ is well-defined and finite-valued for every $\xbf\in\X$, and
\begin{align*}\E\Big[\Vert G(\xbf,\xi)-\E[G(\xbf,\xi)]\Vert_p^2\Big]\leq \sigma_p^2,~~\forall \xbf\in\X.
\end{align*}
\end{assumption}

Hereafter, this condition is sometimes called ``Assumption \ref{assumption: Variance everywhere} w.r.t. the $p$-norm'', which is common in the SP literature, especially in the discussions of  SMD \citep[e.g., by][]{ghadimi2013stochastic}. The stipulation of $\sigma_p\geq 1$ is non-critical; it is   only for the simplification of notations in our results. %Note that  Assumption \ref{L-smoothness} automatically implies the well-defined-ness of $\E\left[G(\xbf,\xi)\right]$. 

We formalize our assumptions of strong convexity and  convexity in Assumptions \ref{SC condition constant all} and \ref{GC condition constant all}, respectively, as below:

\begin{assumption}\label{SC condition constant all}
\textcolor{black}{There exists some scalar $\mu>0$ such that $f(\,\cdot\,,\xi)-\frac{\mu}{2}\Vert\cdot\Vert_q^2$ is convex on $B$ for almost every $\xi\in\Xi$.}
\end{assumption}
This assumption essentially imposes that   $f(\,\cdot\,,\xi)$ is $\mu$-strongly convex on $\X$ w.r.t. the $q$-norm  for almost every $\xi\in\Xi$. As an immediate result of this assumption, the following inequality holds for every pair of solutions $(\xbf_1,\,\xbf_2)\in\X^2$ and  almost every $\xi\in\Xi$: 
\begin{align*}
f(\xbf_1,\xi)-f(\xbf_2,\xi)\geq \langle \gbf,\,\xbf_1-\xbf_2\rangle
+\frac{\mu}{2}\cdot \Vert\xbf_1-\xbf_2\Vert^2_q,%-\kappa(\xi),
\end{align*}
for all $\gbf\in\partial_\xbf f(\xbf_2,\xi)$ and  some given $\mu>0$ and $q\geq 1$.%. Furthermore,   $\E[\kappa(\xi)]=0$.
%For some of our results, we also impose a smoothness condition as below:

\begin{remark}\label{remark: SC condition}
We refer to the above as  ``Assumption  \ref{SC condition constant all}  w.r.t. the $q$-norm'' or ``$\mu$-strong convexity'', which is common in the SAA literature \citep[e.g., in][]{milz2023sample,shalev2010learnability}. Some SP literature \citep[e.g., by][]{ghadimi2012optimal} assumes a relatively more flexible version of strong convexity than Assumption \ref{SC condition constant all} as in the following: 
\begin{align}
\textcolor{black}{F(\xbf_1)-F(\xbf_2)\geq \langle \mathbf g_F,\,\xbf_1-\xbf_2\rangle
+\frac{\mu}{2}\Vert\xbf_1-\xbf_2\Vert^2_q,~~\text{for all~}\mathbf g_F\in\partial F(\xbf_2)~\text{and}~\text{every}~ (\xbf_1,\,\xbf_2)\in\X^2.}\label{assumption: strong convex in literature}
\end{align}
This condition is considered mostly in the discussions of   SMD. We argue that the seemingly higher stringency in Assumption \ref{SC condition constant all} relative to \eqref{assumption: strong convex in literature} does not make the SP problem much easier. Indeed,   lower complexity bounds for  SMD \citep[such as by][]{rakhlin2011making,agarwal2009information} are derived based on the identification of adversarial problems that satisfy  Assumption  \ref{SC condition constant all}. Thus,  when solving an SP problem that satisfies Assumption  \ref{SC condition constant all} instead of  \eqref{assumption: strong convex in literature}, the typical SMD schemes  cannot achieve faster sample complexity rates in general.% when applied to those adversarial problems.
\end{remark}% for the same $q$ for (almost) every $\xi\in\Xi$. %Meanwhile, the set of problems satisfy these two conditions for a flexible set of important SP problems. 

Some results in this section (as well as in several latter parts of this paper) consider the  condition of  (non-strong) convexity below:

\begin{assumption}\label{GC condition constant all}
For  almost every $\xi\in\Xi$, function $f(\,\cdot\,,\xi)$ is convex everywhere on $B$.
\end{assumption}
As an immediate implication of this assumption,  the following inequality holds for every $(\xbf_1,\,\xbf_2)\in\X^2$, all $\gbf\in\partial_\xbf f(\xbf_2,\xi)$, and   almost every $\xi\in\Xi$: 
\begin{align}
f(\xbf_1,\xi)-f(\xbf_2,\xi)\geq \langle \gbf,\,\xbf_1-\xbf_2\rangle.\nonumber
\end{align}
%for some $\mu>0$\textcolor{black}{where is $\mu$}.%, where   $\E[\kappa(\xi)]=0$.

We would like to compare the above with a counterpart assumption that the population-level objective $F(\cdot)$ is convex, which is, again, a common condition in the literature on  SMD \citep[e.g., by][]{nemirovski2009robust,ghadimi2013stochastic}.  Relative to this counterpart condition,  the incremental stringency in Assumption \ref{GC condition constant all} does not make the SP problems much easier; this is because, again, the adversarial problem instances  used to prove lower performance limits  for   SMD in solving the convex SP problems  \citep[such as those constructed by][]{agarwal2009information} often satisfy Assumption \ref{GC condition constant all}. 
From this analysis, one can see that switching from the assumption of $F$ being convex to Assumption \ref{GC condition constant all} does not allow   SMD to achieve a better sample efficiency in general.

We are now ready to formalize the promised  sample complexity bounds in both   strongly convex and   convex cases below.

\begin{theorem}\label{thm: suboptimality} 
Suppose that  Assumptions   \ref{L-smoothness} and \ref{SC condition constant all}  hold  w.r.t. the $q$-norm for a given $q\geq 1$, and that Assumption \ref{assumption: Variance everywhere} holds w.r.t. the $p$-norm for some $p:\,1\leq p\leq \varrho$ where $\varrho$ is the dual exponent of $q$
(with the convention $\varrho=\infty$ when $q=1$).  Then any $(\delta,q)$-approximate solution $\xhb$ to   SAA   \eqref{Eq: SAA} with $\delta\leq 1/N$ satisfies the below: There exists some universal constant $C_1>0$ such that, for any given $\epsilon>0$,
\begin{align}
 \E\left[F(\xhb)-F(\xbf^*)\right]\leq \epsilon,
~~~\text{if }N\geq C_1\cdot\max\left\{\frac{\mathcal L}{\mu},\,\frac{\color{black}\sigma_p^2+\mathcal M^2}{ \mu\epsilon}\right\};\label{thm 2 subopt expected}
\end{align}
and, meanwhile,  for any given $\epsilon>0$ and $\beta\in(0,1)$,
\begin{align}
\Prob\Big[F(\xhb)-F(\xbf^*)\leq \epsilon \Big]\geq1-\beta,
~~\text{if }N\geq C_1\cdot \max\left\{\frac{\mathcal L}{\mu},\,\frac{\color{black}\sigma_p^2+\mathcal M^2}{ \mu\epsilon\beta}\right\}.\label{thm 2 subopt}
\end{align}
\end{theorem}
\begin{proof}{\color{black}\label{proof comment R1}An important component of this proof is to establish that, if one data point is changed to a different i.i.d.\ copy of $\xi$ in SAA, the output solution does not change much, on average, in terms of  the squared distance w.r.t. the $q$-norm.  This   is the manifestation of the innate ``average replace-one (average-RO) stability'' of SAA when it is applied to solving a strongly convex SP problem.  The concept of average-RO stability is introduced by \cite{shalev2010learnability}. To our knowledge, our proof may  be the first to use the average-RO stability to analyze the  non-asymptotic sample complexity of SAA beyond the uniform Lipschitz condition.} (See more discussions on average-RO stability in Remark \ref{RO stability}.) More specifically, we  prove \eqref{thm 2 subopt expected} through two steps below. %he first two steps together prove   Part (a) of the theorem. Then, Step 3 uses results from Part (a) to show Part (b). %Part (c) is finally proved in Step 4. Again, all these proofs are done under Assumption \ref{SC condition constant all slight change}, which is slightly more general than Assumption \ref{SC condition constant all}.

% --- Preamble ---
% \usepackage{amsmath,amssymb}
% \usepackage{pgfplots}
% \pgfplotsset{compat=1.18}
% \usetikzlibrary{arrows.meta}

{\bf Step 1.} Observe that Eq.\,\eqref{suboptimality in computing SAA} (as an implication of the definition of $\xhb$ as in \eqref{define solution}) implies that
\begin{align}
\E\left[F(\xhb)-F(\xbf^*)\right]=&\,\E\left[F(\xhb)-F_N(\xbf^*)\right] \leq  \E\left[F(\xhb)-F_N(\xhb)+\delta\Vert \xhb-\xbf^*\Vert_q\right]\nonumber
\\\leq&\, \frac{\delta^2}{\mu}+\frac{\mu}{4}\E\left[\Vert \xhb-\xbf^*\Vert_q^2\right]+\E\left[F(\xhb)-F_N(\xhb)\right].\label{test SP new result 1} 
\end{align} 
 Invoking Assumption \ref{SC condition constant all}, we have that $F$ is also $\mu$-strongly convex w.r.t. the $q$-norm. Consequently,
\begin{align}
&\E\left[\frac{\mu}{2}\Vert \xhb-\xbf^*\Vert_q^2\right] \leq   \frac{\delta^2}{\mu}+\frac{\mu}{4}\E\left[\Vert \xhb-\xbf^*\Vert_q^2\right]+\E\left[F(\xhb)-F_N(\xhb)\right]\nonumber
  \Longrightarrow   \E\left[\frac{\mu}{4}\Vert \xhb-\xbf^*\Vert_q^2\right]\leq \frac{\delta^2}{\mu}+\E\left[F(\xhb)-F_N(\xhb)\right]. \nonumber
\end{align}
This combined with \eqref{test SP new result 1} implies that
\begin{align}
\E\left[F(\xhb)-F(\xbf^*)\right]\leq \frac{2\delta^2}{\mu}+2\E\left[F(\xhb)-F_N(\xhb)\right].\label{test SP new result 1 updated}
\end{align}
Therefore, it suffices to establish an upper bound on $\E\left[F(\xhb)-F_N(\xhb)\right]$, which is the focus of Step 2 in this proof.

{\bf Step 2.} Below, we let $H(\xbf):=\E[G(\xbf,\xi)]$ for all $\xbf\in\X$. By Assumption \ref{SC condition constant all}, for any given $\xbf,\ybf\in\X$, it holds that  $f(\xbf,\xi)-f(\ybf,\xi)\geq \langle G(\ybf,\xi),\xbf-\ybf\rangle$ for almost every $\xi\in\Xi$. Thus, under Assumption \ref{assumption: Variance everywhere}, $F(\xbf)-F(\ybf)=\E[f(\xbf,\xi)]-\E[f(\ybf,\xi)]\geq \E[\langle G(\ybf,\xi),\xbf-\ybf\rangle]=\langle \E[G(\ybf,\xi)],\xbf-\ybf\rangle=\langle H(\ybf),\xbf-\ybf\rangle$ for any $\xbf,\ybf\in\X$. Consequently, $H(\xbf)\in \partial F(\xbf)$.

With the observation from Step 1, we construct a sequence of alternative formulations of SAA \eqref{Eq: SAA} with $F_{N}^{(j)}(\xbf):=\frac{1}{N}\left(f( \xbf,\xi_{j}')+\sum_{\iota\neq j} f(\xbf,\xi_{\iota})\right)$, where $\xi_j'$ is an i.i.d. copy of $\xi$, for all $j=1,...,N$. Denote that $\boldsymbol\xi^{(j)}_{1,N}=(\xi_1, ...,\xi_{j-1}, \xi_j', \xi_{j+1}, ..., \xi_N)$. Correspondingly, we let $\xhbj:=\widetilde{\xbf}(\boldsymbol\xi^{(j)}_{1,N})$ following the notation  in \eqref{define solution}. 
Below, we establish an overestimate on $N^{-1}\sum_{j=1}^N\E\left[\Vert\xhbj-\xhb\Vert_q^2\right]$.  This overestimate  is to play a key role in bounding  $\E\left[F(\xhb)-F_N(\xhb)\right].$
To that end, we first observe that 
\begin{align}
 F_N(\xhbj)-F_N(\xhb) 
=&\frac{f(\xhbj,\xi_{j})-f(\xhb,\xi_j)}{N}{+}\sum_{\iota\neq j}\frac{f(\xhbj,\xi_{\iota})-f(\xhb,\xi_\iota)}{N}\label{first row}
\\=&\frac{f(\xhbj,\xi_{j})-f(\xhb,\xi_j)}{N} -\frac{f(\xhbj,\xi_{j}')-f(\xhb,\xi_j')}{N}+F_N^{(j)}(\xhbj)-F_N^{(j)}(\xhb)\label{second row}
\\\leq &\frac{f(\xhbj,\xi_{j})-f(\xhb,\xi_j)}{N}-\frac{f(\xhbj,\xi_{j}')-f(\xhb,\xi_j')}{N}+\delta\cdot \Vert \xhbj-\xhb \Vert_{q}\label{third row}
\end{align}
Here, \eqref{first row} and \eqref{second row} are  by the definitions of $F_N$ and $F_N^{(j)}$, and \eqref{third row} is due to the construction of $\xhbj:=\widetilde{\xbf}(\boldsymbol\xi^{(j)}_{1,N})$ (cf. \eqref{suboptimality in computing SAA}).  
 By Assumption \ref{SC condition constant all} and   $G(\xbf,\xi)\in\partial f(\xbf,\xi)$, we have  $f(\xhbj,\xi_{j})-f(\xhb,\xi_j)\leq \langle G(\xhbj,\xi_{j}),\,\xhbj-\xhb\rangle$ as well as $f(\xhb,\xi_j')-f(\xhbj,\xi_{j}')\leq \langle G(\xhb,\xi_{j}'),\,\xhb-\xhbj\rangle$, almost surely.   Combining this with \eqref{third row} leads to the below:
%By convexity, we have
% \begin{align}
% &F_N(\xhbj)-F_N(\xhb)
% \\\leq & \frac{1}{N}\cdot \left\langle \nabla f(\xhbj,\xi_j),\, \xhbj-\xhb\right\rangle
%  +\frac{1}{N}\cdot \left\langle  \nabla f(\xhb,\xi_j'),\, \xhbj-\xhb\right\rangle +\kappa(\xi_j)+\kappa(\xi_j'),~~a.s.\nonumber
% \\  = & \frac{1}{N}\cdot \left\langle \nabla f(\xhbj,\xi_j)-\nabla F(\xhbj),\, \xhbj-\xhb\right\rangle
%  +\frac{1}{N}\cdot \left\langle  \nabla f(\xhb,\xi_j')-\nabla F(\xhb),\, \xhbj-\xhb\right\rangle\nonumber
%  \\&+\frac{1}{N}\cdot \left\langle  \nabla F(\xhbj),\, \xhbj-\xhb\right\rangle+\frac{1}{N}\cdot \left\langle  \nabla F(\xhb),\, \xhbj-\xhb\right\rangle+\kappa(\xi_j)+\kappa(\xi_j')
% %   \\  \leq & \frac{1}{2N}\cdot \left\Vert\nabla f(\xhbj,\xi_j)-\nabla F(\xhbj)\right\Vert^2_*+\frac{1}{2N}\cdot \left\Vert\nabla f(\xhbj,\xi_j)-\nabla F(\xhbj)\right\Vert^2_*
% %\\ +&
% %  \frac{2}{N}\cdot (F(\xhbj)-F(\xhb))+\frac{3L+2}{2N}\Vert\xhbj-\xhb\Vert^2
% \end{align}
\begin{align}
F_N(\xhbj)-F_N(\xhb)\nonumber
\leq & \frac{1}{N}\cdot \left\langle G(\xhbj,\xi_j)-G(\xhb,\xi_j'),\, \xhbj-\xhb\right\rangle
  +\delta\cdot \Vert \xhbj-\xhb \Vert_{q},~~a.s.\nonumber
\\  = & \frac{1}{N}\cdot \left\langle G(\xhbj,\xi_j)-H(\xhbj),\, \xhbj-\xhb\right\rangle
 +\frac{1}{N}\cdot \left\langle  G(\xhb,\xi_j')-H(\xhb),\, \xhb-\xhbj\right\rangle\nonumber
 \\&+\frac{1}{N}\cdot \left\langle  H(\xhbj)-H(\xhb),\, \xhbj-\xhb\right\rangle+\delta\cdot \Vert \xhbj-\xhb \Vert_{q}.\label{TBC here}
\end{align}
% {
% \color{black}
% \begin{align}
% &-\frac{\mu}{2}\Vert\xhb-\xhbj \Vert_q^2\geq F_N(\xhb)-F_N(\xhbj)=\frac{1}{N} [f(\xhb,\xi)-f(\xhbj,\xi)]
% \\\geq &N^{-1}\langle\nabla f(\xhb,\xi),\,\xhb-\xhbj \rangle-\frac{1}{2 N\mu}\Vert \nabla f(\xhb,\xi_j)-\nabla f(\xhbj,\xi_j) \Vert_q^2
% \\\geq &N^{-1}\langle\nabla f(\xhbj,\xi),\,\xhb-\xhbj \rangle-\frac{\mu}{2 N}\Vert \xhbj-\xhb \Vert_q^2
% \\\geq &N^{-1}\langle\nabla f(\xhbj,\xi)-\nabla F(\xhbj),\,\xhb-\xhbj \rangle+N^{-1}\langle \nabla F(\xhbj),\,\xhb-\xhbj\rangle+\frac{\mu}{2 N}\Vert \xhbj-\xhb \Vert_q^2
% \end{align}
% }
% \textcolor{black}{Need to use smoothness to control this term $\frac{1}{N}\cdot \left\langle  \nabla F(\xhb),\, \xhb-\xhbj\right\rangle$, and strongly convex for this term $\frac{1}{N}\cdot \left\langle  \nabla F(\xhbj),\, \xhbj-\xhb\right\rangle$
% } \textcolor{black}{Both terms needs to be controlled via smoothness condition...}
Since   Assumption \ref{L-smoothness} leads to  
\begin{align}
&\,\left\langle  H(\xhbj)-H(\xhb),\, \xhbj-\xhb\right\rangle\nonumber
\\=&\,\left\langle  \nabla F_1(\xhbj)-\nabla F_1(\xhb),\, \xhbj-\xhb\right\rangle+\left\langle  [H(\xhbj)-\nabla F_1(\xhbj)]-[H(\xhb)-\nabla F_1(\xhb)],\, \xhbj-\xhb\right\rangle\nonumber
\\\leq&\, \mathcal L\Vert \xhb-\xhbj\Vert_q^2+2\mathcal M \Vert \xhb-\xhbj\Vert_q,\label{useful inequality to use soon}
%\\\leq& \Vert \nabla F_1(\xhbj)-\nabla F_1(\xhb)\Vert_p\cdot \Vert \xhb-\xhbj\Vert_q+2\mathcal M \Vert \xhb-\xhbj\Vert_q,
\end{align}
we may continue from the above to obtain, through H\"older's and Young's inequalities, for all $\alpha>0$, %\textcolor{black}{$\forall\alpha>0$}:
\begin{align}
\color{black}
F_N(\xhbj)-F_N(\xhb)
 \leq &\, \color{black}\frac{1}{2\alpha \mu N^2}\cdot\left( \Big\Vert G(\xhbj,\xi_j)-H(\xhbj)\Big\Vert^2_p+  \Big\Vert G(\xhb,\xi_j')-H(\xhb)\Big\Vert^2_p \right) \nonumber
\\\color{black}  & \color{black}+ \left[\frac{\mathcal L}{N}+\left(\alpha+\frac{1}{8}\right)\mu\right]\cdot \Vert\xhbj-\xhb\Vert_q^2+\frac{16\mathcal M^2}{\mu N^2}+\frac{4\delta^2}{\mu},~~~a.s.\label{vital ahh}
\end{align}
By strong convexity of $F_N$ (as an immediate result of Assumption \ref{SC condition constant all})  and the definition of $\xhb$ in \eqref{define solution}, we have that
\begin{align}
F_N(\xhbj)-F_N(\xhb)\geq & -\delta\cdot\Vert\xhbj-\xhb \Vert_q+\frac{\mu}{2}\cdot \Vert\xhbj-\xhb\Vert^2_q 
\geq  -\frac{\delta^2}{\mu}+\frac{\mu}{4}\Vert\xhbj-\xhb\Vert^2_q,~~~a.s.\label{strong convexity results FN otherwise}
\end{align}
where the last inequality uses $\delta \Vert\xhbj-\xhb \Vert_q\leq \mu^{-1}\delta^2+0.25 \mu\Vert\xhbj-\xhb \Vert_q^2$.
Combining \eqref{vital ahh} and \eqref{strong convexity results FN otherwise}, we immediately obtain the below after some re-organization and simplification:
$$\left[\left(\frac{1}{8}-\alpha\right)  \mu-\frac{\mathcal L}{N}\right]\cdot \Vert\xhbj-\xhb\Vert_q^2\leq  
\frac{1}{2N^2\mu\alpha} \left\Vert G(\xhbj,\xi_j)-H(\xhbj)\right\Vert^2_p
+\frac{1}{2N^2\mu\alpha} \left\Vert G(\xhb,\xi_j')-H(\xhb)\right\Vert^2_p+\frac{16\mathcal M^2}{\mu N^2}+\frac{5\delta^2}{\mu},~~ a.s.$$
Note that $\xhbj$ and $\xi_j$ are independent, so are $\xhb$ and $\xi'_j$. We therefore have  $\E[\Vert G(\xhbj,\xi_j)-H(\xhbj)\Vert^2_p]=\E\big[\E[\Vert G(\xhbj,\xi_j)-H(\xhbj)\Vert^2_p\,\big\vert\, \xhbj]\big]\leq \sigma^2_p$ and $\E [\Vert G(\xhb,\xi_j')-H(\xhb)\Vert^2_p]=\E \big[\E\left[\left.\Vert G(\xhb,\xi_j')-H(\xhb)\Vert^2_p\right\vert \xhb\right]\big]\leq \sigma^2_p$ under Assumption \ref{assumption: Variance everywhere}. Further because     we may let  $\alpha=1/16$ and it is assumed that $N\geq \frac{C'\mathcal L}{\mu}$ (where  we may as well let $C'\geq 32$),  we then have $\left[\left(\frac{1}{8}-\alpha\right)  \mu-\frac{\mathcal L}{N}\right]^{-1}\leq \frac{32}{\mu}$ and, therefore,
\begin{align}
& \color{black}\frac{1}{N}\sum_{j=1}^N\E\left[\Vert\xhbj-\xhb\Vert_q^2\right]\leq  \frac{1}{N}\sum_{j=1}^N \left[\left(\frac{1}{8}-\alpha\right)  \mu-\frac{\mathcal L}{N}\right]^{-1} \left( 
\frac{\sigma_p^2}{N^2\mu\alpha} +\frac{16\mathcal M^2}{\mu N^2}+\frac{5\delta^2}{\mu}\right)\leq   
512\cdot \frac{\sigma_p^2+\mathcal M^2}{N^2\mu^2}+\frac{160\delta^2}{\mu^2}.\label{important bound}
%\\\Longrightarrow &N^{-1}\sum_{j=1}^N\E\left[\Vert\xhbj-\xhb\Vert_q^2\right]\leq    512\cdot \frac{\sigma_p^2+\mathcal M^2}{N^2\mu^2}+\frac{160\delta^2}{\mu^2}.\label{important bound}
\end{align}

%Recall the notation that $\boldsymbol\xi_{1,N}=(\xi_j:j=1,...,N)$ and $\boldsymbol\xi'_{1,N}=(\xi_j':j=1,...,N)$.  
Because $f(\xhb,\xi_j')$ and $f(\xhbj,\xi_j)$ are identically distributed --- so are $f(\xhb,\xi_j)$ and $f(\xhbj,\xi_j')$ ---  we then obtain that $\E[f(\xhb,\xi'_j)]=\E[f(\xhbj,\xi_j)]$ and that $\E[f(\xhb,\xi_j)]=\E[f(\xhbj,\xi_j')]$. Therefore,
\begin{align}
&\E[F(\xhb)-F_N(\xhb)]\nonumber
\\=&\E\left[\frac{1}{N}\sum_{j=1}^N \left[F(\xhb)-f(\xhb,\xi_j)\right]
\right]=\E\left[\frac{1}{N}\sum_{j=1}^N \left[f(\xhb,\xi_j')-f(\xhb,\xi_j)\right]
\right]\nonumber
\\=&\frac{1}{2N}\sum_{j=1}^N \E \left[f(\xhb,\xi_j')-f(\xhb^{(j)},\xi_j')\right]+\frac{1}{2N}\sum_{j=1}^N \E \left[f(\xhb^{(j)},\xi_j)-f(\xhb,\xi_j)\right]\nonumber
\\\leq &\textcolor{black}{\frac{1}{2N}\sum_{j=1}^N \E \left[\left\langle G(\xhb,\xi_j')-G(\xhb^{(j)},\xi_j),\,\xhb-\xhb^{(j)}\right\rangle\right]}%+\frac{1}{2N}\sum_{j=1}^N \E \left[\langle G(\xhb^{(j)},\xi_j),\,\xhb^{(j)}-\xhb\rangle\right]
\label{to derive here now}
\\=  &\textcolor{black}{\frac{1}{2N}\sum_{j=1}^N \E \left[\left\langle G(\xhb,\xi_j')- H(\xhb),\,\xhb-\xhb^{(j)}\right\rangle\right]+\frac{1}{2N}\sum_{j=1}^N \E \left[\left\langle G(\xhb^{(j)},\xi_j)-H(\xhbj),\,\xhb^{(j)}-\xhb\right\rangle\right]}\nonumber
\\& 
 \textcolor{black}{+\frac{1}{2N}\sum_{j=1}^N \E \left[\left\langle H(\xhb)-H(\xhbj),\, \xhb-\xhbj\right\rangle\right]}\nonumber
\\{\leq}&\frac{1}{2N}\sum_{j=1}^N \E \left[\frac{8}{N\mu}\left\Vert G(\xhb,\xi_j')- H(\xhb)\right\Vert_p^2+\frac{8}{N\mu}\left\Vert H(\xhbj)-G(\xhbj,\xi_j)\right\Vert_p^2 
\right.\nonumber
\\&\left.+\left(\mathcal L+\frac{N\mu}{16}\right)\Vert\xhb-\xhb^{(j)}\Vert_q^2+ {2\mathcal M \Vert \xhb-\xhbj\Vert_q}\right] \label{to derive here now 23}
\\  \leq &\frac{1}{2N}\sum_{j=1}^N \E \left[\frac{16}{N\mu}\sigma_p^2+\left(\mathcal L+\frac{17}{16}N\mu\right)\Vert\xhb-\xhb^{(j)}\Vert_q^2+  \frac{\mathcal M^2}{N\mu}\right]\label{reducing variance}
\leq   C''\cdot \Bigl(\frac{\color{black}\sigma_p^2+\mathcal M^2}{ N\mu}+  \frac{N\delta^2}{\mu}\Bigr), 
\end{align}
for some universal constant $C''>0$.
Here, \eqref{to derive here now}  is based on the  convexity of  $f(\cdot,\xi)$, \textcolor{black}{\eqref{to derive here now 23} is due to Assumption \ref{L-smoothness} (which leads to \eqref{useful inequality to use soon}) as well as Young's inequality},  the first inequality in  \eqref{reducing variance} is the result of invoking both Assumption \ref{assumption: Variance everywhere} and Young's inequality,  and the last inequality in \eqref{reducing variance} %\eqref{final sp} 
is obtained by invoking  \eqref{important bound} and the assumption that $N\geq \frac{C'\mathcal L}{\mu}$.

Recall our stipulation of $N\geq \frac{C'\mathcal L}{\mu}$. Combining this with  \eqref{reducing variance}  and  \eqref{test SP new result 1 updated} leads to $\E[F(\xhb)-F(\xbf^*)]\leq 2C''\cdot \frac{\color{black}\sigma_p^2+\mathcal M^2}{ N\mu}+\frac{(C''\cdot N+2)\delta^2}{\mu}$, which implies the desired result  in  \eqref{thm 2 subopt expected} after some simple re-organization in view of $\delta\leq \frac{1}{N}$.  Then \eqref{thm 2 subopt} is   an immediate result by Markov's inequality.  
\end{proof}
\smallskip

%\begin{remark}\label{use of RO}
%
%\end{remark}

\smallskip

  \begin{remark}\label{remark: Comparison with SMD strongly convex}
 The theorem above confirms the promised sample complexity in \eqref{summary first result}  for SAA  \eqref{Eq: SAA} when it is applied to the strongly convex SP problems. In comparison with the existing results  on  SMD, e.g., as discussed by \cite{lan2020first} and \cite{juditsky2011solving},  the rates in \eqref{thm 2 subopt expected} and \eqref{thm 2 subopt} are identical to the best known rates for the canonical SMD algorithms in terms of the sample size requirement $N$ to achieve the same   solution accuracy.  \textcolor{black}{(See Table \ref{table summary results compare with SMD} for a summary of comparisons with SMD).} Nonetheless,  it is worth noting the presence of  an accelerated   SMD variation, called AC-SA \citep{lan2012optimal,ghadimi2012optimal,lan2020first},  which provably achieves a better rate with $\mathcal L$ for some careful choice of hyperparameters. 
 \end{remark}
 
 {\color{black}
 \begin{remark}\label{rk: variance}
The variance term $\sigma_p^2$     depends on dimensionality $d$ in general. This dependence can be further explicated under some additional assumptions. For instance, suppose that, for some $p\geq 2$ and $\phi_p\geq 0$, the component-wise $p$th  central moment of $G(\xbf,\xi)$ is bounded by $\phi_p^p$ everywhere; namely, for     $\phi_p\geq 0$,  it holds that $\left\Vert G_i(\xbf,\xi)-\E[G_i(\xbf,\xi)]\right\Vert_{L^p}\leq \phi_p$ for all $\xbf\in\X$ and every $i=1,...,d$, where $G_i(\xbf,\xi)$ is  the $i$th entry of $G(\xbf,\xi)$. Because, for $p\geq 2$, the function $(\cdot)^{2/p}$ is concave in `$\cdot$', one may easily see  the following:
\begin{align}
&
\E\bigl[\|G(\xbf,\xi)-\E G(\xbf,\xi)\|_p^2\bigr]\leq (\E\bigl[\|G(\xbf,\xi)-\E G(\xbf,\xi)\|_p^{p}\bigr])^{2/p}=\left(\sum_{i=1}^d\E\bigl[\|G_i(\xbf,\xi)-\E G_i(\xbf,\xi)\|_p^{p}\bigr]\right)^{2/p}\le d^{2/p}\phi_p^2.\label{Eq: bounding dependence on d}
\end{align}
By this observation, in this case, we have $\sigma_p^2\leq d^{2/p}\cdot \phi_{p}^2$ in our bounds, whose dependence on dimensionality reduces when $p$ increases. Particularly, when it is admissible to let $p\geq c\ln d$ for some constant $c>0$, the quantity $\sigma_p^2$ becomes dimension-free.
\end{remark}
}

{\color{black}
\begin{remark}\label{absence of problem parameters}
The SAA problem in \eqref{Eq: SAA} does not involve problem-dependent quantities such as $\sigma_p$, $\mathcal M$, or $\mathcal L$. Consequently, there is no need to estimate these parameters. In contrast, SMD methods \citep[e.g., by][]{nemirovski2009robust,ghadimi2013stochastic} often require knowledge of $\sigma_p$, $\mathcal M$, or $\mathcal L$ to properly set hyperparameters such as   step size. Thus, the absence of these quantities in the SAA problem can sometimes lead to  convenience in implementation.
%The formulation of SAA for solving a $\mu$-strongly convex SP problem does not require estimating the value of $\mu$, nor is it necessary to assess $\sigma_p$, $\mathcal M$, or $\mathcal L$. This feature may sometimes lead to convenience in solving an SP problem.
\end{remark}
}
% \item[(b).] Let $q\in[1,\,2]$. If Assumption \ref{assumption: moment bound everywhere dimension-agnostic} holds w.r.t. the $p$-norm for some $p<\infty$ such that $2\leq p\leq\frac{q}{q-1}$, any optimal solution $\xhb$ to SAA in \eqref{Eq: SAA}  satisfies the following two inequalities for any $\epsilon>0$ and $\beta\in(0,1)$: 
% \begin{align}
% \begin{split}
% &\E[F(\xhb)-F(\xbf^*)]\leq \epsilon,~~\text{if }N\geq \max\left\{\frac{8\mathcal L}{\mu},\,\frac{\color{black}30\psi_p^2+41\mathcal M^2}{\mu\epsilon}\right\};
% \\
% &\Prob\left[\vphantom{V^{V^{V^V}}_{V_V}}F(\xhb)-F(\xbf^*)\leq \epsilon\right]\geq 1-\beta,~~\text{if } N\geq \max\left\{\frac{8\mathcal L}{\mu},\,\frac{\color{black}30\psi_p^2+41\mathcal M^2}{ \mu\epsilon\beta}\right\}.
% \end{split}\label{smooth strong convex sub result 2}
% \end{align}

% \end{itemize}
% \end{theorem}
% \smallskip

\smallskip

Our next theorem is focused on  convex SP. Before its statement, we first introduce some choice of hyperparameters for the Tikhonov-like penalty in \eqref{Eq: SAA-ell2}  given some $q> 1$  and a user-specified accuracy threshold $\epsilon>0$: %We let 
\begin{align}
q'\in(1,2]:\,q'\leq q;~~~~R^*\geq \max\{1,\,V_{q'}(\xbf^*)\}; ~~~~\text{and}~~~~\lambda_0=\frac{\epsilon}{2R^*}.\label{parameter setting}
\end{align}
\begin{theorem}\label{thm: second main theorem convex optimal rate}
Let  $q> 1$. Suppose that the hyperparameters $q'$, $R^*$, and $\lambda_0$ are specified as in \eqref{parameter setting}. Assume that (i) Assumption \ref{L-smoothness} w.r.t.  the $q$-norm, (ii) Assumption \ref{assumption: Variance everywhere} w.r.t.\ the $p$-norm for some $p:\,1\leq p\leq \frac{q}{q-1}$, and (iii) Assumption \ref{GC condition constant all} hold.   Any $(\delta,q)$-approximate solution $\xhb$ to SAA \eqref{Eq: SAA-ell2}  with $\delta\leq    1/N$  satisfies the following   inequalities:  For any given $\epsilon\in(0,1]$,%$\frac{48\sigma_p^2+28\mathcal M^2}{(q'-1)\lambda_0 N}\left(1+\frac{20}{N (q'-1)}\right) + \frac{20\epsilon}{N (q'-1)}$
\begin{align}
\E\left[F(\xhb)-F(\xbf^*)\right]\leq  \epsilon,
~~~~\text{if }
N\geq \frac{C_2 R^*}{q'-1}\cdot \max\left\{\frac{\mathcal L}{\epsilon},~\frac{\sigma_p^2+\mathcal M^2}{\epsilon^2} \right\}; \label{thm 2 subopt RSAA optimal GC}
\end{align}
and, meanwhile,  for any given $\epsilon\in(0,1]$ and $\beta\in(0,1)$,
\begin{align}
\Prob\Big[F(\xhb)-F(\xbf^*)\leq \epsilon\Big]\geq1-\beta,~~~~\text{if }
N\geq \frac{C_2R^*}{q'-1}\cdot \max\left\{\frac{ \mathcal L}{\epsilon},~\frac{\sigma_p^2+\mathcal M^2}{\beta\epsilon^2}\right\}.\label{thm 2 subopt RSAA optimal GC probability}
\end{align}
Here, $C_2>0$ is some universal constant.
\end{theorem}
\begin{proof} The proof below follows that of Theorem \ref{thm: suboptimality} with some important modifications. First,  SAA \eqref{Eq: SAA-ell2} can be {considered as} SAA \eqref{Eq: SAA} to the following new SP problem: 
$
\underset{\xbf\in\X}{\min}\, F_{\lambda_0}(\xbf):=F(\xbf)+\lambda_0V_{q'}(\xbf).$ Hereafter within this proof, we let $\xbf^*_{\lambda_0}\in\arg\min_{\xbf\in\X}F_{\lambda_0}(\xbf)$, denote that $H(\xbf):=\E[G(\xbf,\xi)]$  
 and $H_{\lambda_0}(\xbf):=H(\xbf)+\lambda_0\nabla V_{q'}(\xbf)$ for all $\xbf\in\X$.
%\textcolor{black}{ where  $F_{\lambda_0}(\xbf):=F(\xbf)+\lambda_0V_{q'}(\xbf)$.} 
 We repeat a similar argument as in \eqref{test SP new result 1}-\eqref{test SP new result 1 updated} below:
 Observe that Eq.\,\eqref{suboptimality in computing SAA} for the case of SAA \eqref{Eq: SAA-ell2} (as an implication of the definition of $\xhb$ as in \eqref{define solution}) leads to
\begin{align}
\E\left[F_{\lambda_0}(\xhb)-F_{\lambda_0}(\xbf^*_{\lambda_0})\right]=&\,\E\left[F_{\lambda_0}(\xhb)-F_{{\lambda_0},N}(\xbf^*_{\lambda_0})\right] \leq  \E\left[F_{\lambda_0}(\xhb)-F_{{\lambda_0},N}(\xhb)+\delta\Vert \xhb-\xbf^*_{\lambda_0}\Vert_q\right]\nonumber
\\\leq&\, \frac{2\delta^2}{\lambda_0\cdot (q'-1)}+\frac{\lambda_0\cdot (q'-1)}{8}\E\left[\Vert \xhb-\xbf^*_{\lambda_0}\Vert_q^2\right]+\E\left[F_{\lambda_0}(\xhb)-F_{{\lambda_0},N}(\xhb)\right].\label{test SP new result 1 v2}
\end{align}
 Invoking Assumption \ref{GC condition constant all} and the fact that $V_{q'}$ is $(q'-1)$-strongly convex w.r.t. the $q'$-norm \citep{ben2001ordered} in the sense of \eqref{SC V}, since $q'\in(1,2]$, we have that $F_{\lambda_0}$ is also $[(q'-1)\cdot \lambda_0]$-strongly convex w.r.t. the $q'$-norm. This combined with \eqref{test SP new result 1 v2} and   $q'\leq q$ implies
\begin{align}
%&\E\left[\frac{(q'-1)\cdot \lambda_0}{2}\Vert \xhb-\xbf^*_{\lambda_0}\Vert_{q'}^2\right] \leq   \frac{2\delta^2}{(q'-1)\cdot \lambda_0}+\frac{(q'-1)\cdot \lambda_0}{8}\E\left[\Vert \xhb-\xbf^*_{\lambda_0}\Vert_{q'}^2\right]+\E\left[F_{\lambda_0}(\xhb)-F_{\lambda_0,N}(\xhb)\right]\nonumber
{\color{black}  \E\left[\Vert \xhb-\xbf^*_{\lambda_0}\Vert_{q'}^2\right]\leq \frac{16\delta^2}{3(q'-1)^2\cdot \lambda_0^2}+\frac{8}{3(q'-1)\cdot\lambda_0}\E\left[F_{\lambda_0}(\xhb)-F_{\lambda_0,N}(\xhb)\right].} \label{to plug in}
  \end{align}
 Recalling the  construction of $F_{\lambda_0}$, the fact that $V_{q'}(\cdot)\geq 0$, and the choice of $\lambda_0$, we then have
 \begin{align}
 \E\left[F (\xhb)-F (\xbf^*)-\frac{\epsilon}{2}\right]\leq &\E\left[F (\xhb)-F (\xbf^*)-\lambda_0V_{q'}(\xbf^*)\right]\leq \E\left[F_{\lambda_0} (\xhb)-F_{\lambda_0} (\xbf^*)\right]\nonumber
 \\\leq& \E\left[F_{\lambda_0} (\xhb)-F_{\lambda_0} (\xbf^*_{\lambda_0})\right]\leq \frac{4}{3}\E\left[F_{\lambda_0}(\xhb)-F_{\lambda_0,N}(\xhb)\right]+\frac{8\delta^2}{3\lambda_0(q'-1)},\label{start of all GC new}
 \end{align}
 where the last inequality in \eqref{start of all GC new} is obtained by combining \eqref{to plug in}  and \eqref{test SP new result 1 v2}. By the observation of \eqref{start of all GC new}, one can see that, in order to   control  {$\E[F(\xhb)-F(\xbf^*)]$} it suffices to derive an upper bound on $\E\left[F_{\lambda_0}(\xhb)-F_{\lambda_0,N}(\xhb)\right]$.  The rest of the proof follows this observation, which further leads to the desired results. 
 
 Let   $\boldsymbol\xi^{(j)}_{1,N}=(\xi_1,...,\xi_{j-1},\xi_j',\xi_{j+1},...,\xi_N)$ with $\xi_j'$ being an i.i.d. copy of $\xi$, define $f_{\lambda_0}(\cdot,\xi):=f(\cdot,\xi)+\lambda_0 V_{q'}(\cdot)$, let $G_{\lambda_0}(\cdot,\xi):=G(\cdot,\xi)+\lambda_0\nabla V_{q'}(\cdot)$, and denote that $\xhbj:=\widetilde{\xbf}(\boldsymbol\xi^{(j)}_{1,N})$ (cf., the definition of $\widetilde{\xbf}$ in \eqref{define solution}).   Under Assumption \ref{GC condition constant all}, we can follow exactly the same argument for \eqref{TBC here} in the proof of Theorem \ref{thm: suboptimality} to obtain 
\begin{multline}
 F_{\lambda_0,N}(\xhbj)-F_{\lambda_0,N}(\xhb) 
 \leq  
\frac{1}{N}\cdot \left\langle G_{\lambda_0}(\xhbj,\xi_j)-H_{\lambda_0}(\xhbj),\, \xhbj-\xhb\right\rangle
 \\+\frac{1}{N}\cdot \left\langle  G_{\lambda_0}(\xhb,\xi_j')-H_{\lambda_0}(\xhb),\, \xhb-\xhbj\right\rangle +\frac{1}{N}\cdot \left\langle  H_{\lambda_0}(\xhbj)-H_{\lambda_0}(\xhb),\, \xhbj-\xhb\right\rangle+\delta\cdot \Vert \xhbj-\xhb \Vert_{q}.\label{to use soon afterwards} 
\end{multline}
Observe that 
\begin{align}
&\left\langle  H_{\lambda_0}(\xhbj)-H_{\lambda_0}(\xhb),\, \xhbj-\xhb\right\rangle\nonumber
%\\=&\left\langle  \nabla F_1(\xhbj)-\nabla F_1(\xhb),\, \xhbj-\xhb\right\rangle+\left\langle  \nabla F_2(\xhbj)-\nabla F_2(\xhb),\, \xhbj-\xhb\right\rangle
\\=&\left\langle  \nabla F_1(\xhbj)-\nabla F_1(\xhb),\, \xhbj-\xhb\right\rangle\nonumber
+\left\langle  [H(\xhbj)-\nabla F_1(\xhbj)]-[H(\xhb)-\nabla F_1(\xhb)],\, \xhbj-\xhb\right\rangle\nonumber
\\&+\left\langle   \lambda_0 \nabla V_{q'}(\xhbj)- \lambda_0 \nabla V_{q'}(\xhb),\, \xhbj-\xhb\right\rangle\nonumber
\\\leq&   \mathcal L\left\Vert \xhbj-\xhb\right\Vert_q^2+2\mathcal M \Vert \xhb-\xhbj\Vert_q +\lambda_0\cdot (\Vert \xhbj -\xbf^0\Vert_{q'}+\Vert  \xhb-\xbf^0\Vert_{q'})\cdot\Vert\xhbj-\xhb\Vert_{q'},%\label{new useful inequality to ensure v1} 
\label{new useful inequality to ensure v1}
\end{align}
where \eqref{new useful inequality to ensure v1} is due to Assumption \ref{L-smoothness} and  a property of $V_{q'}(\cdot)=0.5\Vert\cdot-\xbf^0\Vert_{q'}^2$ (as shown in \eqref{property of Vq} of Appendix \ref{sec: preliminary}); that is, $\Vert \nabla V_{q'}(\cdot ) \Vert_{p'}=\Vert\cdot-\xbf^0\Vert_{q'}$ for $p'=q'/(q'-1)$.
%, and the last inequality is due to $q\leq q'$. 

Note that $\xhbj$ and $\xi_j$ are independent, so are $\xhb$ and $\xi'_j$.
Assumption  \ref{assumption: Variance everywhere} then implies that $\E\Big[\Big\Vert G_{\lambda_0}(\xhbj,\xi_j)-H_{\lambda_0}(\xhbj)\Big\Vert^2_p\Big]=\E\left[\E\left[\left.\left\Vert G_{\lambda_0}(\xhbj,\xi_j)-H_{\lambda_0}(\xhbj)\right\Vert^2_p\right\vert \xhbj\right]\right]\leq\sigma^2_p$ and  $\E\left[ \left\Vert G_{\lambda_0}(\xhb,\xi_j')-H_{\lambda_0}(\xhb)\right\Vert^2_p\right]\leq\sigma^2_p$.  Further noting that $q'\leq q$,  we may then continue from \eqref{to use soon afterwards}  above to obtain, for any $\alpha>0$:  
\begin{align}
&\E[F_{\lambda_0,N}(\xhbj)-F_{\lambda_0,N}(\xhb)]\nonumber
\\\leq & \E\left[\frac{1}{2\alpha (q'-1)\lambda_0 N^2}\cdot \left\Vert G_{\lambda_0}(\xhbj,\xi_j)-H_{\lambda_0}(\xhbj)\right\Vert^2_p+\frac{1}{2\alpha(q'-1)\lambda_0   N^2}\cdot \left\Vert G_{\lambda_0}(\xhb,\xi_j')-H_{\lambda_0}(\xhb)\right\Vert^2_p\right]\nonumber
\\ &+
  \left[\frac{\mathcal L}{N}+\left(\frac{1}{16}+2\alpha\right)  \lambda_0  (q'-1) \right]\E\Vert\xhbj-\xhb\Vert_q^2+\frac{16\mathcal M^2}{\lambda_0\cdot (q'-1)N^2}+\frac{\lambda_0\E\left[(\Vert \xhbj -\xbf^0\Vert_{q'}+\Vert  \xhb-\xbf^0\Vert_{q'})^2\right]}{4\alpha N^2  \cdot (q'-1)} \nonumber
   \\&+\frac{\lambda_0(q'-1)\cdot \E\Vert  \xhb-\xhbj\Vert_{q'}^2}{16}+\frac{4\delta^2}{\lambda_0(q'-1)} \nonumber
  \\\leq & \frac{\sigma_p^2}{\alpha (q'-1)\lambda_0 N^2} +
  \left[\frac{\mathcal L}{N}+\left(\frac{1}{8}+2\alpha\right)\lambda_0\cdot (q'-1) \right]\E\Vert\xhbj-\xhb\Vert_{q'}^2+\frac{16\mathcal M^2}{\lambda_0\cdot (q'-1)N^2}\nonumber
  \\&+\frac{\lambda_0}{\alpha N^2  \cdot (q'-1)}\cdot  \E[\Vert  \xhb-\xbf^0\Vert_{q'}^2]+\frac{4\delta^2}{\lambda_0(q'-1)} 
\label{vital ahh 2 GC}
\end{align}
 where \eqref{vital ahh 2 GC}    invokes the aforementioned implications of Assumption  \ref{assumption: Variance everywhere}, the relationship that $\E[\Vert \xhb-\xbf^0\Vert_{q'}^2]=\E[\Vert \xhbj-\xbf^0\Vert_{q'}^2]$, and the assumption that $1<q'\leq q$.  As mentioned, $V_{q'}(\cdot)$   is $(q'-1)$-strongly convex w.r.t. the $q'$-norm \citep{ben2001ordered} with $q'\in(1,2]$. Thus, by Assumption \ref{GC condition constant all}, $F_{\lambda_0,N}$ entails $[(q'-1)\lambda_0]$-strong convexity  in the sense of Assumption \ref{SC condition constant all} also w.r.t.\, the $q'$-norm.   This strong convexity --- when combined with the definition of $\xhb$ in \eqref{define solution} and the fact that $q'\leq q$ --- leads to $F_{\lambda_0,N}(\xhbj)-F_{\lambda_0,N}(\xhb)\geq -\delta\Vert\xhb-\xhbj \Vert_q+0.5\cdot \lambda_0\cdot (q'-1)\Vert \xhb-\xhbj\Vert_{q'}^2\geq -\delta\Vert\xhb-\xhbj \Vert_{q'}+0.5\cdot \lambda_0\cdot (q'-1)\Vert \xhb-\xhbj\Vert_{q'}^2$. This together with \eqref{vital ahh 2 GC} (for  $\alpha$ selected to be some universal constant),   the   condition that $N\geq \frac{C_2 \mathcal L}{(q'-1)\lambda_0}$, and  Young's inequality leads to the below, for some universal constant $\widehat C_1>0$:
\begin{align}
\E\left[  \Vert\xhbj-\xhb\Vert^2_{q'}\right] \leq \widehat C_1\frac{\sigma_p^2+\mathcal M^2}{(q'-1)^2\lambda^2_0 N^2}+\frac{\widehat C_1}{N^2 (q'-1)^2}\E[\Vert  \xhb-\xbf^0\Vert_{q'}^2]+\frac{\widehat C_1\delta^2}{\lambda_0^2(q'-1)^2}.
\label{proven gap here}%\textcolor{black}{\LArge a.s. to be added to ICML paper} \textcolor{blue}{\Large Modified as well}
\end{align}
We  observe that $f_{\lambda_0}(\xhb,\xi_j')$ and $f_{\lambda_0}(\xhbj,\xi_j)$ are identically distributed, so are the pair of $f_{\lambda_0}(\xhb,\xi_j)$ and $f_{\lambda_0}(\xhbj,\xi_j')$.  
Therefore, 
\begin{align}
&\E[F_{\lambda_0}(\xhb)-F_{\lambda_0,N}(\xhb)] =\E\left[\frac{1}{N}\sum_{j=1}^N \left[F_{\lambda_0}(\xhb)-f_{\lambda_0}(\xhb,\xi_j)\right]\nonumber
\right]=\E\left[\frac{1}{N}\sum_{j=1}^N \left[f_{\lambda_0}(\xhb,\xi_j')-f_{\lambda_0}(\xhb,\xi_j)\right]
\right]\nonumber
%\\=&\frac{1}{N}\sum_{j=1}^N \E\left[f(\xhb,\xi_j')-f(\xhb^{(j)},\xi_j')\right]
%\\=&\frac{1}{N}\sum_{j=1}^N \E\left[f(\xhb^{(j)},\xi_j)-f(\xhb,\xi_j)\right]
\\=&\frac{1}{2N}\sum_{j=1}^N \E\left[f_{\lambda_0}(\xhb,\xi_j')-f_{\lambda_0}(\xhb^{(j)},\xi_j')\right]+\frac{1}{2N}\sum_{j=1}^N \E\left[f_{\lambda_0}(\xhb^{(j)},\xi_j)-f_{\lambda_0}(\xhb,\xi_j)\right]\nonumber
\\\leq &\frac{1}{2N}\sum_{j=1}^N \E\left[\langle G_{\lambda_0}(\xhb,\xi_j')-H_{\lambda_0}(\xhb),\,\xhb-\xhb^{(j)}\rangle\right] +\frac{1}{2N}\sum_{j=1}^N \E\left[\langle G_{\lambda_0}(\xhb^{(j)},\xi_j)-H_{\lambda_0}(\xhb^{(j)}),\,\xhb^{(j)}-\xhb\rangle\right]\nonumber
\\&+\frac{1}{2N}\sum_{j=1}^N\E[\langle H_{\lambda_0}(\xhb)-H_{\lambda_0}(\xhbj),\,\xhb-\xhbj\rangle].\nonumber
\end{align}
Invoking \eqref{new useful inequality to ensure v1}, Young's inequality, and the assumption that $q'\leq q$, we further obtain
\begin{align}
&\E[F_{\lambda_0}(\xhb)-F_{\lambda_0,N}(\xhb)]\nonumber
% \\{\leq}&\frac{1}{2N}\sum_{j=1}^N \E \left[\frac{8\left\Vert G_{\lambda_0}(\xhb,\xi_j')-  H_{\lambda_0}(\xhb)\right\Vert_p^2}{N(q'-1)\lambda_0}+\frac{8\left\Vert \nabla  F_{\lambda_0}(\xhbj)-G_{\lambda_0}(\xhbj,\xi_j)\right\Vert_p^2}{N (q'-1)\lambda_0}+\textcolor{black}{2\mathcal M \Vert \xhb-\xhbj\Vert_{q'}}
% \right.\nonumber
% \\&\left.+\left(\frac{N(q'-1)\lambda_0}{16}+\mathcal L\right)\Vert\xhb-\xhb^{(j)}\Vert_{q'}^2\vphantom{\frac{M}{N}}+\lambda_0 (\Vert \xhbj -\xbf^0\Vert_{q'}+\Vert  \xhb-\xbf^0\Vert_{q'}) \Vert\xhbj-\xhb\Vert_{q'}\right] \label{final result almost GC}
\\{\leq}&\frac{1}{2N}\sum_{j=1}^N \E \left[\frac{8}{N(q'-1)\lambda_0}\left\Vert G_{\lambda_0}(\xhb,\xi_j')- H_{\lambda_0}(\xhb)\right\Vert_p^2 %+\frac{N(q'-1)\lambda_0}{8}\Vert \xhb-\xhbj\Vert^2_q
+\frac{8}{N (q'-1)\lambda_0}\left\Vert H_{\lambda_0}(\xhbj)-G_{\lambda_0}(\xhbj,\xi_j)\right\Vert_p^2\right.\nonumber
\\+&\left.\frac{8\mathcal M^2}{N(q'-1)\lambda_0}+\left(\frac{N(q'-1)\lambda_0}{4}+\mathcal L\right)\Vert\xhb-\xhb^{(j)}\Vert_{q'}^2\right.\left.\vphantom{\frac{M}{N}}+\frac{4\lambda_0}{N(q'-1)}  (\Vert \xhbj -\xbf^0\Vert_{q'}+\Vert  \xhb-\xbf^0\Vert_{q'})^2\right]. \label{a second item to combine}
\end{align}
Recall that (i) it is assumed that $N\geq \frac{C_2\mathcal L}{(q'-1)\lambda_0}$; (ii) $\xhbj$ and $\xhb$ are identically distributed; and (iii) Assumption \ref{assumption: Variance everywhere} holds. We then combine  \eqref{proven gap here} and \eqref{a second item to combine} above to obtain %\textcolor{black}{$20\lambda_0$ in 57, and should be equality in 58}
\begin{align}
\E[F_{\lambda_0}(\xhb)-F_{\lambda_0,N}(\xhb)] 
%\\{\leq}&\frac{8\sigma_p^2+4\mathcal M^2}{N(q'-1)\lambda_0} +\frac{3 N(q'-1)\lambda_0}{16} \E[\Vert\xhb-\xhbj\Vert_{q'}^2]+ \frac{8\lambda_0}{N(q'-1)}\cdot \E[\Vert \xhb-\xbf^0\Vert_{q'}^2]\nonumber %\textcolor{black}{Modified}%+% N(q'-1)\lambda_0\cdot \Vert\xhb-\xhbj\Vert_q
\leq &\, \widehat C_2\frac{\sigma_p^2+\mathcal M^2}{(q'-1)\lambda_0 N}+\widehat C_2\frac{\lambda_0}{N\cdot (q'-1)}\E[\Vert  \xhb-\xbf^0\Vert_{q'}^2]  +\frac{\widehat C_2\delta^2N}{\lambda_0(q'-1)}
\label{final result almost GC second last}
\\= &\,\widehat C_2\frac{\sigma_p^2+\mathcal M^2}{(q'-1)\lambda_0 N}+2\widehat C_2\frac{\lambda_0}{N (q'-1)}\E[V_{q'}  (\xhb)]{+\frac{\widehat C_2\delta^2N}{\lambda_0(q'-1)}}, \label{final result almost GC last}
\end{align}
for some universal constant $\widehat C_2>0$,
where     \eqref{final result almost GC last} holds by the definition of $V_{q'}$. In view of   the definition of $\xhb$, %\textcolor{black}{$\mathbb E$ of the fist line} %\textcolor{black}%{superlink in EC.11 in the following is not directed to EC.11}
%\begin{align}
%& \E\left[\frac{10 \delta^2}{\lambda_0}  +\frac{1}{20}\lambda_0 V_{q'}(\xbf^*)\right]
%\nonumber
%\\\overset{Young's}{\geq}&\delta\E\Bigl[\sqrt{2V_{q'}(\xbf^*)}\Bigr]= \delta\E[\Vert\xbf^0-\xbf^*\Vert_{q'}]\nonumber
\begin{align}
& \E\left[\frac{10 \delta^2}{\lambda_0} +\frac{1}{10}\lambda_0V_{q'}(\xhb)+\frac{1}{10}\lambda_0 V_{q'}(\xbf^*)\right]
\nonumber
\\\overset{\text{Young's}}{\geq}~&\delta\E\Bigl[\sqrt{2V_{q'}(\xhb)}+\sqrt{2V_{q'}(\xbf^*)}\Bigr]= \delta\E[\Vert\xhb-\xbf^0\Vert_{q'}+\Vert\xbf^0-\xbf^*\Vert_{q'}]\geq \delta\E[\Vert\xhb-\xbf^*\Vert_{q'}]\nonumber
\\
\overset{\text{Eq.\,\eqref{suboptimality in computing SAA}}}{\geq}& \E[F_{\lambda_0,N}(\xhb)-F_{\lambda_0,N}(\xbf^*)]=\E[F_N(\xhb)+\lambda_0V_{q'}(\xhb)-F_N(\xbf^*)-\lambda_0 V_{q'}(\xbf^*)]\nonumber
%\\=&\E[F_N(\xhb)+\lambda_0V_{q'}(\xhb)-F(\xbf^*_{\lambda_0})-\lambda_0 V_{q'}(\xbf^*_{\lambda_0})]\textcolor{black}{To delete this line}\nonumber
\\=&\E[F_N(\xhb)+\lambda_0V_{q'}(\xhb)-F(\xbf^*)-\lambda_0 V_{q'}(\xbf^*)]\nonumber
\\\overset{\text{Eq.\,\eqref{final result almost GC last}}}{\geq} &\E[F(\xhb)+\lambda_0V_{q'}(\xhb)-F(\xbf^*)-\lambda_0 V_{q'}(\xbf^*)]-\widehat C_2\frac{\sigma_p^2+\mathcal M^2}{(q'-1)\lambda_0 N}-2\widehat C_2\frac{\lambda_0}{N (q'-1)}\E[V_{q'}  (\xhb)]-\frac{\widehat C_2\delta^2N}{\lambda_0(q'-1)} \nonumber
\\\geq&\E[\lambda_0V_{q'}(\xhb)-\lambda_0 V_{q'}(\xbf^*)]-\widehat C_2\frac{\sigma_p^2+\mathcal M^2}{(q'-1)\lambda_0 N}-2\widehat C_2\frac{\lambda_0}{N (q'-1)}\E[V_{q'}  (\xhb)]-\frac{\widehat C_2\delta^2N}{\lambda_0(q'-1)}.\nonumber
%\end{align}
\end{align}
 In view of the assumption that $N\geq  C_2\frac{(\sigma_p^2+\mathcal M^2)R^*}{(q'-1) \epsilon^2}\geq C_2\frac{1}{q'-1}$   (where we have utilized the assumption that $\sigma_p\geq 1$, $R^*\geq 1$,  and $0<\epsilon\leq 1$, and we may as well also let $C_2\geq 20 \widehat C_2 $), we can re-arrange the inequality above into the below for some universal constants $\widehat C_3 >0$: %\textcolor{black}{should be $\leq$ instead of =, and should be $\frac{\varepsilon}{2}$}
\begin{align*}
\E[\lambda_0V_{q'}(\xhb)]\leq &\frac{11}{8}\E[\lambda_0V_{q'}(\xbf^*)] +\frac{\widehat C_3 }{ \lambda_0(q'-1)}\left(\frac{\sigma_p^2+\mathcal M^2}{ N}+\delta^2N\right)+\frac{\widehat C_3 \delta^2}{\lambda_0},%\leq  {\frac{21}{32}\epsilon}+\frac{\widehat C_3 }{ \lambda_0(q'-1)}\left(\frac{\sigma_p^2+\mathcal M^2}{ N}+\delta^2N\right)+\frac{\widehat C_3 \delta^2}{\lambda_0}, 
\end{align*}
Joining this inequality, the specification of $\lambda_0:=\frac{\epsilon}{2R^*}$, and \eqref{final result almost GC last} (cf.,  $N \geq C_2\frac{1}{q'-1}\geq  {20}\widehat C_2\frac{1}{q'-1}$   again) leads to %\textcolor{black}{superlink EC.11}%\textcolor{black}{40$\varepsilon$ and $\frac{\varepsilon}{8}$}
\begin{align}
\E[F_{\lambda_0}(\xhb)-F_{\lambda_0,N}(\xhb)]\leq&  \frac{\widehat C_4(\sigma_p^2+\mathcal M^2)}{(q'-1)\lambda_0 N}  + {\frac{11}{160}\epsilon}+\frac{\widehat C_4\delta^2}{\lambda_0}+\frac{\widehat C_4\delta^2N}{\lambda_0(q'-1)},
\label{end of all GC}
\end{align}
for some universal constant $\widehat C_4>0$. % leads to $N\geq \frac{C_2}{q'-1}$. We may let $C_2=320$.
Then, combining \eqref{start of all GC new}, \eqref{end of all GC}, and the assumption that $\delta\leq  \frac{1}{N} $, we obtain the first inequality of this theorem in \eqref{thm 2 subopt RSAA optimal GC}  after some re-organization. 

Furthermore, if we invoke Markov's inequality  together with \eqref{start of all GC new} and \eqref{end of all GC}, we then obtain \eqref{thm 2 subopt RSAA optimal GC probability} as the second inequality of  this theorem.
\end{proof}
 \smallskip

 \begin{remark}\label{remark: Comparison with SMD convex}
If one  selects  $R^*$   to be comparable to $V_{q'}(\xbf^*)$, it can be  observed that Theorem \ref{thm: second main theorem convex optimal rate} confirms the promised sample complexity in  \eqref{summary first result} for SAA \eqref{Eq: SAA-ell2}  when it is applied to solving a  convex SP problem. Similar to Remark \ref{remark: Comparison with SMD strongly convex} for the strongly convex case above, this theorem shows that the rate of SAA's sample complexity matches  with that of the canonical  SMD methods as discussed by \cite{nemirovski2009robust,ghadimi2013stochastic} and \cite{lan2020first} in solving convex SP problems.  Nonetheless, it is worth noting here that an accelerated variation of  SMD, called   AC-SA \citep{lan2012optimal}, achieves a   better rate on $\mathcal L$ with some more careful design \textcolor{black}{\label{dependence on L revision}--- namely, AC-SA has   a better
dependence on $\mathcal{L}$}.
 \end{remark}

\begin{remark}\label{one norm case}
The stipulation of $q>1$ (and thus not including the choice of $q=1$) is non-critical. Indeed, in the non-trivial case with $d>1$, following the existing discussions of   SMD in the 1-norm setting  \citep{nemirovski2009robust}, the case where Assumption \ref{L-smoothness} holds for  $q=1$ can be subsumed by the consideration of the case with $q=1+\frac{1}{\ln d}>1$ by the fact that 
\begin{align}\Vert \vbf \Vert_{1+\frac{1}{\ln d}}\leq \Vert\vbf\Vert_1\leq e\cdot \Vert \vbf \Vert_{1+\frac{1}{\ln d}},\label{convert norms here}
\end{align} where $e$ is the base of   natural logarithm.
\end{remark}

A proper selection of $\lambda_0$ for this theorem relies on $R^*$, an overestimate of $V_{q'}(\xbf^*)$, which is  equal to half of the squared $q'$-norm distance between the optimal solution $\xbf^*$ and any user-specified initial guess $\xbf^0$. Assuming (straightforward variations of) the knowledge of such a distance  is not uncommon in related literature \citep[e.g., as in][]{loh2011high,loh2017statistical,liu2022high}. In practice, when little is known about the SP's problem structure, one may choose $\xbf^0$ to be any feasible solution and specify $R^*$ to be coarsely large; for instance, one may let $R^*$  be half of the squared $q'$-norm diameter of $\X$, if it is bounded. Starting from this  coarse selection, one may then perform some empirical hyperparameter search for better values of $R^*$  (and thus $\lambda_0$) with the aid of cross validation. Meanwhile, if some problem structure about the SP problem is known, one may incorporate such {\it a priori} knowledge into the construction of $V_{q'}$. For instance, if it is known that $\xbf^*$ satisfies the weak sparsity condition (or the budget/capacity constraint) that $\Vert \xbf^*\Vert_1\leq r$ for some known $r$ \citep{negahban2012unified,bugg2021logarithmic}, then, in view of  \eqref{convert norms here} above, we may construct the regularization term with $q'=1+\frac{1}{\ln d}$ and $\xbf^0=\mathbf 0$. Correspondingly,  $R^*=
\frac{1}{2}\cdot e^2\cdot r^2$.

\begin{remark}\label{discussion in dimension on d}
In the results of both Theorems \ref{thm: suboptimality} and \ref{thm: second main theorem convex optimal rate}, the complexity bounds are completely free from any metric entropy terms, leading to new complexity rates that exhibit significantly better dependence on $d$ than  the complexity benchmark \eqref{reduced rate}.  Furthermore, while   \eqref{reduced rate} is applicable to light-tailed SP problems, our results hold under heavy tails with only a bounded second moment of the underlying randomness.  
To complement Theorems \ref{thm: suboptimality} and \ref{thm: second main theorem convex optimal rate} above,  we are to  additionally discuss our metric entropy-free results specifically under light tails and their comparison with \eqref{reduced rate}   later in Section \ref{sec: Metric entropy-free}.
\end{remark}

\begin{remark} \label{remark: compare to OT p=2}
In comparison with the state-of-the-art  benchmark for the  sample complexity  bounds under heavy tails  by \cite{oliveira2023sample},  our results above are focused on the  ``most heavy-tailed'' cases considered by the benchmark --- only the second moment is assumed to be bounded. In this case, the said benchmark   is summarized below: 
Let $\X^{*,\epsilon}$ be the set of $\epsilon$-suboptimal solutions; namely,
\begin{align}
\X^{*,\epsilon}:=\{\xbf  \in\mathcal X:\, F(\xbf)\leq F(\xbf^*)+\epsilon\}.\label{suboptimal solution set}
\end{align}
Suppose that $\X^{*,\epsilon}$ is bounded and
\begin{align}
\vert f(\xbf,\xi)-f(\ybf,\xi)\vert\leq M(\xi)\cdot \Vert\xbf-\ybf\Vert_q,~~\forall\, \xbf,\ybf\in\X^{*,\epsilon},~\xi\in\Xi,\label{Eq: holder inequality assumption}
\end{align}
then
it holds that
\begin{align}
\Prob[F(\xhb)-F(\xbf^*)\leq \epsilon]\geq 1-\beta-2\rho,
%\\
~~\text{if }N\geq% \begin{cases} 
 % O\left(\frac{\mathbf M_2\cdot     \left(d+ \ln \frac{1}{\beta}\right)}{\epsilon\cdot \mu} +\frac{\upsilon_{\xbf^*}^2}{\epsilon^2}\right) &\text{for SAA \eqref{Eq: SAA} in $\mu$-strongly convex SP};
 % \\
 O\left(\frac{\mathbf M_2\cdot   \left[\Big(\gamma(\X^{*,\epsilon})\Big)^2 + (\mathcal D^{*,\epsilon})^{2}\cdot\ln \frac{1}{\beta}\right]}{\epsilon^2} \right),
%  &\text{for SAA  \eqref{Eq: SAA} in convex SP},
% \end{cases}
\label{Eq: history bound further specialized}
\end{align}
 where  $\mathcal D^{*,\epsilon}$ is the diameter of $\X^{*,\epsilon}$ and $\gamma(\X^{*,\epsilon})$ is a generic chaining-based metric entropy term that grows at the rate of $O(\sqrt{d}\cdot \mathcal D^{*,\epsilon})$ in general ---   elevating the dependence on $d$ in the sample complexity again. Meanwhile, $\mathbf M_2$ is the second moment of  $M(\xi)$ and the  probability term   $\rho$  is given as $$\rho:=\max\left\{\Prob\left[N^{-1}\sum_{j=1}^N\Big(M(\xi_j)\Big)^2>2\mathbf M_2\right],\,\Prob\left[N^{-1}\sum_{j=1}^N\left[f(\xbf^*,\xi_j)-F(\xbf^*)\right]^2>2\upsilon_{\mathbf x^*}\right]\right\},$$
 for $\upsilon_{\xbf^*}$ being the variance of $f(\xbf^*,\xi)$.  Note that the relationship between the term $\rho$ and sample size $N$ is not explicitly reflected in the sample size requirement in \eqref{Eq: history bound further specialized}. Some further explication of this dependence typically requires additional assumptions on the underlying randomness. We provide one illustration (in comparable conditions to our result for SP problems under light tails) in Remark \ref{comparison results to be added}.

To facilitate comparison between our results and \eqref{Eq: history bound further specialized}, we consider some conversions of notations below. %First, %by definition, it is verifiable that  $\mathbf M_2\approx\sup_{\xbf\in\X^{*,\epsilon}}\mathbb E[\Vert G(\xbf,\xi)\Vert_p^2]$, which is comparable to $\sigma_p^2+\mathcal M^2$   in general when $\mathcal L=0$. 
{\color{black}First, under Assumption \ref{L-smoothness}, $\mathbf M_2$ admits the characterization below: Since $\E[G(\xbf,\xi)]=\nabla F_1(\xbf)+\mathbf g_{F_2}$ for some $\mathbf g_{F_2}\in \partial F(\xbf)-\nabla F_1(\xbf)$, it follows that, in some relatively favorable regimes (e.g., when the problem is unconstrained), $\mathbf M_2 \approx O\!\left(\mathcal L^2\cdot (\mathcal D^{*,\epsilon})^2 + \mathcal M^2 + \sigma_p^2\right)$.}
Second, it is also worth noting that  $\mathcal D^{*,\epsilon}$ is generally comparable  to $\mathcal D_{q}$, the $q$-norm diameter of the feasible region $\X$. Indeed, one may easily construct scenarios where the largest distance between any two $\epsilon$-suboptimal solutions can be insignificantly different from the largest possible distance between any two feasible solutions. One such example is for the expected objective function to be close to a  constant along the affine subspace that includes the pair of feasible solutions with the largest distance between them. Likewise, in general, $M(\xi)$ has to be large enough to apply globally    for all $\xbf,\,\ybf\in\X$.  Thus, \eqref{Eq: holder inequality assumption} is also comparable to \eqref{Lipschitz condition}, which is verifiably more stringent than   Assumption \ref{L-smoothness} for our results. 
{\color{black}
Furthermore, in the case of   $\mu$-strong convexity,  one can further explicate  \eqref{Eq: history bound further specialized}  by noting that $(\mathcal D^{*,\epsilon})^2\leq O(\mu^{-1}\epsilon)$. With the conversion of notations above, as well as the fact that $\gamma(\X^{*,\epsilon})\leq O(\sqrt{d}\cdot \mathcal D^{*,\epsilon})$, we may rewrite \eqref{Eq: history bound further specialized} as $\Prob[F(\xhb)-F(\xbf^*)\leq \epsilon]\geq 1-\beta-2\rho$, if it holds that
\begin{align}
N\geq   \begin{cases}O\Bigg(\left(\frac{\mathcal L^2}{\mu^2}+\frac{\sigma_p^2+\mathcal M^2}{\mu\epsilon} \right)\cdot \bigg(d+\ln(1/\beta)\bigg)\Bigg)&\text{for $\mu$-strongly convex SP};
\\
\\
O\left( \mathcal D^2_q\cdot \frac{(\mathcal L^2\cdot \mathcal D^2_q+\sigma_p^2+\mathcal M^2)\cdot (d+\ln(1/\beta))}{\epsilon^{2}}\right)&\text{for  convex SP}.
\end{cases}
%  &\text{for SAA  \eqref{Eq: SAA} in convex SP},
% \end{cases}
\label{Eq: history bound further specialized for comparison}
\end{align}
}

A side-by-side comparison between \eqref{Eq: history bound further specialized for comparison} and our results in Theorems \ref{thm: suboptimality} and \ref{thm: second main theorem convex optimal rate}  is provided in Table \ref{side-by-side comparison}. One may see that our complexity bounds  entail  three potential advantages as summarized below: 
\begin{table}[h!]
{\color{black}
\centering
\small
\caption{Comparison of sample complexity bounds relative to \eqref{Eq: history bound further specialized for comparison} under Assumption~\ref{L-smoothness}. Here, \(\mathcal{E}(\epsilon):=\{F(\xhb)-F(\xbf^*)\le \epsilon\}\),  $ \mathcal D_q^2\approx V_{q'}(\xbf^*)$, and $q'$ is a hyperparameter. All universal constants are suppressed. }\label{side-by-side comparison}
\renewcommand{\arraystretch}{1.2}
\setlength{\tabcolsep}{6pt}
\begin{center}
\begin{tabular}{l|c|c}
\toprule
& \textbf{Benchmark sample bounds in \eqref{Eq: history bound further specialized for comparison}} 
& \textbf{Sample bounds in Theorems \ref{thm: suboptimality} and \ref{thm: second main theorem convex optimal rate}} \\[1mm]
\hline\addlinespace[1.5pt]
Target&$\Prob[\mathcal E(\epsilon)]\ge 1-\beta-2\rho$& \(\Prob[\mathcal E(\epsilon)]\ge 1-\beta\)
\\[0.5mm]\hline
\multicolumn{3}{c}{\textbf{Strongly convex} (\(\mu>0\))}\\[0.5mm]\hline\addlinespace[3pt]
\(\mathcal L=0\) 
& \( \dfrac{(\sigma_p^2+\mathcal M^2)\big(d+\ln(1/\beta)\big)}{\mu\,\epsilon} \)
& \( \dfrac{\sigma_p^2+\mathcal M^2}{\mu\,\epsilon\,\beta}\)\\[2mm]\addlinespace[3pt]
\(\mathcal M=0\) 
& \( \left(\dfrac{\mathcal L^2}{\mu^2}+\dfrac{\sigma_p^2}{\mu\,\epsilon}\right)\!\big(d+\ln(1/\beta)\big) \)
& \( \max\!\left\{\dfrac{\mathcal L}{\mu},\,\dfrac{\sigma_p^2}{\mu\,\epsilon\,\beta}\right\}\)\\[2mm]\addlinespace[3pt]
\(\mathcal L,\mathcal M>0\) 
& \( \left(\dfrac{\mathcal L^2}{\mu^2}+\dfrac{\sigma_p^2+\mathcal M^2}{\mu\,\epsilon}\right)\!\big(d+\ln(1/\beta)\big) \)
& \( \max\!\left\{\dfrac{\mathcal L}{\mu},\,\dfrac{\sigma_p^2+\mathcal M^2}{\mu\,\epsilon\,\beta}\right\}\)\\[2mm]\hline\addlinespace[3pt]
\multicolumn{3}{c}{\textbf{Convex} (\(\mu=0\))}\\[0.5mm]\hline\addlinespace[3pt]
\(\mathcal L=0\) 
& \( \mathcal D_q^2\cdot \dfrac{\big(\sigma_p^2+\mathcal M^2\big)\big(d+\ln(1/\beta)\big)}{\epsilon^{2}} \)
& \(\dfrac{ V_{q'}(\xbf^*)}{\,q'-1\,}\cdot \dfrac{\sigma_p^2+\mathcal M^2}{\beta\,\epsilon^{2}}\)\\[2mm]\addlinespace[3pt]
\(\mathcal M=0\) 
& \( \mathcal D_q^2\cdot \dfrac{\big(\mathcal L^2\mathcal D_q^2+\sigma_p^2\big)\big(d+\ln(1/\beta)\big)}{\epsilon^{2}} \)
& \(\dfrac{ V_{q'}(\xbf^*)}{\,q'-1\,}\cdot \max\!\left\{\dfrac{\mathcal L}{\epsilon},\,\dfrac{\sigma_p^2}{\beta\,\epsilon^{2}}\right\}\)\\[2mm] \addlinespace[3pt]
\(\mathcal L,\mathcal M>0\) 
& \( \mathcal D_q^2\cdot \dfrac{\big(\mathcal L^2\mathcal D_q^2+\sigma_p^2+\mathcal M^2\big)\big(d+\ln(1/\beta)\big)}{\epsilon^{2}} \)
& \(\dfrac{ V_{q'}(\xbf^*)}{\,q'-1\,}\cdot \max\!\left\{\dfrac{\mathcal L}{\epsilon},\,\dfrac{\sigma_p^2+\mathcal M^2}{\beta\,\epsilon^{2}}\right\}\)\\[3mm]\bottomrule\addlinespace[3pt]
\end{tabular}
\end{center}
}
\end{table}
\begin{itemize}
\item First, similar to Remark \ref{discussion in dimension on d}, the bounds provided in the said theorems  also exhibit    non-trivially better rates with dimensionality $d$ than \eqref{Eq: history bound further specialized for comparison} (and thus \eqref{Eq: history bound further specialized}) in general, due to the worst-case polynomial growth of $[\gamma(\X^{*,\epsilon})]^2$ with $d$.%, due to the avoidance of metric entropy terms.  
\item Second, our results  make use of the potential smoothness of the objective function to obtain  sharper bounds. 
In  the more adversarial case of $\mathcal L=0$, our derived complexity grows linearly with  $\mathcal M^2$, leading to comparable   rates to  \eqref{Eq: history bound further specialized for comparison}   (and, thus,  to \eqref{Eq: history bound further specialized}).  
Nonetheless, when   $\mathcal L$  is more dominant than $\mathcal M$,   our new   complexity  bounds  become potentially more efficient.
Particularly, in the more desirable case of $\mathcal M=0$ and $\epsilon$ is reasonably small,  the rates in \eqref{thm 2 subopt} and \eqref{thm 2 subopt RSAA optimal GC probability} of our theorems above can be simplified into
\begin{multline*}
\Prob\left[F(\xhb)-F(\xbf^*)\leq \epsilon\right]\geq 1-\beta, ~~
\text{if\,}N\geq \begin{cases}
O\left( \frac{\sigma_p^2}{\mu\cdot \epsilon\cdot \beta}\right)&\text{for SAA  \eqref{Eq: SAA} in $\mu$-strongly convex SP};
\\[8pt]
O\left( \frac{V_{q'}(\xbf^*)}{q'-1} \frac{\sigma_p^2}{\epsilon^2\cdot \beta} \right)&\text{for SAA  \eqref{Eq: SAA-ell2} in convex SP},
\end{cases}
\end{multline*}
showing a region of parameters to allow SAA's sample complexity to be free from the impact of Lipschitz constants of $\nabla F_1$ and $F_2$.   {\color{black}In contrast,  the benchmark results in \eqref{Eq: history bound further specialized} and  \eqref{Eq: history bound further specialized for comparison} do not explicitly benefit from  the presence (or dominance) of a  component  that admits Lipschitz continuous gradient in the objective function.}
\item Third, our bounds in \eqref{thm 2 subopt} and \eqref{thm 2 subopt RSAA optimal GC probability} provide  explicit dependence on  the significance level $\beta$, while, in contrast, \eqref{Eq: history bound further specialized} carries a more opaque quantity $\rho$. Although further explications of \eqref{Eq: history bound further specialized} have been provided by \cite{oliveira2023sample}, stronger conditions on lighter tails are additionally stipulated therein.
\item Last, while the benchmark result assumes a bounded feasible region, Theorems \ref{thm: suboptimality} and \ref{thm: second main theorem convex optimal rate}    apply to scenarios where   $\X$ is potentially unbounded. 
\end{itemize}
\end{remark}
%Again, the results by \cite{oliveira2023sample} carry the metric entropy term, thus exhibiting a heavy dependence on dimesnionality. In contrast, our  results are free of any metric entropy-terms and thus  substantially less  

{\color{black}
\begin{remark}\label{rk: variance v2}
As noted in Remark~\ref{rk: variance}, the value of $\sigma_p^2$ generally depends on $d$. If, in addition, there exists $\phi_p\geq 0$ and $p\ge2$ such that
$\|G_i(\xbf,\xi)-\E [G_i(\xbf,\xi)]\|_{L^p}\le \phi_p$ for each  $i=1,\ldots,d$  and every $\xbf\in\X$,
then we can further invoke \eqref{Eq: bounding dependence on d} to explicate  the said dependence by  $\sigma_p^2\leq d^{2/p}\phi_p^2$. Correspondingly, the dimensional dependence  diminishes   when  it is possible to let $p\ge c\ln d$  for some constant $c>0$.
\end{remark}}

 \begin{remark}\label{RO stability}
 An important component of our proofs resorts to a seemingly novel argument  based on the ``average RO stability'' \citep{shalev2010learnability}, which is related to the average stability \citep{rakhlin2005stability},  uniform-RO stability \citep{shalev2010learnability}, and uniform stability \citep{bousquet2002stability}. While it is known that the average-RO stability can lead to error bounds for learning algorithms \citep{shalev2010learnability}, seldom is there a sample complexity bound for  SAA based on such a stability type in comparable settings of our consideration.  In contrast, as mentioned in Section \ref{sec: related results}, most existing  SAA theories are based on either the ``uniform convergence'' theories, such as the $\epsilon$-net \citep{shapiro2021lectures} or the generic chaining \citep{oliveira2023sample}, or the  variations of uniform   (RO-) stability theories, such as by \citet{feldman2019high,shalev2010learnability,shalev2009stochastic}, and  \citet{klochkov2021stability}.  Therefore, we think that our average-RO stability-based proof approach may also   be of independent interest to some readers. 
 \end{remark}%One may see more discussions on how the average-RO stability is incorporated in our proofs from Remark \ref{use of RO}.% and later Appendix \ref{sec: Proof of MT2}.
 
{\label{comment artifact 1}\color{black} Our results in \eqref{thm 2 subopt} and \eqref{thm 2 subopt RSAA optimal GC probability} of Theorems \ref{thm: suboptimality} and \ref{thm: second main theorem convex optimal rate}, respectively,  provide sample complexity bounds at the rate of $O(1/\beta)$ in ensuring a significance level of $\beta\in(0,1)$.  Parts (i) and (ii) of Proposition \ref{lower bound} in Appendix \ref{beta subsection lower bound}  show that such a rate is intrinsic to the SAA problem, instead of a proof artifact.}

\subsection{Large deviations-type, metric entropy-free complexity  bounds}\label{sec: Metric entropy-free}

The previous subsection presents a complexity rate of $O(1/\beta)$ with the significance level $\beta\in(0,1)$. In comparison, when the underlying randomness is light-tailed, many existing results, such as in \eqref{summary typical results} and \eqref{reduced rate},  present more desirable, (poly-)logarithmic  dependence on  $1/\beta$ (although most of these results are polynomial in some  metric entropy terms). It then prompts the question  whether   metric entropy-free bounds can also be derived while preserving a  (poly-)logarithmic rate with $1/\beta$ under similar conditions.    Our affirmative answer to this question is presented in this subsection. %while   This section then answers the question of whether 

The assumptions of consideration include a formalized statement of the Lipschitz condition in \eqref{Lipschitz condition} and  that of the underlying randomness, both provided in  the below:

\begin{assumption}\label{assumption f lipschitz}
There exist constants $q>1$, $p>2$, and $\psi_M\ge 1$, and a deterministic and $\mathcal B(\Xi)$-measurable function
$M:\Xi\to\R_+$ such that:
\begin{itemize}
\item[(a)] For any $(\xbf,\ybf)\in\X^2$ and every $\xi\in\Xi$,  
$
\vert f(\xbf,\xi)-f(\ybf,\xi)\vert\leq M(\xi)\cdot \Vert\xbf-\ybf\Vert_q $.
\item[(b)]  It holds that   $\Vert M(\xi)\Vert_{L^p}\leq \psi_M$.
\end{itemize}
\end{assumption}

Assumption \ref{assumption f lipschitz}.(a) is imposed w.r.t.\,the $q$-norm. This is a very common condition in the SAA literature \citep{shapiro2021lectures,shapiro2003monte,shapiro2005complexity}. Compared to the uniform Lipschitz condition as in Eq.\,\eqref{lipschitz global}, Assumption \ref{assumption f lipschitz}.(a) does not impose that $M(\xi)$ should be $\xi$-invariant and thus is more flexible. Meanwhile, Assumption \ref{assumption f lipschitz}.(b) means that $M(\xi)$ should have a finite $p$th  moment, though this random variable can still be heavy-tailed. A related condition is  also imposed by \cite{oliveira2023sample}. 
The specification that $q>1$ (and thus excluding the case of $q=1$) is  non-critical due to the same argument as in Remark \ref{one norm case}.

For perhaps  more interesting results of this section, we impose  light-tailed counterparts   to the above as stated in Assumptions \ref{assumption f lipschitz v2} and \ref{assumption f lipschitz v3} below:
\begin{assumption}\label{assumption f lipschitz v2}
There exist constants $\varphi\geq 1$, $q>1$ and a deterministic and $\mathcal B(\Xi)$-measurable function $M:\Xi\rightarrow\R_+$ such that the following hold:
\begin{itemize}
\item[(a)]  Assumption \ref{assumption f lipschitz}.(a) holds w.r.t.\,the $q$-norm.
\item[(b)] For all $t\geq 0$, it holds that
$
\Prob[M(\xi)\geq t]\leq 2\exp\left(- {t}/{\varphi}\right).
$
\end{itemize}
\end{assumption}
  Assumption \ref{assumption f lipschitz v2}.(b) imposes subexponential tails, a common form of light-tailed assumptions, for the underlying distribution and  is    comparable to the  light-tailed  conditions  discussed  by \cite{shapiro2003monte,shapiro2005complexity} and \cite{shapiro2021lectures}. 

\begin{assumption}\label{assumption f lipschitz v3}
There exist constants $\varphi\geq 1$, $q>1$, and a deterministic and $\mathcal B(\Xi)$-measurable function $M:\Xi\rightarrow\R_+$ such that the following hold:
\begin{itemize}
\item[(a)]  Assumption \ref{assumption f lipschitz}.(a) holds  w.r.t.\,the $q$-norm.
\item[(b)] For   all $t\geq 0$, it holds that
$
\Prob[M(\xi)\geq t]\leq 2\exp\left(- {t^2}/{\varphi^2}\right).
$
\end{itemize}
\end{assumption}
  
 Assumption \ref{assumption f lipschitz v3}.(b) imposes sub-Gaussian tails, which is another form of light-tailed  assumptions for the underlying distribution and is relatively stronger than Assumption \ref{assumption f lipschitz v2}.(b). This assumption   is   similar to  the counterpart conditions imposed by \cite{nemirovski2009robust} in establishing a  large deviations bound for   SMD. 

{\color{black}
\begin{remark} 
\label{rem:orlicz-view}
It is worth noting that Assumptions~\ref{assumption f lipschitz}-\ref{assumption f lipschitz v3} admit a unified formulation in terms of Orlicz norms. Let $\Psi:\R_+\to\R_+$ be a Young function and define the  
Orlicz norm of a nonnegative random variable $X$ by
\[
\|X\|_{\Psi}
:= \inf\bigl\{ t>0 : \E[\Psi(X/t)] \le 1 \bigr\}.
\]
Then the finiteness of $\|M(\xi)\|_{\Psi}$ corresponds to the following
standard choices of~$\Psi$:
\begin{itemize}
\item[(i)] \textit{Finite $p$th moment (potentially heavy-tailed).}
For $\Psi(t)=t^p$ with $p>2$, finiteness of $\|M(\xi)\|_{\Psi}$
is equivalent (up to constants depending only on $p$) to 
$\|M(\xi)\|_{L^p}<\infty$, which is equivalent to  Assumption~\ref{assumption f lipschitz}.(b).

\item[(ii)] \textit{Subexponential tails.}
For $\Psi_1(t)=\exp(t)-1$, finiteness of $\|M(\xi)\|_{\Psi_1}$ implies
a subexponential tail bound of the form, for some constants $c_1'>0$:
\[
\Prob[M(\xi)\ge t] \le 2\exp\Big(-c_1'\cdot t/\|M(\xi)\|_{\Psi_1}\Big), \quad \text{for } t\ge 0,
\]
which matches Assumption~\ref{assumption f lipschitz v2}.(b).

\item[(iii)] \textit{Sub-Gaussian tails.}
For $\Psi_2(t)=\exp(t^2)-1$, finiteness of $\|M(\xi)\|_{\Psi_2}$ implies
a sub-Gaussian tail bound, for some constants $c_2'>0$:
\[
\Prob[M(\xi)\ge t] \le  2 \exp\Big(-c_2'\cdot t^2/\|M(\xi)\|^2_{\Psi_2}\Big), \quad \text{for } t\ge 0,
\]
which matches Assumption~\ref{assumption f lipschitz v3}.(b). 
\end{itemize}
We therefore can view Assumptions~\ref{assumption f lipschitz}--\ref{assumption f lipschitz v3} as three special variations of Orlicz-type conditions. For additional discussions on Orlicz norm and Orlicz space, we would like to refer interested readers to \cite{vershynin2018high}. 
\end{remark}
}

We are now ready to formally state our sample complexity  below, where we recall the notation that $\mathcal D_{q'}$ is the $q'$-norm diameter of the feasible region. Note that we also strengthen Assumption \ref{GC condition constant all} from `almost every $\xi\in\Xi$' to `every $\xi\in\Xi$' in this subsection.% and  $e$ is the base of natural logarithm. 

\begin{proposition}\label{thm: large deviation} Let $q\in(1,2]$, $\epsilon\in(0,1]$, and $\beta\in(0,1)$ be any fixed scalars. Denote by $\xhb$  a $(\delta,q)$-approximate solution to  SAA    \eqref{Eq: SAA-ell2} with its parameters specified to be
\begin{align}
q'\in(1,2]:\,q'\leq q,~~~R^*\geq\max\{1,\mathcal D_{q'}^2\},~~~\lambda_0=\frac{\epsilon}{2R^*},~~\text{and}~~\delta\leq \frac{1}{N}.\label{parameter choices new}
\end{align}
 Suppose that   (i) $N\geq 3$,  (ii)    $f(\cdot,\xi)$ is convex on $B$ for every $\xi\in\Xi$, and (iii) $\X$ is bounded with $\mathcal D_{q'}$ being its $q'$-norm diameter; namely $\sup_{\xbf_1,\xbf_2\in\X}\Vert \xbf_1-\xbf_2\Vert_{q'}\leq \mathcal D_{q'}$.
\begin{itemize}
\item[(a)] Under Assumption \ref{assumption f lipschitz} w.r.t. the $q$-norm,   
it holds that
\begin{multline*}
\Prob[F(\xhb)-F(\xbf^*)\leq \epsilon]\geq 1-\beta,~~~ \text{if}~~
N\geq    \left(C_3\cdot\frac{\psi_M^{2}}{\epsilon^{2}} \cdot \textcolor{black}{B_{N,\beta,q'}}\cdot R^*\right)^{\frac{p}{p-2}}\cdot {\beta^{-\frac{2}{p-2}}}+ \left(C_3\cdot\frac{{p} \mathcal D_{q'}}{p-1}\cdot\frac{\psi_M}{\epsilon}\right)^{\frac{p}{p-1}}\beta,
\end{multline*}
where $\textcolor{black}{B_{N,\beta,q'}}:= \frac{1}{q'-1}\ln N\cdot \ln\frac{N}{\beta} $ and $C_3>0$ is some universal constant.
\item[(b)] Under Assumption \ref{assumption f lipschitz v2} w.r.t. the $q$-norm, 
 it holds that
\begin{align*}
\Prob[F(\xhb)-F(\xbf^*)\leq \epsilon]\geq 1-\beta,
%~~~ \text{if}~~N\geq C_3\cdot \left[\frac{\mathcal M^2}{\epsilon^2}\cdot B+ \left(\frac{\psi_M^{2}}{\epsilon^{2}}\cdot B\right)^{1-2/p}\cdot\frac{1}{\beta^{2/p-1}}\right],
~~~ \text{if}~~
N\geq C_4\cdot R^*\cdot\frac{\varphi^2}{\epsilon^2}\cdot  \frac{  \ln N\cdot \ln^3  ({ N/\beta})}{q'-1},
\end{align*}
where   $C_4>0$ is some universal constant.

\item[(c)] Under Assumption \ref{assumption f lipschitz v3} w.r.t. the $q$-norm, it holds that
\begin{align*}
\Prob[F(\xhb)-F(\xbf^*)\leq \epsilon]\geq 1-\beta,
%~~~ \text{if}~~N\geq C_3\cdot \left[\frac{\mathcal M^2}{\epsilon^2}\cdot B+ \left(\frac{\psi_M^{2}}{\epsilon^{2}}\cdot B\right)^{1-2/p}\cdot\frac{1}{\beta^{2/p-1}}\right],
~~~ \text{if}~~
N\geq C_5\cdot R^*\cdot \frac{\varphi^2}{\epsilon^2}\cdot   \frac{\ln N\cdot \ln^2  ({ N/\beta})}{q'-1},
\end{align*}
where   $C_5>0$ is some universal constant.
\end{itemize}
\end{proposition}
\begin{proof} The core of the proof  lies in weakening the assumption of the uniform  Lipschitz condition imposed by a pillar result \citep[by][]{feldman2019high} formally stated in Proposition \ref{very important generalization proposition}.  
To that end, we start by introducing a few notations. Let $f_{\lambda_0}(\xbf,\xi):=f(\xbf,\xi)+\lambda_0 V_{q'}(\xbf)$ for $(\xbf,\xi)\in\X\times \Xi$. For any given $t\geq 0$,  denote that 
\begin{align}
\text{$\mathcal E_t:=\{\xi\in\Xi:\,  M(\xi)  \leq t\}$
and   $\mathcal E_t^N:=\{\boldsymbol\xi_{1,N}=(\xi_1,...,\xi_N)\in\Xi^N:\,M(\xi_j)\leq t,~\forall j=1,...,N\}$}.\label{event total consideration}
\end{align}
\textcolor{black}{An illustration of the relations between $\mathcal E_t$ and $\mathcal E_t^N$ is provided in Figure \ref{fig: schema}.(a).}

Furthermore, we use the following notations related to a family of $\zeta_1,...,\zeta_j,...,\zeta_N,\zeta_1',...,\zeta'_j,...,\zeta_N'\in\Xi$:
\begin{itemize}
\item We denote that $\boldsymbol\zeta_{1,N} =(\zeta_1,...,\zeta_N)$ and $\boldsymbol\zeta_{1,N}' =(\zeta_1',...,\zeta_N')$; namely, $\boldsymbol\zeta_{1,N}$ is the collection of all $\zeta_j$ and $\boldsymbol\zeta_{1,N}'$ is that of all $\zeta_j'$. 
\item We also denote that \begin{align}\boldsymbol\zeta^{(j)}_{1,N} =(\zeta_1,...,\zeta_{j-1},\,\zeta_j',\,\zeta_{j+1},...,\zeta_N)\label{special notation}
\end{align}
for all $j=1,...,N$. With this notation, all but one components of $\boldsymbol\zeta^{(j)}_{1,N}$ are the same as $\boldsymbol\zeta_{1,N}$ and the only difference  is  in the $j$th component. Namely, by switching the $j$th component of $\boldsymbol\zeta_{1,N}$  with $\zeta_j'$, one obtains $\boldsymbol\zeta^{(j)}_{1,N}$. 
\end{itemize}
%We are now ready to prove Proposition \ref{thm: large deviation}.
We introduce short-hand notations: \begin{align}\gamma:=\frac{4\cdot (t+\lambda_0 \mathcal D_{q'})^2}{N\cdot \lambda_0\cdot (q'-1)}~~\text{and}~~\varkappa_\delta:=\frac{4\delta\cdot (t+\lambda_0\mathcal D_{q'})}{\lambda_0\cdot (q'-1)}.\label{eq: define gamma} 
\end{align} 
Finally, we define  $d_{\gamma}:\Xi^{N}\times \Xi^N\rightarrow \R_+$ to be the Hamming distance such that
\begin{align}
d_{\gamma}(\boldsymbol\zeta_1,\boldsymbol\zeta_2) = \sum_{j=1}^N \gamma\cdot \mathds{1}(\zeta_{1,j}\neq \zeta_{2,j}),~~~~~\text{for~}(\boldsymbol\zeta_1,\boldsymbol\zeta_2)\in\Xi^{N}\times \Xi^{N}.\label{hamming}
\end{align}
We are now ready to start our proof.

\begin{figure}
    \centering
    \small 
    \begin{tabular}{cc}
    \includegraphics[width=0.47\textwidth]{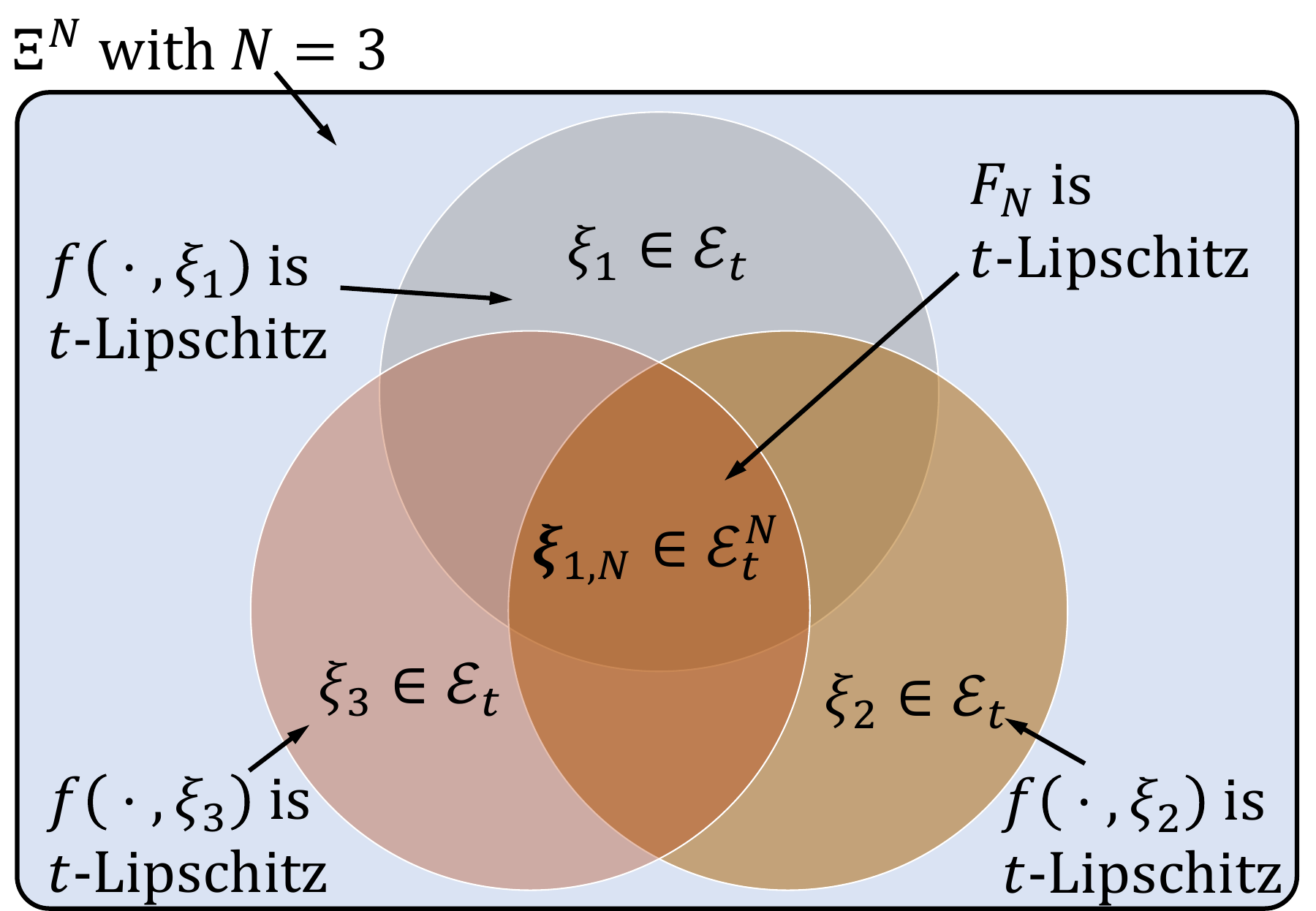} &
% --- First two rows: (a), (c), (e) --------------------------------
\begin{tikzpicture}[
    >=latex,
    node distance=0.7cm,
    font=\small,
    mainstep/.style={
        rounded corners,
        draw=black,
        very thick,
        top color=gray!10,
        bottom color=gray!10,
        inner sep=6pt,
        text width=6cm,
        align=left
    },
    substep/.style={
        rounded corners,
        draw=black,
        thick,
        fill=white,
        inner sep=4pt,
        text width=6cm,
        align=left
    }
]

%--- Step 1: content inside the big box ---
\node[substep, fill=none, draw=none, text width=7cm] (s1title)
    {\textbf{Step 1.} Show $g(\xi_{1,N},\xi)$ satisfies conditions for \textbf{Proposition \ref{very important generalization proposition}}.};

\node[substep, below=0.1cm of s1title] (s11)
    {\textbf{Step 1.1.} Show $g(\,\cdot\,,\zeta)$ is Lipschitz continuous w.r.t.\ Hamming distance $d_\gamma$ for all $\zeta\in\Xi$.};

\node[substep, below=0.1cm of s11] (s12)
    {\textbf{Step 1.2.} Show $g\in[0,2\mathcal D_{q'}(t+\lambda_0\cdot \mathcal D_{q'})]$.};

% Big rounded rectangle around Step 1 things (in the background)
\begin{scope}[on background layer]
\node[mainstep, fit=(s1title) (s11) (s12)] (step1) {};
\end{scope}

%--- Step 2 ---
\node[mainstep, below=0.3cm of step1] (step2)
    {\textbf{Step 2.} Apply {\bf Proposition \ref{very important generalization proposition}} and {\bf Lemma \ref{important lemma: approximation error}} on $g(\xi_{1,N},\xi)$ to obtain inequality~\eqref{almost done complete p explicit}.};

%--- Step 3 ---
\node[mainstep, below=0.3cm of step2] (step3)
    {\textbf{Step 3.} Simplify \eqref{almost done complete p explicit} to obtain the desired.};

%--- Arrows ---
\draw[->, thick] (step1.south) -- ++(0,-0.25) -- (step2.north);
\draw[->, thick] (step2.south) -- ++(0,-0.25) -- (step3.north);

\end{tikzpicture}  \\
(a). %Proof roadmap for Proposition \ref{thm: large deviation} &
&(b).  %Relations between $\mathcal E_t$ and $\mathcal E_t^N$ \\
%\multicolumn{2}{c}{\includegraphics[width=0.8\textwidth]{compare_2_linear.eps}}
%\\
%\multicolumn{2}{c}{(e).}
\end{tabular}
\caption{\color{black}(a) A Venn diagram to illustrate  connections among $\Xi^N$, $\mathcal E_t$, and $\mathcal E_t^N$ in the case with $N=3$, and  (b) a proof roadmap for Proposition \ref{thm: large deviation}.}\label{fig: schema}
\end{figure}

{\it [To prove Part (a)]:} Our proof is in three steps. In Step 1, we construct a surrogate function $g$ and prove that it satisfies all the conditions needed to invoke Proposition \ref{very important generalization proposition}. Step 2 shows that  applying Proposition \ref{very important generalization proposition} on $g$ can imply a useful inequality for $f_{\lambda_0}(\widetilde{\xbf}(\cdot),\cdot)-f_{\lambda_0}(\xbf^*,\zeta)$ as in \eqref{almost done complete p explicit}. The necessary correspondence to enable the said implication is proven in Lemma \ref{important lemma: approximation error}, which shows that the difference between $g$ and $f_{\lambda_0}(\widetilde{\xbf}(\cdot),\cdot)-f_{\lambda_0}(\xbf^*,\zeta)$ can be well controlled when their arguments are within a set that defines a high-probability event. Step 3 then simplifies inequality \eqref{almost done complete p explicit} to further obtain the desired results. \textcolor{black}{A summary of this roadmap is presented in  Figure \ref{fig: schema}.(b).}

%{\bf Step 1.} 
%which proves the desired result in Step 1.

% Consider the following function
% \begin{align} 
% f(\xbf(\cdot),\,\cdot\,):\,\Xi^N\times \Xi\rightarrow \Re_+.
% \end{align}

{\color{black}\bf Step 1.}
Define the said surrogate  function $g:\Xi^N\times \Xi\rightarrow\R$  as  
\begin{multline}
g(\boldsymbol\zeta_{1,N},\zeta):=   \min\Bigg\{2\mathcal D_{q'}\cdot (\lambda_0\mathcal D_{q'}+t),
\\  \inf_{\boldsymbol z\in \mathcal E_t^N }\left\{\mathds{1}(\zeta\in\mathcal E_t) \cdot \left[f_{\lambda_0}(\widetilde{\xbf}(\boldsymbol z),\zeta)+d_{\gamma}\left(\boldsymbol\zeta_{1,N},\boldsymbol z\right)+\mathcal D_{q'}\cdot (\lambda_0\mathcal D_{q'}+t)-f_{\lambda_0}(\xbf^*,\zeta) \vphantom{V^{V^V}_{V_V}}\right]\right\}\Bigg\},\label{constructed new function}
\end{multline}
with the convention that infimum over $\emptyset$ is $+\infty$. Note that $\Xi^N\times\Xi$ is a standard Borel space and $\mathcal E_t^N\in\mathcal B(\Xi^N)$.
Define
$
\mathcal H(\boldsymbol\zeta_{1,N},\zeta,\boldsymbol z)
:=\mathds{1}(\zeta\in \mathcal E_t)\Big(f_{\lambda_0}(\widetilde{\xbf}(\boldsymbol z),\zeta)
+d_\gamma(\boldsymbol\zeta_{1,N},\boldsymbol z)
+\mathcal D_{q'}(\lambda_0\mathcal D_{q'}+t)
-f_{\lambda_0}(\xbf^*,\zeta)\Big).$
Because $\widetilde{\xbf}:\Xi^N\to\X$ is $\mathcal B(\Xi^N)$-measurable,
$f_{\lambda_0}$ is $\mathcal B(B)\otimes\mathcal B(\Xi)$-measurable as well as $d_\gamma$ is Borel measurable on $\Xi^N\times\Xi^N$,
the mapping $\mathcal H$ is Borel measurable on $\Xi^N\times\Xi\times\Xi^N$.
Therefore,   $(\boldsymbol\zeta_{1,N},\zeta)\ \mapsto\ \inf_{\boldsymbol z\in\mathcal E_t^N} \mathcal H(\boldsymbol\zeta_{1,N},\zeta,\boldsymbol z)$ is universally measurable on $\Xi^N\times\Xi$ \citep[in view of Corollary 7.42.1 and Proposition 7.47 by][]{bertsekas1996stochastic}. 
Consequently, $g$ is measurable w.r.t.\ the $P^{N+1}$-completion of  $\mathcal B(\Xi^N\times \Xi)$ and $g(\boldsymbol\xi_{1,N},\xi)$ is measurable w.r.t.\ the $\Prob$-completion of   $\sigma(\boldsymbol\xi_{1,N},\xi)$.
%;because $(\Omega,\mathcal F,\Prob)$ is complete, it follows that $g(\boldsymbol\xi_{1,N},\xi)$ is $\mathcal F$-measurable.
Below, we would like to show that Proposition \ref{very important generalization proposition} applies to $g$, using the result from Step 1. To that end, we will verify that $g$ satisfies the conditions for the said proposition. %Note here   $g$ is universally measurable;
%since $(\Omega,\mathcal F,\Prob)$ is complete, $g(\boldsymbol\xi_{1,N},\xi)$
%is $\mathcal F$-measurable.

{\bf Step 1.1.} We would like to first show that the following inequality holds for every $j=1,...,N$ and for all $(\zeta_1,...,\zeta_N,\zeta'_1,...,\zeta_N')\in\Xi^{2N},\,\zeta\in\Xi$:
\begin{align}
\vert g(\boldsymbol\zeta_{1,N},\zeta)-g(\boldsymbol\zeta^{(j)}_{1,N},\zeta)\vert\leq \gamma:=\frac{4\cdot (t+\lambda_0 \mathcal D_{q'})^2}{N\cdot \lambda_0\cdot (q'-1)},\label{redefine gamma}
\end{align}
where  we recall the definition of  $\boldsymbol\zeta^{(j)}_{1,N}$  (as well as its relations with $\boldsymbol\zeta_{1,N}$ and $\zeta_j'$)   in \eqref{special notation}.
To that end,   we observe that, for all $(\boldsymbol\zeta_{1,N},\zeta_j',\zeta)\in\Xi^N\times \Xi\times \mathcal E_t,$ and every $j=1,...,N,$ 
\begin{align}
&\bigg\vert\inf_{\boldsymbol z\in\mathcal E_t^N }\left\{f_{\lambda_0}(\widetilde\xbf(\boldsymbol z),\zeta)+\mathcal D_{q'}\cdot (\lambda_0\mathcal D_{q'}+t)-f_{\lambda_0}(\xbf^*,\zeta)+d_{\gamma}(\boldsymbol\zeta_{1,N},\boldsymbol z)\right\}\nonumber
\\
&~~~~-\inf_{\boldsymbol z\in\mathcal E_t^N }\left\{\vphantom{V^{V^V}_{V_V}}f_{\lambda_0}(\widetilde\xbf(\boldsymbol z),\zeta)+\mathcal D_{q'}\cdot (\lambda_0\mathcal D_{q'}+t)-f_{\lambda_0}(\xbf^*,\zeta)+d_{\gamma}(\boldsymbol\zeta^{(j)}_{1,N},\boldsymbol z)\right\}\bigg\vert \nonumber
\\\leq &\sup_{\zbf\in\mathcal E_t^N}\left\vert \left\{\vphantom{V^{V^V}_{V_V}}f_{\lambda_0}(\widetilde\xbf(\boldsymbol z),\zeta)+\mathcal D_{q'}\cdot (\lambda_0\mathcal D_{q'}+t)-f_{\lambda_0}(\xbf^*,\zeta)+d_{\gamma}(\boldsymbol\zeta_{1,N},\boldsymbol z)\right\}\right.\nonumber
\\
&~~~~\left.- \left\{f_{\lambda_0}(\widetilde\xbf(\boldsymbol z),\zeta)+\mathcal D_{q'}\cdot (\lambda_0\mathcal D_{q'}+t)-f_{\lambda_0}(\xbf^*,\zeta)+d_{\gamma}(\boldsymbol\zeta^{(j)}_{1,N},\boldsymbol z)\right\}\right\vert \nonumber
\\=&\sup_{\zbf\in\mathcal E_t^N}\left\vert   d_{\gamma}(\boldsymbol\zeta_{1,N},\boldsymbol z)-d_{\gamma}(\boldsymbol\zeta^{(j)}_{1,N},\boldsymbol z) \right\vert\leq\gamma,\nonumber
\end{align}
where the last inequality is due to the triangular inequality satisfied by Hamming distance. This then further leads to
\begin{multline}
\bigg\vert \inf_{\boldsymbol z\in\mathcal E_t^N }\Big\{\mathds{1}(\zeta\in\mathcal E_t)\cdot\left[f_{\lambda_0}(\widetilde\xbf(\boldsymbol z),\zeta)+\mathcal D_{q'}\cdot (\lambda_0\mathcal D_{q'}+t)-f_{\lambda_0}(\xbf^*,\zeta)+d_{\gamma}(\boldsymbol\zeta_{1,N},\boldsymbol z)\right]\Big\}
\\
 -\inf_{\boldsymbol z\in\mathcal E_t^N }\bigg\{\mathds{1}(\zeta\in\mathcal E_t)\cdot\left[f_{\lambda_0}(\widetilde\xbf(\boldsymbol z),\zeta)+\mathcal D_{q'}\cdot (\lambda_0\mathcal D_{q'}+t)-f_{\lambda_0}(\xbf^*,\zeta)+d_{\gamma}(\boldsymbol\zeta^{(j)}_{1,N},\boldsymbol z)\right]\bigg\}\bigg\vert\leq\gamma,\label{explicit bound for g}
\end{multline}
for every $(\boldsymbol\zeta_{1,N},\zeta_j',\zeta)\in\Xi^N\times \Xi\times \Xi$ and all $j=1,...,N.$

{\bf Step 1.2.}
 We would like to secondly  show that the range of $g$  obey $g(\cdot,\cdot)\in[0,\,2\mathcal D_{q'} (t+\lambda_0\mathcal D_{q'})]$.
 To that end, we first observe that the upper bound to the range is evident by the definition of $g$. To prove the lower bound (namely, non-negativity of $g$), we   observe that, for any $\xbf\in \X$,   when $\zeta\in\mathcal E_t$, then Assumption \ref{assumption f lipschitz}.(a) combined with $q'\leq q$ implies that $\vert f_{\lambda_0}(\xbf,\zeta)-f_{\lambda_0}(\xbf^*,\zeta)\vert \leq (t+\lambda_0\mathcal D_{q'})\cdot \Vert\xbf-\xbf^*\Vert_{q'}\leq (t+\lambda_0\mathcal D_{q'})\cdot\mathcal D_{q'}.$ Thus, $0\leq f_{\lambda_0}(\xbf,\zeta)+(t+\lambda_0\mathcal D_{q'})\cdot \mathcal D_{q'}-f_{\lambda_0}(\xbf^*,\zeta)$ and, consequently, $g(\cdot,\zeta)\geq 0$. Meanwhile, in the case where $\zeta\notin\mathcal E_t$, we have $g(\cdot,\zeta)=0$. Combining both cases, we have established that the range of $g(\cdot,\cdot)$ meets the desired criteria.   Steps 1.1 and 1.2 together have now verified all the conditions needed to apply Proposition \ref{very important generalization proposition}. Comparing \eqref{explicit bound for g} with the definition of $g$ (as   in \eqref{constructed new function}), we know that  the desired relationship in \eqref{redefine gamma} holds.

{\color{black}\bf Step 2.} We are now ready to apply Proposition \ref{very important generalization proposition} on $g$. %\textcolor{black}{This should be directed to EC.33}. 
Step  1 has verified   the conditions for the said proposition, which then immediately implies that 
\begin{align}
 \Prob\left[\left\vert \E_{\xi}[g(\boldsymbol\xi_{1,N},\xi)]-\frac{1}{N}\sum_{j=1}^N  g(\boldsymbol\xi_{1,N},\xi_j)\right\vert\geq  c \Delta_\gamma(t)\right]\leq \frac{\beta}{2},\label{use of proposition}
\end{align}
where $c >0$ is some universal constant, $\boldsymbol\xi_{1,N}=(\xi_1,...,\xi_N)\in\Xi^N$ collects $N$-many i.i.d. copies of $\xi$,  $\gamma$ is defined as in \eqref{redefine gamma}, and $\Delta_\gamma(t)$    is a short-hand notation defined as
\begin{align} 
\Delta_\gamma(t):= \gamma \ln N  \cdot \ln \frac{2N}{\beta}+\sqrt{\frac{4\mathcal D_{q'}^2(\lambda_0\mathcal D_{q'}+t)^2}{N}\cdot \ln(2/\beta)}.\label{eq: define delta}
\end{align}  
%This combined with Lemma \ref{important lemma: light-tailed} leads to the following for 
Let  $\varkappa_\delta:=\frac{4\delta\cdot (t+\lambda_0\mathcal D_{q'})}{\lambda_0\cdot (q'-1)}$ and recall the notation
$\xhb:=\widetilde{\xbf}(\boldsymbol\xi_{1,N})$. We first observe that
\begin{align}
&\Prob\Big[\E_{\xi}\left[\left\{f_{\lambda_0}(\xhb,\xi)-f_{\lambda_0}(\xbf^*,\xi)\right\}\cdot \mathds{1}(\xi\in\mathcal E_t)\right]  \leq c\cdot\Delta_\gamma(t)+ \varkappa_\delta+\delta\cdot  \mathcal D_{q'}\Big] \nonumber
\\\geq&\Prob\Big[  \E_{\xi}\left[\left\{f_{\lambda_0}(\xhb,\xi)-f_{\lambda_0}(\xbf^*,\xi)\right\}\cdot \mathds{1}(\xi\in\mathcal E_t)\right]  \leq c\cdot\Delta_\gamma(t) +\varkappa_\delta+\delta\cdot  \mathcal D_{q'} ,\,\boldsymbol\xi_{1,N}\in \mathcal E_t^N\Big] \nonumber 
\\
\geq&\Prob\left[  \E_{\xi}[\left\{f_{\lambda_0}(\xhb,\xi)-f_{\lambda_0}(\xbf^*,\xi)\right\} \mathds{1}(\xi\in\mathcal E_t)]-\sum_{j=1}^N  \frac{f_{\lambda_0}(\xhb,\xi_j)-f_{\lambda_0}(\xbf^*,\xi_j)}{N} \leq c \Delta_\gamma(t)  +\varkappa_\delta,\,\boldsymbol\xi_{1,N}\in \mathcal E_t^N\right], \label{reducing term}
\end{align}
where \eqref{reducing term} is due to \eqref{suboptimality in computing SAA}, the definition of $\mathcal D_{q'}$, as well as $q'\leq q$ (which implies that $\frac{1}{N}\sum_{j=1}^N f_{\lambda_0}(\xbf^*,\xi_j)-\frac{1}{N}\sum_{j=1}^Nf_{\lambda_0}(\xhb,\xi_j)\geq -\delta\cdot \mathcal D_{q'}$). In view of the fact that $ \boldsymbol\xi_{1,N}\in\mathcal E_t^N\Longrightarrow \xi_j\in\mathcal E_t$ for all $j=1,...,N$, we then may continue from the above to obtain:
\begin{align}
&\Prob\Big[\E_{\xi}\left[\left\{f_{\lambda_0}(\xhb,\xi)-f_{\lambda_0}(\xbf^*,\xi)\right\}\cdot \mathds{1}(\xi\in\mathcal E_t)\right]  \leq c\cdot\Delta_\gamma(t)+ \varkappa_\delta+\delta\cdot  \mathcal D_{q'}\Big] \nonumber
\\\geq  &
\Prob\left[\left\vert \E_{\xi}[[f_{\lambda_0}(\xhb,\xi)
-f_{\lambda_0}(\xbf^*,\xi)]\cdot \mathds{1}(\xi\in\mathcal E_t)]\vphantom{\frac{1}{N}\sum_{j=1}^N}\right.\right.\nonumber\\&~~~~~~\left.\left.-\frac{1}{N}\sum_{j=1}^N  [f_{\lambda_0}(\xhb,\xi_j)-f_{\lambda_0}(\xbf^*,\xi_j)]\cdot \mathds{1}(\xi_j\in\mathcal E_t)\right\vert\leq  c\cdot\Delta_\gamma(t)+\varkappa_\delta,\vphantom{\frac{1}{N}}\,\boldsymbol\xi_{1,N}\in \mathcal E_t^N\right] \label{test new 2}
\\= &
\Prob\left[\left\vert \E_{\xi}\left[\vphantom{V^{V^V}}[f_{\lambda_0}(\xhb,\xi)+\mathcal D_{q'}\cdot (\lambda_0\mathcal D_{q'}+t)-f_{\lambda_0}(\xbf^*,\xi)]\cdot \mathds{1}(\xi\in\mathcal E_t)\right]\vphantom{\frac{1}{N}}\right.\right.\nonumber
\\&
\left.\left.-\frac{1}{N}\sum_{j=1}^N  [f_{\lambda_0}(\xhb,\xi_j)+\mathcal D_{q'}\cdot (\lambda_0\mathcal D_{q'}+t)-f_{\lambda_0}(\xbf^*,\xi_j)]\cdot \mathds{1}(\xi_j\in\mathcal E_t)\right\vert\leq  c\cdot\Delta_\gamma(t)+\varkappa_\delta,\,\boldsymbol\xi_{1,N}\in \mathcal E_t^N\right]\nonumber
\\ \text{\tiny (Lemma \ref{important lemma: approximation error})}~{\geq } &\,\,\,
\Prob\left[\left\vert \E_{\xi}[g(\boldsymbol\xi_{1,N},\xi)]-\frac{1}{N}\sum_{j=1}^N  g(\boldsymbol\xi_{1,N},\xi_j)\right\vert\leq  c\cdot\Delta_\gamma(t),\,\boldsymbol\xi_{1,N}\in \mathcal E_t^N\right] \nonumber
\\\geq &1-
\Prob\left[\left\vert \E_{\xi}[g(\boldsymbol\xi_{1,N},\xi)]- \sum_{j=1}^N  \frac{g(\boldsymbol\xi_{1,N},\xi_j)}{N}\right\vert\geq  c \Delta_\gamma(t)\right]-\sum_{j=1}^N \Prob [ \xi_{j}\notin\mathcal E_t]
\stackrel{\eqref{use of proposition}}{\geq} 
1-\frac{\beta}{2}-\sum_{j=1}^N \Prob [ \xi_{j}\notin\mathcal E_t],\label{almost done complete p explicit}
\end{align}
where  the first inequality in \eqref{almost done complete p explicit} is due to the combination of the union bound and the De Morgan's law.

{\bf Step 3.} Below, we will invoke several observations to simplify the probability bound in \eqref{almost done complete p explicit} into the desired results.  Firstly, we observe that
\begin{align}
&\E_{\xi}\left\{\left[f_{\lambda_0}(\xhb,\xi)-f_{\lambda_0}(\xbf^*,\xi)\right]\cdot \mathds{1}(\xi\in\mathcal E_t)\right\}
%=\E_{\xi}\left[\left.f_{\lambda_0}(\xhb,\xi)-f_{\lambda_0}(\xbf^*,\xi)\vphantom{V^{V^V}}\right\vert  \xi\in\mathcal E_t\right]\cdot  \Prob [\xi\in\mathcal E_t]\nonumber
\\=&\E_{\xi}
\left[ f_{\lambda_0}(\xhb,\xi) -f_{\lambda_0}(\xbf^*,\xi)\right] -\E_{\xi}
\left\{\left[f_{\lambda_0}(\xhb,\xi)-f_{\lambda_0}(\xbf^*,\xi)\vphantom{V^{V^V}}\right]\cdot  \mathds{1}(\xi\notin\mathcal E_t)\right\}%\cdot  \Prob [\xi\notin\mathcal E_t],
\nonumber
\\=&F_{\lambda_0}(\xhb) -F_{\lambda_0}(\xbf^*) -\E_{\xi}
\left\{\left[f_{\lambda_0}(\xhb,\xi)-f_{\lambda_0}(\xbf^*,\xi)\vphantom{V^{V^V}}\right]\cdot  \mathds{1}(\xi\notin\mathcal E_t)\right\}.\label{to be used more}
\end{align}
We may achieve further explication of the last term above  by invoking (i) Assumption \ref{assumption f lipschitz}.(a), (ii) the choice of parameter that $q'\leq q$, and (iii) the definition of $\mathcal D_{q'}$, obtaining
$\E_{\xi}
\left\{\left[f_{\lambda_0}(\xhb,\xi)-f_{\lambda_0}(\xbf^*,\xi)\vphantom{V^{V^V}}\right]\cdot  \mathds{1}(\xi\notin\mathcal E_t)\right\}\leq     \E_{\xi}
\left\{\left[\mathcal D_{q'}\cdot M(\xi)+\lambda_0\mathcal D_{q'}^2\vphantom{V^{V^V}}\right] \cdot \mathds{1}(M(\xi)> t)\right\}.$
Note here  $M(\xi)\cdot \mathds{1}(M(\xi)> t)$ is   non-negative. We may continue from the above to obtain:  
\begin{align}&\E_{\xi}
\left\{\left[f_{\lambda_0}(\xhb,\xi)-f_{\lambda_0}(\xbf^*,\xi)\vphantom{V^{V^V}}\right]\cdot  \mathds{1}(\xi\notin\mathcal E_t)\right\}\leq  \mathcal D_{q'}\cdot\mathbb E[M(\xi)\cdot   \mathds{1}(M(\xi)> t)]+\lambda_0\mathcal D_{q'}^2\cdot\Prob\left[ M(\xi)> t\right]%=\mathcal D_{q'}\cdot\int_{0}^\infty \Prob[Z(\xi)\geq s] ds\nonumber
\\=&\mathcal D_{q'}\cdot\int_{0}^\infty \Prob\left[M(\xi)\cdot   \mathds{1}(M(\xi)> t)\geq s\vphantom{V^{V^V}_{V_V}}\right] ds +\lambda_0\mathcal D_{q'}^2\cdot \Prob\left[ M(\xi)> t\right]\nonumber
\\=& \mathcal D_{q'}\cdot\int_{t}^\infty \Prob\left[M(\xi)\cdot   \mathds{1}(M(\xi)> t)\geq s\vphantom{V^{V^V}_{V_V}}\right] ds +\mathcal D_{q'}\cdot\int_{0}^t\Prob\left[M(\xi)\cdot   \mathds{1}(M(\xi)> t)\geq s\vphantom{V^{V^V}_{V_V}}\right] ds+\lambda_0\mathcal D_{q'}^2\cdot \Prob\left[ M(\xi) > t\right]\nonumber
\\\leq& \mathcal D_{q'}\cdot\int_{t}^\infty\Prob\left[M(\xi)\geq s, M(\xi)\geq t\vphantom{V^{V^V}_{V_V}}\right] ds + \mathcal D_{q'}\cdot\int_{0}^t\Prob\left[M(\xi)\geq s, M(\xi)\geq t\vphantom{V^{V^V}_{V_V}}\right] ds +\lambda_0\mathcal D_{q'}^2\cdot \Prob\left[ M(\xi) > t\right]\nonumber
\\=& \mathcal D_{q'}\int_{t}^\infty \Prob\left[M(\xi)\geq s\vphantom{V^{V^V}_{V_V}}\right] ds +\mathcal D_{q'}\int_{0}^{t}\Prob\left[ M(\xi)\geq t\right]ds+\lambda_0\mathcal D_{q'}^2\cdot\Prob\left[ M(\xi) > t\right]\nonumber
\\\leq & \mathcal D_{q'}\int_{t}^{\infty}\frac{\psi_M^p}{s^p}ds +  \mathcal D_{q'}\int_{0}^{t}\frac{\psi_M^p}{t^p}ds+\lambda_0\mathcal D_{q'}^2\cdot \frac{\psi_M^p}{t^p}=\frac{\mathcal D_{q'}}{p-1}\cdot \frac{\psi_M^p}{t^{p-1}}+\mathcal D_{q'}\frac{\psi_M^p}{t^{p-1}}+\lambda_0\mathcal D_{q'}^2\cdot \frac{\psi_M^p}{t^p}\label{to explain}
\end{align}
where the  inequality in Eq.\,\eqref{to explain} is due to  Markov's inequality under Assumption \ref{assumption f lipschitz}.(b).% with $p>2$.

% $$
% \frac{\sigma_M^p}{({N\epsilon})^{(p-1)/2}}\leq \epsilon\Longrightarrow N\geq \frac{\sigma_M^{2p/(p-1)}}{\epsilon^{2/(p-1)+1}}
% $$

Invoking (i) Equations \eqref{to be used more} and \eqref{to explain},  (ii) the choice    that $R^*\geq \max\{1,\mathcal D_{q'}^2\}$, (iii) the relationship that $\lambda_0=\frac{\epsilon}{2R^*}$, %$\cdot V_{q'}(\xbf^*)=\frac{\epsilon}{2 R^*} V_{q'}(\xbf^*)\leq \frac{\epsilon}{4}$,  
(iv) the definition of $F_{\lambda_0}$, and (v) for every $j=1,...,N$, the fact that $\Prob[\xi_j\notin \mathcal E_t]\leq \frac{\psi^p_M}{t^p} $ for any given $t>0$ (which is due to Markov's inequality under Assumption \ref{assumption f lipschitz}.(b)), we can re-organize \eqref{almost done complete p explicit}   into the below:
\begin{align}
 F(\xhb)-F(\xbf^*)-\frac{\epsilon}{4}\leq &F(\xhb)-F(\xbf^*)-\lambda_0 V_{q'}(\xbf^*)
 \leq   
 \,F_{\lambda_0}(\xhb)-F_{\lambda_0}(\xbf^*) \nonumber
\\\leq & c \Delta_\gamma(t)+\varkappa_\delta+ \mathcal D_{q'}\delta +\left(\frac{\mathcal D_{q'}p}{p-1}\cdot \frac{1}{t^{p-1}} + \frac{\epsilon}{2 t^p} \right)\cdot \psi_M^p,\label{almost the result a}
\end{align}
  with probability at least $1-\frac{\beta}{2}-\sum_{j=1}^N \Prob [ \xi_{j}\notin\mathcal E_t]\geq 1-\frac{\beta}{2}-\frac{N\psi^p_M}{t^p}.$
  Here, we recall that $\Delta_\gamma(\cdot)$ is defined as in \eqref{eq: define delta}. 

Recall that $\lambda_0 = \frac{\epsilon}{2R^*}$  {and $\gamma$ is chosen as in (\ref{redefine gamma})}. We   let $t=\left(\frac{2N}{\beta}\right)^{1/p}\psi_M$ in \eqref{almost the result a} to achieve 
\begin{align}
F(\xhb)-F(\xbf^*)\, 
  {\leq}~ &c\cdot \Delta_{\gamma}\left(\left(\frac{2N}{\beta}\right)^{1/p}\psi_M\right)+\frac{p\mathcal D_{q'}}{p-1}\cdot \psi_M\cdot \left(\frac{\beta}{2N}\right)^{1-1/p}+\frac{\epsilon\cdot \beta}{4 N}+\varkappa_\delta+\delta \cdot \mathcal D_{q'}+\epsilon/4\nonumber
\\\stackrel{\eqref{eq: define delta}}{\leq} & \frac{\widehat C_1\cdot R^*\cdot \left[\left(\frac{2N}{\beta}\right)^{1/p}\psi_M+\lambda_0 \mathcal D_{q'}\right]^2}{N\cdot  {\epsilon}\cdot (q'-1)}\cdot \ln N  \cdot \ln \frac{2N}{\beta}+(1+ \beta/N)\cdot \frac{\epsilon}{4}+\varkappa_\delta+\delta\cdot\mathcal D_{q'}\nonumber
\\&+\widehat C_1\sqrt{\frac{\mathcal D_{q'}^2\left[\lambda_0\mathcal D_{q'}+\left(\frac{2N}{\beta}\right)^{1/p}\psi_M\right]^2}{N}\cdot \ln(2/\beta)}+\frac{p\mathcal D_{q'}}{p-1}\cdot \psi_M\cdot \left(\frac{\beta}{2N}\right)^{1-1/p},\nonumber
\end{align}
% \textcolor{black}{
% $$N\geq\frac{R^*\cdot \left[\left(\frac{2N}{\beta}\right)^{1/p}\psi_M+\lambda_0 \mathcal D_{q'}\right]^2}{  {\epsilon^2}\cdot (q'-1)}\cdot \ln N  \cdot \ln \frac{2N}{\beta}$$
% $$\delta\leq\cdot \frac{\epsilon^2(q'-1)}{R^*(\frac{2N}{\beta})^{1/p}\psi_M+\epsilon \mathcal D_{q'}}$$
% $$N\geq\frac{\mathcal D_{q'}^2\left[\lambda_0\mathcal D_{q'}+\left(\frac{2N}{\beta}\right)^{1/p}\psi_M\right]^2}{\epsilon^2}\cdot \ln(2/\beta)$$
% $$\delta\leq\frac{\epsilon}{\mathcal D_{q'}}$$
% $$N\geq\beta[\frac{\epsilon(p-1)}{p\mathcal D_{q'}\psi_M}]^{-\frac{p}{p-1}}$$
% }
for some universal constant $\widehat C_1>0$, with probability at least $1-\beta$. 
% \textcolor{blue}{Let $\varkappa_{\delta}\leq\epsilon/4$ and $\delta\mathcal D_{q'}\leq\epsilon/4$, we would have $\delta\leq \frac{\epsilon^2(q'-1)}{32R^*(4N/\beta)^{1/p}\psi_M+16\epsilon \mathcal D_{q'}}$ and $\delta\leq\frac{\epsilon}{4\mathcal D_{q'}}$. Due to $p\geq2$, we could further explicate the range of $\delta$ to be $\delta\leq\min\{\frac{\epsilon^2(q'-1)}{32\sqrt{2}R^*(N/\beta)^{1/p}\psi_M+16\epsilon \mathcal D_{q'}}, \frac{\epsilon}{4\mathcal D_{q'}}\}$. Let $\frac{p\mathcal D_{q'}}{p-1}\cdot \psi_M\cdot \left(\frac{\beta}{2N}\right)^{1-1/p}\leq \epsilon/12$, we would have $N\geq\left(C_3\cdot\frac{p\mathcal D_{q'}}{p-1}\cdot\frac{\psi_M}{\epsilon}\right)^{\frac{p}{p-1}}\beta$. Let each of the remaining terms less than or equal to $\epsilon/12$, we would have\\ $N\geq \left(C_3\cdot\frac{\psi_M^{2}R^*}{\epsilon^{2}(q'-1)}\ln N\cdot\ln\frac{N}{\beta}\right)^{\frac{p}{p-2}}\cdot {\beta^{-\frac{2}{p-2}}}$ and $N\geq\left(C_3\cdot\frac{\psi_M^{2}\mathcal D_{q'}^2}{\epsilon^{2}}\ln N\cdot\ln\frac{N}{\beta}\right)^{\frac{p}{p-2}}\cdot {\beta^{-\frac{2}{p-2}}}$. With the choice of $R^*$ in \eqref{parameter choices new}, we could specify that the former implies the latter, and we could further explicate the range of $N$ to be $N\geq\left(C_3\cdot\frac{\psi_M^{2}R^*}{\epsilon^{2}(q'-1)}\ln N\cdot\ln\frac{N}{\beta}\right)^{\frac{p}{p-2}}\cdot {\beta^{-\frac{2}{p-2}}}+\left(C_3\cdot\frac{p\mathcal D_{q'}}{p-1}\cdot\frac{\psi_M}{\epsilon}\right)^{\frac{p}{p-1}}\beta$. This finishes the proof of Part (a).} 
Re-organizing the above in view of     $\epsilon\in(0,1]$, the specified choice of parameter $\lambda_0$,  {$\psi_M\geq 1$, $R^*\geq\max\{1,\mathcal D_{q'}^2\}$} and assumption that $\delta\leq 1/N$, we obtain  the desired  in Part (a).

{\it [To prove Part (b)]:} The proof argument is that for Part (a), except that the inequalities used starting from \eqref{to explain} vary due to the assumption of light tails (Assumption \ref{assumption f lipschitz v2}.(b)). More specifically, \eqref{to explain} should be modified into
\begin{align}
&\E_{\xi}
\left\{\left[f_{\lambda_0}(\xhb,\xi)-f_{\lambda_0}(\xbf^*,\xi)\vphantom{V^{V^V}}\right]\cdot  \mathds{1}(\xi\notin\mathcal E_t)\right\}  \nonumber
\\\leq&\mathcal D_{q'}\int_{t}^\infty  \Prob [M(\xi)\geq s]ds +\mathcal D_{q'}\int_{0}^{t} \Prob [M(\xi)\geq t]ds+\lambda_0\mathcal D_{q'}^2\cdot \Prob\left[ M(\xi)\geq t\right]\nonumber
\\\leq& 2\mathcal D_{q'}\int_{t}^{\infty}\exp\left(-  {s}/{\varphi}\right)ds +2\mathcal D_{q'}\int_{0}^{t}\exp\left(-  {t}/{\varphi}\right)ds + 2\exp\left(-  {t}/{\varphi}\right)\cdot \lambda_0 \mathcal D_{q'}^2\label{to explain exponential here}
\\\leq &(2\mathcal D_{q'}\varphi+\epsilon)  \cdot \exp(-t/\varphi) + 2\mathcal D_{q'}t\cdot \exp(-t/\varphi),\nonumber
\end{align}
where  \eqref{to explain exponential here} is the direct result of Assumption \ref{assumption f lipschitz v2}.(b) and the stipulation that $\lambda_0=0.5\epsilon/R^*\leq 0.5\epsilon/\mathcal D_{q'}^2$.
Then, the same argument as in obtaining \eqref{almost the result a} implies that 
$
F(\xhb)-F(\xbf^*)- {\epsilon}/{4}
\leq c\Delta_\gamma(t)+[2(\varphi +t) \cdot \mathcal D_{q'}+\epsilon]\cdot \exp(-t/\varphi)+\varkappa_\delta+\delta\cdot  \mathcal D_{q'},
$
 with probability at least $1-\frac{\beta}{2}-2N\exp(-t/\varphi)$. Here, we recall again that $\Delta_\gamma(\cdot)$ is as defined in \eqref{eq: define delta}. We may as well let $t:=\varphi\ln (4N/\beta)$ herein. 
 % \textcolor{blue}{Let $\varkappa_{\delta}\leq\epsilon/4$ and $\delta\mathcal D_{q'}\leq\epsilon/4$, we would have $\delta\leq\min\{\frac{q'-1}{ 32  \cdot \varphi\cdot \ln (4N/\beta)\cdot \varphi\cdot R^*+ 16\epsilon \mathcal D_{q'}}\cdot \epsilon^2, \frac{\epsilon}{4\mathcal D_{q'}}\}$. Let each of the rest terms to be less than or equal to $\epsilon/4$, we would have $N\geq C_4\cdot\frac{[R^*\varphi\ln{(4N/\beta)}]^2+(\epsilon\mathcal D_{q'})^2}{R^*\epsilon^2(q'-1)}\cdot (\ln{N}\ln{(4N/\beta)})$, and $N\geq C_4\cdot[\frac{\mathcal D_{q'}^4}{(R^*)^2}+\frac{[\mathcal D_{q'}\varphi\ln{(4N/\beta)}]^2}{\epsilon^2}]\ln{(2/\beta)}$. With the choice of $R^*$ in \eqref{parameter choices new}, we could specify that the former implies the latter, and we could further explicate the range of $N$ to be $N\geq C_4\cdot R^*\cdot\frac{\varphi^2}{\epsilon^2}\cdot  \frac{\ln N\cdot \ln^3  ({ N/\beta})}{q'-1}$. This finishes the proof of Part (b).} 
 After re-organizing the above based on the assumption that $R^*\geq \mathcal D_{q'}^2$, we then immediately obtain that
 $\Prob[F(\xhb)-F(\xbf^*)\leq \epsilon]\geq 1-\beta$ holds when
$$N\geq \widehat C_2\cdot\frac{[\varphi\cdot\ln(4N/\beta)]^2}{\epsilon^2}\cdot \left( \frac{R^*\cdot \ln N\cdot \ln  ({4N/\beta})}{q'-1}+   {\mathcal D_{q'}^2 } \cdot \ln(2/\beta)\right) +\widehat C_2\cdot\mathcal D_{q'}\cdot\frac{\varphi \beta}{\epsilon}\cdot\ln{(4N/\beta)}$$ and $\delta\leq 1/N$,
%Assert that $N\leq \widehat C_4\cdot\frac{(\mathcal M+\varphi)^2 R^*}{\epsilon^5}\cdot\ln^3(1/\beta)+\widehat C_4\cdot\frac{\varphi \beta}{\epsilon}$. This is possible if $\frac{\ln^3(1/\beta)}{\epsilon^3}\geq\ln^3 N$ 
%Now, we would like to prove that $N\leq \widehat C\frac{\mathcal D_{q'}^2}{(q'-1)}\frac{(\mathcal M+\psi)^2}{\epsilon^3}$
% \textcolor{black}{$$\gamma\ln N\ln{(4N/\beta)}\leq \epsilon\Rightarrow\frac{(t+\lambda_0\mathcal D_{q'})^2}{N\lambda_0(q'-1)}\ln N\ln{(4N/\beta)}\leq\epsilon\Rightarrow\frac{(\varphi\ln{(4N/\beta)}+\frac{\epsilon}{R^*}\mathcal D_{q'})^2}{N\frac{\epsilon}{R^*}(q'-1)}\ln N\ln{(4N/\beta)}\leq\epsilon$$}
% \textcolor{black}{$$\sqrt{\frac{\mathcal D_{q'}^2(\lambda_0\mathcal D_{q'}+t)^2}{N}\cdot \ln(2/\beta)}\leq \epsilon\Rightarrow N\geq\frac{\mathcal D_{q'}^2(\lambda_0\mathcal D_{q'}+t)^2}{\epsilon^2}\cdot \ln(2/\beta)\Rightarrow N\geq\frac{\mathcal D_{q'}^2(\frac{\epsilon^2}{(R^*)^2}\mathcal D_{q'}^2+[\varphi\ln{(4N/\beta)}]^2)}{\epsilon^2}\cdot \ln(2/\beta)$$}
% % \textcolor{black}{fist term: $N\geq\frac{[R^*\varphi\ln{(4N/\beta)}]^2+(\epsilon\mathcal D_{q'})^2}{R^*\epsilon^2(q'-1)}\cdot (\ln{N}\ln{(4N/\beta)})$, and $N\geq[\frac{\mathcal D_{q'}^4}{(R^*)^2}+\frac{[\mathcal D_{q'}\varphi\ln{(4N/\beta)}]^2}{\epsilon^2}]\ln{(2/\beta)}$, second term: $N\geq\frac{\beta \mathcal D_{q'}}{\epsilon}[\varphi +\varphi\ln{(4N/\beta)}]$, third term: $\delta\leq \widehat C'\cdot \frac{\epsilon^2(q'-1)}{R^*\varphi\ln{(4N/\beta)}+\epsilon \mathcal D_{q'}}$, fourth term: $\delta\leq\frac{\epsilon}{\mathcal D_{q'}}$}
for some universal constant $\widehat C_2>0$. This  can then be further simplified into the  desired for Part (b) by further noting  that $N\geq 3$ (thus $\ln N\geq 1$), $\varphi\geq 1$, $R^*\geq \max\{1,\mathcal D_{q'}^2\}$, $\beta\in(0,\,1)$, and $\epsilon\in(0,\,1]$. %\textcolor{black}{why theres no square root in this equation}. 

{\it [To show Part (c)]:} The proof is closely similar to that for Part (b), except that \eqref{to explain exponential here} should be modified into:
$\E_{\xi}
\left[\left.f_{\lambda_0}(\xhb,\xi)-f_{\lambda_0}(\xbf^*,\xi)\vphantom{V^{V^V}}\right\vert  \xi\notin\mathcal E_t\right]\cdot  \Prob [\xi\notin\mathcal E_t] 
=\mathcal D_{q'}\cdot\int_{t}^\infty  \Prob [M(\xi)\geq s]ds +\mathcal D_{q'}\cdot\int_{0}^t  \Prob [M(\xi)\geq t]ds+\lambda_0\mathcal D_{q'}^2\cdot \Prob\left[ M(\xi)\geq t\right]
\leq2\mathcal D_{q'}\cdot\left[\int_{t}^{\infty}\exp\left(- {s^2}/{\varphi^2}\right)ds + \int_{0}^{t}\exp\left(- {t^2}/{\varphi^2}\right)ds\right]+2\lambda_0 \mathcal D_{q'}^2\cdot \exp(-t^2/\varphi^2).$
Observe that, by some simple algebra,
$
\int_{t}^{\infty}\exp\left(-  {s^2}/{\varphi^2}\right)ds=0.5\cdot {\varphi\sqrt{\pi}}  \cdot \textbf{erfc}\left( {t}/{\varphi}\right),$
where $\textbf{erfc}(\cdot )$ is the complementary error function associated with a standard Gaussian distribution. By the well known inequality that $\textbf{erfc}\left(\frac{t}{\varphi}\right)\leq \exp(-t^2/\varphi^2)$ as per, e.g., \cite{chiani2003new}, we   can utilize the above in the same argument as in  Part (b) to obtain
$
F(\xhb)-F(\xbf^*)- {\epsilon}/{4}
\leq c\Delta_\gamma(t)+[(\varphi \sqrt{\pi} +2t)  \cdot\mathcal D_{q'}+\epsilon]\cdot \exp(-t^2/\varphi^2)+\varkappa_\delta+\delta\cdot  \mathcal D_{q'},$
with probability at least $1-\frac{\beta}{2}-2N \exp(-t^2/\varphi^2)$. If we let $t=\varphi\sqrt{\ln(4N/\beta)}$, we can obtain the desired result in Part (c) after re-organization 
and simplification based on the assumption that $\delta\leq 1/N$, $N\geq 3$ (such that $\ln N\geq 1$), $\varphi\geq 1$, $R^*\geq \max\{1,\mathcal D_{q'}^2\}$, $\beta\in(0,\,1)$, and $\epsilon\in(0,\,1]$.
\end{proof}

\smallskip 

The results in Proposition \ref{thm: large deviation} are implicit; the sample requirement is   that  $N$ should be greater than some poly-logarithmic function of $N$ itself. While  we would like to argue that these  are informative enough to assess the sample efficiency of SAA, we   present some more explicit complexity bounds in the theorem below.

\begin{theorem}\label{Most explicit result here}  
Let $\xhb$ be a $(\delta,q')$-approximate solution to SAA \eqref{Eq: SAA-ell2} and   $\Phi:= \frac{\varphi R^*e}{ (q'-1)\cdot \epsilon}$, where   $e$ is the base of natural logarithm. The following hold  for any given $\epsilon\in(0,1]$ and $\beta\in(0,1)$:%,  given any $\epsilon\in(0,1]$ and $\beta\in(0,1)$
\begin{enumerate}
\item[(a).]
Under the same set of assumptions as in Proposition \ref{thm: large deviation}.(b),  for some universal constant $C_6>0$, %\textcolor{black}{$C_6$}
\begin{align*}
\Prob[F(\xhb)-F(\xbf^*)\leq \epsilon]\geq 1-\beta,
 ~~~ \text{if}~~
N\geq  \frac{C_6\cdot R^*\cdot\varphi^2}{(q'-1)\cdot \epsilon^2} \cdot    \ln \left(\Phi\ln\frac{e}{\beta}\right)\cdot\ln^3\frac{\Phi}{\beta}.
\end{align*}
\item[(b).]  Under the same set of assumptions as in Proposition \ref{thm: large deviation}.(c),   for some universal constant $C_7>0$,
\begin{align*}
\Prob[F(\xhb)-F(\xbf^*)\leq \epsilon]\geq 1-\beta,~ \text{if}~
N\geq \frac{C_7\cdot R^*\cdot\varphi^2}{(q'-1)\cdot \epsilon^2} \cdot  \ln \left(\Phi\ln\frac{e}{\beta}\right)\cdot\ln^2\frac{ \Phi}{\beta}.
\end{align*}
%where $\Phi:= \frac{C_6\varphi R^*}{ (q'-1)\cdot \epsilon}$ and $e$ is the base of natural logarithm,  given any $\epsilon\in(0,1]$ and $\beta\in(0,1)$.
\end{enumerate}
\end{theorem}
\begin{proof} 
This current theorem is to explicate the  sample requirement in Proposition \ref{thm: large deviation}. More specifically, the said proposition requires that the sample size $N$ should be greater than some poly-logarithmic function of $N$, leading to implicit sample complexity estimates. In achieving explication,  both parts of this   theorem are proven in two steps. In Step 1, we   find ``coarse'' sample complexity bounds in the following form: whenever $N\geq \widetilde{\mathrm{N}}$, for some $\widetilde{\mathrm{N}}$ independent from $N$,  SAA \eqref{Eq: SAA-ell2} provably achieves the desired efficacy. Then, in Step 2, we   show that the  sample complexity bounds in Proposition \ref{thm: large deviation} can be reduced to requiring $N$ to be greater than some poly-logarithmic functions of $\widetilde{\mathrm{N}}$ (substituting ``poly-logarithmic of $N$'' by ``poly-logarithmic of $\widetilde{\mathrm{N}}$'' in the original error bounds from the said proposition).  The claimed results then immediately follow.

%Our proof employs the following notations: For any $\xbf\in\X$, we let  $f_{\lambda_0}(\xbf,\xi):=f(\xbf,\xi)+\lambda_0 V_{q'}(\xbf)$, define $F_{\lambda_0,N}(\xbf):=\frac{1}{N}\sum_{j=1}^N f_{\lambda_0}(\xbf,\xi_j)$, and denote $F_{\lambda_0}(\xbf):=F(\xbf)+\lambda_0 V_{q'}(\xbf)$, throughout this proof. Also let $
%\xbf^{*}_{\lambda_0}\in\underset{\xbf\in\X}{\arg\min}\, F_{\lambda_0}(\xbf).$

{\it [To show Part (a)]:} Our proof has two steps as mentioned: Step 1   derives a coarse sample complexity bound; then, Step 2 refines this coarse bound to obtain the desired result.

{\bf Step 1.}  
To derive the said coarse bound, we make use of a function $h:(0,\infty)\rightarrow \R$ defined as
$h(z) := z - A(\ln z)^\alpha - B\ln z$ given $\alpha\in\{3,4\}$, $A>0$ and $B>0$.  
Let
$
N_0:= \widetilde C_1\cdot (A+B+e)\cdot \left[\ln\!\big(\widetilde C_1(A+B+e)\big)\right]^\alpha,
$
where $\widetilde C_1>0$ is a universal constant. In this step, we would like to  show that $h$ is   increasing  and that $h(N_0)\geq 0$, such that $h(N)\geq 0$ whenever $N\geq N_0$. We first show that $h$ is increasing on $[N_0,\infty)$. To that end, we would like to show that
$
h'(z)=1-\frac{\alpha A(\ln z)^{\alpha-1}+B}{z}\geq 0
$
for all $z\geq N_0$.
 Since $\widetilde C_1$ can be selected such that  $N_0\ge e^\alpha\ge e^{\alpha-1}$, the function
$h'$ is increasing on $[N_0,\infty)$, and so for any $z\ge N_0$, we have
$
1-\frac{\alpha A(\ln z)^{\alpha-1}+B}{z}
\ge
1-\frac{\alpha A(\ln N_0)^{\alpha-1}+B}{N_0}.
$
Also,  by the definition of $N_0$, because $\ln(\widetilde C_1(A+B+e))\ge 1$ and 
$\ln\big(\ln(\widetilde C_1(A+B+e))\big)\le \ln(\widetilde C_1(A+B+e))$, we have
$
\ln N_0
=
\ln\!\big(\widetilde C_1(A+B+e)\big)
+\alpha \ln\!\Big(\ln\!\big(\widetilde C_1(A+B+e)\big)\Big) \le
(\alpha+1)\ln\!\big(\widetilde C_1(A+B+e)\big).
$ Therefore,
\begin{align*}
\frac{\alpha A(\ln N_0)^{\alpha-1}+B}{N_0}
&\le
\frac{\alpha A(\alpha+1)^{\alpha-1}\left[\ln(\widetilde C_1(A+B+e))\right]^{\alpha-1}+B}
{\widetilde C_1(A+B+e)\left[\ln(\widetilde C_1(A+B+e))\right]^\alpha}
\le
\frac{\alpha(\alpha+1)^{\alpha-1}+1}{\widetilde C_1}
\le 1
\end{align*}
when the universal constant $\widetilde C_1$ is    large  enough.
Hence $h'(z)\ge 0$ for all $z\ge N_0$, so $h$ is increasing on $[N_0,\infty)$. Next, we verify $h(N_0)\ge 0$. Using the same bound on $\ln N_0$ as above as well as  $\ln(\widetilde C_1(A+B+e))\ge 1$ and $\alpha\ge 1$, we have
$
A(\ln N_0)^\alpha
\le
A(\alpha+1)^\alpha \left[\ln\!\big(\widetilde C_1(A+B+e)\big)\right]^\alpha\le
(\alpha+1)^\alpha (A+B+e)\left[\ln\!\big(\widetilde C_1(A+B+e)\big)\right]^\alpha$
and
$
B\ln N_0
\le
(\alpha+1)B\ln\!\big(\widetilde C_1(A+B+e)\big)
\le
(\alpha+1)(A+B+e)\left[\ln\!\big(\widetilde C_1(A+B+e)\big)\right]^\alpha.
$
%since.
Combining these two inequalities, we obtain
\[
A(\ln N_0)^\alpha + B\ln N_0
\le
[(\alpha+1)^\alpha+\alpha+1]\cdot\,(A+B+e)\left[\ln\!\big(\widetilde C_1(A+B+e)\big)\right]^\alpha
\]
%for some universal constant $\widetilde C_2>0$ (depending only on $\alpha\in\{3,4\}$, hence universal here).
Thus, by choosing $\widetilde C_1\ge [(\alpha+1)^\alpha+\alpha+1]$, we get
$h(N_0)=N_0-A(\ln N_0)^\alpha-B\ln N_0\ge 0.$
By the monotonicity shown earlier, if $N\ge N_0$, then
\[
h(N)\ge h(N_0)\ge 0
\quad\Longrightarrow\quad
N\ge A(\ln N)^\alpha+B\ln N.
\]

In the next, we use the properties of function $h$ proven above to derive the said  coarse bound on sample requirement. Note that the sample requirement that $N\geq C_4 R^*\frac{\varphi^2}{\epsilon^2}\cdot \frac{\ln N\cdot \ln^3(N/\beta)}{q'-1}$ in  Proposition \ref{thm: large deviation}.(b) is satisfied, if $N\geq 4 C_4 R^*\frac{\varphi^2}{\epsilon^2}\cdot \frac{\ln^4 N+\ln N\cdot \ln^3(1/\beta)}{q'-1}$. Letting  $\alpha=4$,  $A:= 4C_4 R^*\frac{\varphi^2}{(q'-1)\cdot\epsilon^2}$, and $B:= 4C_4 R^*\cdot\frac{\varphi^2}{\epsilon^2}\cdot  \frac{    \ln^3  ({ 1/\beta})}{q'-1}  $,  we have that there exists  a universal constant $\widetilde C_2>0$ such that
 \begin{align}N\geq   \frac{\widetilde C_2\cdot R^*\cdot \varphi^2}{(q'-1)\cdot \epsilon^2}\cdot  \ln^3(e/\beta) \cdot \left[\ln\left(\widetilde C_2\cdot \frac{R^*\cdot \varphi}{(q'-1)\cdot \epsilon}\cdot  \ln(e/\beta) \right) \right]^4\Longrightarrow N \geq 4 C_4 R^*\frac{\varphi^2}{\epsilon^2}\cdot \frac{\ln^4 N+\ln N\cdot \ln^3(1/\beta)}{q'-1},\label{sample bounds explicit here part a}
 \end{align} 
 leading to the said coarse sample bound for $
\Prob\left[F(\xhb)-F(\xbf^*)\leq  {\epsilon} \right]\geq 1-\beta$ in view of Proposition \ref{thm: large deviation}.(b).

{\bf Step 2.} Let $ \widetilde{\mathrm{N}}_1:=\frac{\widetilde C_2\cdot R^*\cdot \varphi^2}{(q'-1)\cdot \epsilon^2}\cdot  \ln^3(e/\beta) \cdot \left[\ln\left(\widetilde C_2\cdot \frac{R^*\cdot \varphi}{(q'-1)\cdot \epsilon}\cdot  \ln(e/\beta) \right) \right]^4$.  Below, we will show the desired result by considering two different cases when comparing the sample requirement in \eqref{sample bounds explicit here part a} and that in Proposition \ref{thm: large deviation}.(b): %\textcolor{black}{should it be Proposition \ref{thm: large deviation}.(b)?}

{\bf Case 1.} In the case of $\widetilde{\mathrm{N}}_1\geq  \widetilde{\mathrm{N}}_2:=C_4\cdot R^*\cdot\frac{\varphi^2\cdot\ln \widetilde{\mathrm{N}}_1\cdot \ln^3(\widetilde{\mathrm{N}}_1/\beta)}{\epsilon^2\cdot (q'-1)}$, we have that
%$\frac{\widetilde C_3\cdot R^*\cdot\varphi^2}{(q'-1)\cdot \epsilon^2} \cdot    \ln \left( \frac{ e R^*\cdot \varphi}{(q'-1)\cdot \epsilon}\ln\frac{e}{\beta}\right)\cdot\ln^3{ \frac{e R^*\cdot \varphi}{(q'-1)\cdot \epsilon\cdot\beta}}$,
% for some large enough selection of the universal constant $\widetilde C_3>0$, we have that
\begin{align}
\widetilde{\mathrm{N}}_1\geq \widetilde{\mathrm{N}}_2   = C_4\cdot R^*\cdot\frac{\varphi^2\cdot\ln \widetilde{\mathrm{N}}_1\cdot \ln^3(\widetilde{\mathrm{N}}_1/\beta)}{\epsilon^2\cdot (q'-1)}. \label{case 1 here}
\end{align}
In the same case, there are two scenarios to study as the below:

 \noindent  {\bf Scenario 1.1.} Let $\widetilde{\mathrm{N}}_1\geq N$.  Whenever $N \geq\widetilde{\mathrm{N}}_2= C_4\cdot R^*\cdot\frac{\varphi^2\cdot\ln \widetilde{\mathrm{N}}_1\cdot \ln^3(\widetilde{\mathrm{N}}_1/\beta)}{\epsilon^2\cdot (q'-1)}$, (because $\beta\in(0,\,1)$ and $x\mapsto \ln x\cdot\ln^3 (x/\beta)$ is increasing for $x\geq 1$) it must hold  that $N \geq C_4\cdot R^*\cdot\frac{\varphi^2\cdot\ln N\cdot \ln^3(N/\beta)}{\epsilon^2\cdot (q'-1)}.$
Then, based on Proposition \ref{thm: large deviation}.(b),  we have $\Prob\left[F(\xhb)-F(\xbf^*)\leq \epsilon\vphantom{V^{V^V}_{V_V}}\right]\geq1-\beta$.

 \noindent {\bf Scenario 1.2.} Alternatively, let $N>\widetilde{\mathrm{N}}_1$. Then the outcome of Step 1 as in \eqref{sample bounds explicit here part a} ensures $\Prob\left[F(\xhb)-F(\xbf^*)\leq \epsilon\vphantom{V^{V^V}_{V_V}}\right]\geq1-\beta$.

{\bf Case 2.} Consider the case where $\widetilde{\mathrm{N}}_1 <  \widetilde{\mathrm{N}}_2$. Thus, whenever $N\geq \widetilde{\mathrm{N}}_2$, it must hold that $N\geq \widetilde{\mathrm{N}}_1$. Then, the result of Step 1 as in \eqref{sample bounds explicit here part a} ensures that 
$\Prob\left[F(\xhb)-F(\xbf^*)\leq \epsilon\vphantom{V^{V^V}_{V_V}}\right]\geq1-\beta.$

Combining both Cases 1 and 2 above, we know that $N\geq \widetilde {\mathrm{N}}_2$ is always adequate to satisfy the sample requirement in Proposition \ref{thm: large deviation}.(b). In view of the fact  that  $R^*\geq 1$, $\varphi\geq 1$,  $\epsilon\in(0,\,1]$, and thus $ \widetilde{\mathrm{N}}_1\leq \left(\frac{\widetilde C_2\cdot R^*}{(q'-1)}\right)^{5} \cdot\frac{\varphi^6}{\epsilon^6}   \ln^7(e/\beta)\leq \left(\frac{\widetilde C_2\cdot R^*}{(q'-1)}\right)^{5} \cdot\frac{\varphi^6e^7}{\epsilon^6\beta^7}$. Consequently, invoking the definition of $\widetilde{\mathrm{N}}_2$ as in \eqref{case 1 here}, we obtain $\widetilde{\mathrm{N}}_2 \leq \frac{C_4\cdot R^*\cdot\varphi^2}{\epsilon^2\cdot (q'-1)}\cdot\ln\left[ \left(\frac{\widetilde C_2\cdot R^*}{(q'-1)}\right)^{5} \cdot\frac{\varphi^6}{\epsilon^6}   \ln^7(e/\beta) \right]\cdot \ln^3\left[\left(\frac{\widetilde C_2\cdot R^*}{(q'-1)}\right)^{5} \cdot\frac{\varphi^6e^7}{\epsilon^6\beta^8}\right]$.  Therefore, after   simplification, we have that the satisfaction of the sample requirement in Theorem \ref{Most explicit result here}.(a) implies that $N\geq \widetilde {\mathrm{N}}_2$, as desired in Part (a) of this theorem.
 
\smallskip

{\it [To show Part (b)]:} Again, our proofs are in two steps: Step 1 provides a coarse sample complexity bound;  Step 2  then follows the same argument as the second step of Part (a) to show the claimed result.

{\bf Step 1.} This step derives a ``coarse bound'' for the sample complexity.  To that end,  we may invoke  Step 1 of the proof for Part (a). Note that the sample requirement in  Proposition \ref{thm: large deviation}.(c) is satisfied, if $N\geq 2 C_5 R^*\frac{\varphi^2}{\epsilon^2}\cdot \frac{\ln^3 N+\ln N\cdot \ln^2(1/\beta)}{q'-1}$. Letting  $\alpha=3$,  $A:= 2C_5 R^*\frac{\varphi^2}{(q'-1)\cdot\epsilon^2}$, and $B:= 2C_5 R^*\cdot\frac{\varphi^2}{\epsilon^2}\cdot  \frac{    \ln^2  ({ 1/\beta})}{q'-1}  $  in defining function $h$ in  Step 1 of the proof for Part (a) and construct that $
N_0:= \widetilde C_3\cdot (A+B+e)\cdot \left[\ln\!\big(\widetilde C_3(A+B+e)\big)\right]^\alpha
$ for some universal constant $\widetilde C_3>0$,  we can follow the same argument as in Step 1 of Part (a) to obtain that, for some universal constant $\widetilde C_4>0$,
 \begin{align}N\geq   \frac{\widetilde C_4\cdot R^*\cdot \varphi^2}{(q'-1)\cdot \epsilon^2}\cdot  \ln^2(e/\beta) \cdot \left[\ln\left( \frac{ \widetilde C_4 R^*\cdot \varphi}{(q'-1)\cdot \epsilon}\cdot  \ln(e/\beta) \right) \right]^3\Longrightarrow N \geq 2 C_5 R^*\frac{\varphi^2}{\epsilon^2}\cdot \frac{\ln^3 N+\ln N\cdot \ln^2(1/\beta)}{q'-1},\label{sample bounds explicit here}
 \end{align} 
 leading to the second coarse sample bound for $
\Prob\left[F(\xhb)-F(\xbf^*)\leq  {\epsilon} \right]\geq 1-\beta$ in view of Proposition \ref{thm: large deviation}.(c).

{\bf Step 2.} We may now invoke the same argument as in Step 2 of the proof for Part (a), except that now we let $\widetilde{\mathrm{N}}_1:=  \frac{\widetilde C_4\cdot R^*\cdot \varphi^2}{(q'-1)\cdot \epsilon^2}\cdot  \ln^2(e/\beta) \cdot \left[\ln\left(\frac{\widetilde C_4 R^*\cdot \varphi}{(q'-1)\cdot \epsilon}\cdot  \ln(e/\beta) \right) \right]^3$ and $\widetilde{\mathrm{N}}_2:=C_5\cdot R^*\cdot \frac{\varphi^2}{\epsilon^2}\cdot   \frac{\ln \widetilde{\mathrm{N}}_1\cdot \ln^2  ({\widetilde{\mathrm{N}}_1/\beta})}{q'-1}$ therein. Also different from Part (a) of this proof in this analysis, we  invoke   Proposition \ref{thm: large deviation}.(c) instead of  Proposition \ref{thm: large deviation}.(b).  Then, we have that $N\geq \widetilde{\mathrm{N}}_2$ implies the satisfaction of sample requirement in Proposition \ref{thm: large deviation}.(c).  Since (in view of $R^*\geq 1$, $\varphi\geq 1$,   $\epsilon\in(0,\,1]$) we have  $ \widetilde{\mathrm{N}}_1\leq \left(\frac{\widetilde C_4\cdot R^*}{(q'-1)}\right)^{4} \cdot\frac{\varphi^5}{\epsilon^5}   \ln^5(e/\beta)\leq \left(\frac{\widetilde C_4\cdot R^*}{(q'-1)}\right)^{4} \cdot\frac{\varphi^5e^5}{\epsilon^5\beta^5}$,  the stipulation of $N\geq \widetilde{\mathrm{N}}_2$ is then implied by $N\geq  \frac{C_5\cdot R^*\cdot\varphi^2}{\epsilon^2\cdot (q'-1)}\cdot    {\ln \left[\left(\frac{\widetilde C_4\cdot R^*}{(q'-1)}\right)^{4} \cdot\frac{\varphi^5}{\epsilon^5}   \ln^5(e/\beta)\right]\cdot \ln^2\left[\left(\frac{\widetilde C_4\cdot R^*}{(q'-1)}\right)^{4} \cdot\frac{\varphi^5e^5}{\epsilon^5\beta^6}\right]} $, which then leads to the desired  result  for Part (b) of Theorem \ref{Most explicit result here} after simplification.
%we obtain the desired result for Part (b) of Theorem \ref{Most explicit result here} after some simplification
\end{proof}

\smallskip

Both parts of  Theorem \ref{Most explicit result here} provide  sample complexity bounds  as  further explications of   Proposition \ref{thm: large deviation}.  In particular, if we note that $R^*$ can be      comparable to  $\mathcal D_{q'}^2$, Part (b) of this theorem  confirms the promised complexity in \eqref{metric entropy free bounds intro}.    Again, the sample complexity  rates from Theorem \ref{Most explicit result here} are  completely free from any metric entropy terms. As a result, under our assumption of   light-tailed underlying distributions, the new  bounds are more appealing in terms of the dependence on dimensionality $d$ as compared to the  benchmark   \eqref{reduced rate}. In addition,  barring some other (poly-)logarithmic terms, the said improvement w.r.t.   $d$    is achieved at some arguably small   compromise: the logarithmic rate on $1/\beta$ in \eqref{reduced rate} has now become poly-logarithmic.   In further   comparison with the results of \cite{bugg2021logarithmic} for high-dimensional SP problems under the structural assumption   that the feasible region is representable by a  simplex, our bounds seem to entail comparable insensitivity to $d$ without any structural assumption alike. 

%\begin{remark}\label{test compare SA} 
\begin{remark}\label{comparison results to be added} %, such similarities in complexity bounds provide another theoretical evidence that the solution accuracy of the SAA is theoretically ensured to be comparable to the canonical SMD under similar conditions.%, when provided with the same sample size.

% under comparable, subgaussian assumptions

% Comparing Part (c) of Proposition \ref{thm: large deviation} with

%\end{remark}

In the heavy-tailed case under the conditions including (i) convexity,  (ii) Assumption \ref{assumption f lipschitz}.(a), and (iii) the boundedness of the $p'$th central moment of $[M(\xi)]^2$, the state-of-the-art benchmark  by  \cite{oliveira2023sample} is as the following: For a given $p'\geq 2$,
\begin{multline}
\Prob[F(\xhb)-F(\xbf^*)\leq \epsilon]\geq 1-\beta,
\\~~\text{if }N\geq
 O\left(\frac{\mathbf M_2\cdot   \left(\left(\gamma(\X^{*,\epsilon})\vphantom{V^{V^V}}\right)^2 + (\mathcal D^{*,\epsilon})^{2}\cdot\ln \frac{1}{\beta}\right)}{\epsilon^2} +  p'\cdot \left( \frac{\widetilde{\mathbf M}_{p'}}{\mathbf M_2 }+\frac{\widetilde \upsilon_{\xbf^*,p'}}{\upsilon_{\xbf^*}}\right)\cdot {\beta^{-2/p'}} \right).
 \label{Eq: history bound further specialized pth moment bounded}
\end{multline}
where the notations follow the same as those for \eqref{Eq: history bound further specialized}, except that $\widetilde{\mathbf M}_{p'}$ and $\widetilde \upsilon_{\xbf^*,p'}$ are the $p'$th central moments of $[M(\xi)]^2$ and $(f(\xbf^*,\xi)-F(\xbf^*))^2$, respectively.  Note that the case with a fixed $p'\geq 2$ in \eqref{Eq: history bound further specialized pth moment bounded} corresponds to our results with $p:=2p'\geq 4$ in  Assumption \ref{assumption f lipschitz}.(b).
 % In the case of $p=2$, our results in Theorems  \ref{thm: second main theorem convex optimal rate} should apply theand the comparative results have been discussed in Remark \ref{remark: compare to OT p=2} above. 
Under comparable assumptions, we first observe that 
Part (a) of Proposition \ref{thm: large deviation} is, again, more advantageous in the dependence on dimensionality $d$ due to the avoidance of metric entropy terms such as $\left(\gamma(\X^{*,\epsilon})\right)^2$, which is polynomial in $d$ in general. Meanwhile, the comparison in terms of the dependence on $\epsilon$ is more subtle, as discussed in the following: 

It is worth noting that, when $p=2p'\geq 4$ is relatively small, the benchmark result in \eqref{Eq: history bound further specialized pth moment bounded} presents more appealing dependence on $\epsilon$ --- with a rate of $O(\epsilon^{-2})$ in \eqref{Eq: history bound further specialized pth moment bounded} versus  $O(\epsilon^{-2-\frac{4}{p-2}})=O(\epsilon^{-2-\frac{2}{p'-1}})$ in Proposition \ref{thm: large deviation}.(a). Yet, when $p
\geq c\ln \frac{1}{\epsilon}+2$ for some universal constant $c>0$, that is,  $p$ is larger than some relatively small threshold, the said two rates on $\epsilon$ become comparable. We leave closing the remaining difference between those two results for   $p\in [4,\,c\ln\frac{1}{\epsilon}+2)$ to future research.
%we argue that, because both \eqref{Eq: history bound further specialized pth moment bounded} and our results apply to the same SAA formulation under the same conditions, our rates can be combined with \eqref{Eq: history bound further specialized pth moment bounded}  to certify the performance of SAA; that is, the sample requirement is the smaller between \eqref{Eq: history bound further specialized pth moment bounded} and what is explicated in Part (a) of Proposition \ref{thm: large deviation}. In this combination, our result would be mostly useful when the dimensionality $d$ is massive but the desired  accuracy $\epsilon$ is relatively less stringent --- a   regime where \eqref{Eq: history bound further specialized pth moment bounded} alone could potentially be vacuous. %Furthermore, Theorem \eqref{Eq: history bound further specialized pth moment bounded} can be combined with \eqref{Eq: history bound further specialized pth moment bounded}.
\end{remark}

{\color{black}
\begin{remark}\label{Compare with generic chaining}It is worth noting that the generic chaining-based result in \eqref{Eq: history bound further specialized pth moment bounded} as per \cite{oliveira2023sample} also applies to light-tailed scenarios; namely, when there exists some scalar $U$ such that  $\widetilde{\mathbf M}_{p'}\leq U$ and $\widetilde \upsilon_{\xbf^*,p'}\leq U$   for all $p'\geq2$. In such a case, one may optimize $p'$ to obtain the following sample complexity for light-tailed SP problems given $\beta\in(0,e^{-1}]$: It holds that $\Prob[F(\xhb)-F(\xbf^*)\leq \epsilon]\geq 1-\beta$, if 
\begin{align}
N\geq
 O\left(\frac{\mathbf M_2\cdot   \left[\left(\gamma(\X^{*,\epsilon})\vphantom{V^{V^V}}\right)^2 + (\mathcal D^{*,\epsilon})^{2}\cdot\ln \frac{1}{\beta}\right]}{\epsilon^2} +    \left( \frac{1}{\mathbf M_2 }+\frac{1}{\upsilon_{\xbf^*}}\right)\cdot U\cdot  \ln\frac{1}{\beta} \right).
 \label{Eq: history bound further specialized pth moment bounded v2}
\end{align}
% --- Polished academic ---
Generic chaining can yield sharper bounds than classical covering-number arguments \citep{vershynin2018high}. The  advantage is especially clear when the feasible set is a simplex (or admits a simplex representation). In those cases, the metric entropy term grows only logarithmically with the ambient dimension \(d\). However, as noted in Remark~\ref{remark: compare to OT p=2}, the generic-chaining complexity \(\gamma(\X^{*,\epsilon})\) typically scales as \(O\left(\sqrt{d}\,\mathcal D^{*,\epsilon}\right)\). In contrast, Theorem~\ref{Most explicit result here}(b) has a more favorable dependence on \(d\) because it avoids metric entropy terms altogether.
\end{remark}
}
{\color{black}
\begin{remark}\label{remark varphi} In general, $\varphi^2$ in Theorem~\ref{Most explicit result here} depends on the dimension $d$. Appendix~\ref{sec: lower bound} gives two constructions showing that $\varphi^2$ can scale as $O(d)$  and $O(\ln d)$, respectively.
\end{remark}}

 Under  comparable sub-Gaussian assumptions as Part (b) of Theorem \ref{Most explicit result here}, large deviations bounds for   canonical SMD \citep[e.g., Eq (2.65) of][]{nemirovski2009robust} are available as the below, after some straightforward conversion of notations:
\begin{align}
O\left(\frac{\mathcal D_{q}^2\varphi^2}{\epsilon^{2}}\cdot \ln^2\frac{1}{\beta}\right).\label{SA result}
\end{align}
To facilitate  comparison with \eqref{SA result}, we have the following corollary of  Theorem \ref{Most explicit result here}.(b). Here, we recall the notation that $G_i$ is the $i$th component of $G$ and $e$ is the base of natural logarithm. Here,  we use the fact that, for every $\xi\in\Xi$,  the subdifferential $\partial_{\xbf} f(\xbf,\xi)\neq \emptyset$ for all $\xbf\in \X$ if $f(\cdot,\xi)$ is convex on $B$. We also denote by $\partial_{x_i} f(\xbf,\xi)$ the subdifferential w.r.t.\! the $i$th component of $\xbf$.
 %implies a sample complexity bound of SAA that is almost identical to \eqref{SA result}. 

{\color{black}\begin{corollary}\label{most explicit}
Let $d\geq 8$, $q'\in(1,2]$, $\epsilon\in(0,1]$, and $\beta\in(0,1)$. Suppose that  $N\geq 3$,    $f(\cdot,\xi)$ is convex on $B$ for all $\xi\in\Xi$, and      $\X$ admits a bounded $q'$-norm diameter  $\mathcal D_{q'}\geq 1$.  Consider   a $(\delta,q')$-approximate solution  $\xhb$ to  SAA    \eqref{Eq: SAA-ell2} with $R^*=\max\{1,\mathcal D_{q'}^2\}$,  $\lambda_0=0.5 {\epsilon}/{R^*}$, and $\delta\leq N^{-1}$. 
\begin{description}
\item[(a).] If there exists some $\varphi_{f,p}>0$,  $p\in[2,\,q'/(q'-1)]$, and a Borel measurable function $\widetilde M_p:\Xi\rightarrow\R_+$ such that (i) for all $\xbf\in\X$ and every $\xi\in\Xi$, it holds that $ \Vert\gbf\Vert_{p}\leq \widetilde M_{p}(\xi)$ for some $\gbf\in\partial_{\xbf} f(\xbf,\xi)$ and  (ii) for all $t\geq 0$, $\Prob\Big[\widetilde M_{p}(\xi)\leq t\Big]\geq 1-2\exp(-t^2/\varphi_{f,p}^2)$, then
\begin{align}
\Prob[F(\xhb)-F(\xbf^*)\leq \epsilon]\geq 1-\beta,~ \text{if}~
N\geq \frac{C_8\cdot R^*\cdot\varphi_{f,p}^2}{(q'-1)\cdot \epsilon^2} \cdot  \ln \left(\Phi_1'\ln\frac{e}{\beta}\right)\cdot\ln^2\frac{\Phi_1'}{\beta} ,\label{SMD comparable light-tailed results in corollary}
\end{align}
where $\Phi_1':= \frac{e R^*}{ (q'-1)\cdot \epsilon}\cdot \varphi_{f,p}$ and  $C_8>0$ is some universal constant.
\item[(b).]
Let  $q'=1+\big(\ln d-1\big)^{-1}$. Suppose that   $\sup\{\Vert \xbf_1-\xbf_2\Vert_1:\,\xbf_1,\xbf_2\in\X\}\leq \mathcal D_1$    for some $\mathcal D_{1}\geq \max\{1,\,\mathcal D_{q'}\}$.  If  there exists some $\varphi_{f,\infty}>0$ and a Borel measurable function $\widetilde M_i:\Xi\rightarrow\R_+$ such that (i)  for all $\xbf\in\X$ and every $\xi\in\Xi$, it holds that $ \vert g_i\vert \leq \widetilde M_{i}(\xi)$ for some $g_i\in\partial_{x_i} f(\xbf,\xi)$ for every $i=1,...,d$, and (ii) $\Prob\Big[\widetilde M_i(\xi)\leq t\Big]\geq 1-2\exp(-t^2/\varphi_{f,\infty}^2)$ for every $i=1,...,d$ and all $t\geq 0$, then it holds that
\begin{align*}
\Prob[F(\xhb)-F(\xbf^*)\leq \epsilon]\geq 1-\beta,~ \text{if}~
N\geq \frac{C_9\cdot \max\{1,\mathcal D_{1}^2\}\cdot \varphi_{f,\infty}^2\cdot\ln^2 d}{ \epsilon^2} \cdot \ln \left(\Phi_2'\ln\frac{e}{\beta}\right)\cdot\ln^2\frac{\Phi_2'}{\beta},
\end{align*}
where $\Phi_2':= \frac{e \max\{1,\mathcal D_{1}\}}{   \epsilon}\cdot \varphi_{f,\infty}\cdot \ln d$ and $C_9>0$ is   some universal constant.% and $e$ is the base of natural logarithm.
\end{description}
\end{corollary}
\begin{proof}Proof is postponed to Appendix \ref{proof of corollary 1}.\end{proof}

\bigskip
Up to   some   (poly-)logarithmic terms, the complexity bound  in Part (a) of this corollary  implies a sample complexity bound of SAA that is almost identical to \eqref{SA result}. 
 This is --- in addition to what has been discussed as in Remarks \ref{remark: Comparison with SMD strongly convex} and \ref{remark: Comparison with SMD convex} --- yet another theoretical evidence that the solution accuracy of   SAA is theoretically ensured to be comparable to   canonical SMD under similar conditions. \textcolor{black}{(See Table \ref{table summary results compare with SMD} for a summary of comparisons with SMD).} Part (b) of the corollary identifies a case where the sample complexity bound is almost dimension-free, up to some $O(\ln^2 d)$ term.  
}

% \begin{remark}
% Our  settings of discussion as stated in Section \ref{sec: intro} assumes everywhere differentiability of $f(\cdot,\xi)$ for every $\xi\in\Xi$. This assumption is non-critical and can be dropped with some further analysis. Specifically, for a given $\xi\in\Xi$, the  Lipschitz continuity of $f(\cdot,\xi)$ implies its almost everywhere differentiability.  Furthermore, for any choice of $\delta>0$, we may always associate $f(\xbf,\xi)$
%  with its ``smoothed'' approximation of 
%  $f_\delta(\xbf,\xi):=\E_{\ubf}[f(\xbf+\delta\ubf,\xi)]$, where the expectation in $\E_{\ubf}$
%  is over $\ubf$, a standard Gaussian distribution on $\R^d$. According to \cite{nesterov2017random}, $f_\delta(\xbf,\xi)$
%  is everywhere differentiable, and, now that $f(\cdot,\xi)$
%  is 
% $M(\xi)$-Lipschitz, 
%  $f_\delta(\xbf,\xi)$ is also 
% $M(\xi)$-Lipschitz. Then the (unique) solution to SAA   \eqref{Eq: SAA-ell2}   can be viewed as the limiting solution to $\min_{\xbf\in\X}N^{-1}\sum_{j=1}^N f_\delta(\xbf,\xi_j)+\lambda_0 V_{q'}(\xbf)$
%   as $\delta\rightarrow 0^+$, allowing our results under differentiability of $f(\cdot,\xi)$
%  to be directly applicable to the scenarios where $f(\cdot,\xi)$
%  is 
% $M(\xi)$-Lipschitz but not necessarily everywhere differentiable. A similar argument is discussed by \cite{guigues2017non} to justify a related differentiability condition for an analysis of the SAA's confidence bounds.
% \end{remark}

\subsection{Sample complexity beyond Lipschitz conditions}\label{sec: non-Lipschitzian}
This section  is  focused on non-Lipschitzian scenarios where there is not necessarily a known upper bound on the (global) Lipschitz constants for either $F$ or (any part of) its (sub)gradient. We start with our assumptions, where we recall  that $G:\X\times\Xi\rightarrow\R^d$ is a deterministic and universally measurable function such that  $G(\xbf,\xi)\in \partial_\xbf f(\xbf,\xi)$ for all $\xbf\in\X$ and  almost every  $\xi\in\Xi$.

% Similar to Section \ref{sec: SA-comparable}, we assume throughout Section \ref{sec: non-Lipschitzian} again that $F$ is everywhere differentiable. As formalized below: %Furthermore, we also seek generalizations to several other components of our conditions. 

% \begin{assumption}\label{assumption: differentiability} Function
% $F$ is everywhere differentiable on $\X$.
% \end{assumption}
% Similar argument as in Remark \ref{remark: differentiability} would  allow the results under $F$'s everywhere differentiability to apply to some non-differentiable cases.  Some additional assumptions are discussed  below.
 
\begin{assumption}\label{SC condition constant}
 For  given $q\in[1,2]$ and $p\in[2,\infty):\,p\leq  \varrho$, where $\varrho$ is the dual exponent of $q$, there exists an optimal solution $\xbf^*$ to \eqref{Eq: SP problem statement} such that the following two conditions hold:
\begin{itemize} 
\item[(a)]
  There exists some  $\mu>0$ such that, for every $\xbf\in\X$, and  almost every $\xi\in\Xi$,
\begin{align}
f(\xbf,\xi)-f(\xbf^*,\xi)\geq \langle G(\xbf^*,\xi),\,\xbf-\xbf^*\rangle+\frac{\mu}{2}\cdot \Vert\xbf-\xbf^*\Vert^2_q.\nonumber
\end{align}
%Furthermore,    $\E[\kappa(\xi)]=0$.
\item[(b)]  Function $F$ is differentiable at $\xbf^*$ with $\E[G(\xbf^*,\xi)]=\nabla F(\xbf^*)$  and
$\big\Vert G(\xbf^*,\xi)-\E[G(\xbf^*,\xi)]\big\Vert_{L^p}\leq \psi_{p}$  for some $\psi_p\geq 1$.
\end{itemize}
\end{assumption}
\smallskip

\begin{remark}\label{remark: SC condition weaker}
Assumption \ref{SC condition constant}.(a)    non-trivially relaxes Assumption \ref{SC condition constant all}, while the latter  is the same strong convexity assumption as in some SAA literature, e.g., by \cite{milz2023sample}.   If we change $\xbf^*$ into every $\ybf\in\X$,  then Assumption \ref{SC condition constant}.(a) is  reduced to Assumption \ref{SC condition constant all} as a special case. (Additional discussions on the applicability of Assumption \ref{SC condition constant all} is   in Remark \ref{remark: SC condition}). Our assumption on the underlying randomness in Assumption \ref{SC condition constant}.(b) is also  flexible, as it only imposes   the $p$th central moment of the subgradient   $G(\xbf^*,\xi)$ --- merely at one optimal solution $\xbf^*$ --- to be finite.  Also under this assumption  as well as Assumption \ref{assumption: Variance suboptimality} below,  it is easy to verify that $\E[G(\xbf^*,\xi)]\in\partial F(\xbf^*)$.
\end{remark}

For part of our results in this section, we also consider a stronger version of Assumption \ref{SC condition constant}.(b) as the below, where we recall the definition of   the set of $\epsilon$-suboptimal solutions $\X^{*,\epsilon}$ as in \eqref{suboptimal solution set}:
%In contrast, by a close examination, one can verify that the results by \cite{oliveira2023sample} impose the condition of the bounded $p$th central moments of $\nabla f(\xbf,\xi)$ for all solutions $\xbf$ in the  set of  $\epsilon$-suboptimal solutions $\X^{*,\epsilon}$, as stated in Assumption \ref{assumption: Variance suboptimality} subsequently. Now that $\X^{*,\epsilon}$   contains all (exact) optimal solutions $\xbf^*$,  Assumption \ref{assumption: Variance suboptimality} is a stronger condition than Assumption \ref{SC condition constant}.(b) above. %Some 
%Some of our results also imposes the following condition, which resembles the counterpart assumptions by \cite{oliveira2023sample}.

\begin{assumption}\label{assumption: Variance suboptimality}
 For  some given  $p\in[2,\,\infty)$, $\epsilon>0$,  $\psi_p\geq 1$, and for all $\xbf^{*,\epsilon}\in \X^{*,\epsilon}$, it holds that  $F$ is differentiable at $\xbf^{*,\epsilon}$ with $\E[G (\xbf^{*,\epsilon},\xi)]=\nabla F(\xbf^{*,\epsilon})$ and
$\left\Vert G (\xbf^{*,\epsilon},\xi)-\E[G (\xbf^{*,\epsilon},\xi)]\right\Vert_{L^p}\leq \psi_{p}$.
%Meanwhile,  $\E[G(\xbf^{*,\epsilon})]=\nabla F(\xbf^{*,\epsilon})$ for all $\xbf^{*,\epsilon}\in \X^{*,\epsilon}$.
\end{assumption}

%Instead of Assumption \ref{assumption: Variance suboptimality}, some of our results in this section are based on Assumption \ref{assumption: Variance everywhere}, which is introduced and explained in Section \ref{sec: strong convex assumptions} above.

% \begin{remark}\label{remark: moment suboptimal}
% Compared to Assumption \ref{SC condition constant}.(b) above, Assumption \ref{assumption: Variance suboptimality}  additionally imposes here that the $p$th central moments of $\nabla f(\,\cdot\,,\xi)$ at all the $\epsilon$-suboptimal solutions to the SP problem  in \eqref{Eq: SP problem statement} are bounded from above by $\psi_p^p$.  This assumption is comparable to the conditions imposed by \cite{oliveira2023sample}  on the underlying randomness. It is worth noting that Assumption \ref{assumption: Variance suboptimality}   is comparable to imposing \eqref{moment bound here} for all $\xbf\in\X$ in more adversarial cases, such as when $f(\cdot,\xi)$ is close to a constant.

%Below, we sometimes refer to this assumption as  ``Assumption \ref{SC condition constant}.(b)  w.r.t. the $%This assumption, as a counterpart to Assumption \ref{assumption: Variance everywhere}, essentially requires that the $p$th central moment of the gradient $\nabla f(\xbf^*,\xi)$ is finite.  

 \noindent    Intuitively, this assumption means that the underlying randomness admits a bounded $p$th central moment for  all the $\epsilon$-suboptimal solutions.
 
 %It is possible to construct scenarios where this assumption is comparable to imposing \eqref{moment bound here} 

%  For some of our results, we also consider the following light-tailed counterparts to Assumption \ref{assumption: Variance suboptimality} above:
 
%  \begin{assumption}\label{assumption: Variance suboptimality sub-gaussian}
%  For  given  $p\in[2,\,\infty)$ and $\epsilon>0$, there exists a scalar $\psi_p\in\R_+$ such that, for all $\xbf^{*,\epsilon}\in \X^{*,\epsilon}$:
% \begin{align}\left\Vert \nabla f(\xbf^{*,\epsilon},\xi)-\nabla F(\xbf^{*,\epsilon})\right\Vert_{L^p}\leq \psi_{p}.\label{moment bound here}
% \end{align}
% \end{assumption}

We are now ready to present our results.% for strongly convex SP. %Our first theorem applies to scenarios even when the Lipschitz constant may be unbounded.

\begin{theorem}\label{thm: first main theorem} Let $N\geq 3$, $q\in[1,2]$, and $p\in[2,\infty):\,p\leq  \varrho$, where $\varrho$ is the dual exponent of $q$.  Suppose that    Assumption  \ref{SC condition constant} holds.
% \begin{itemize}
% \item[(a).] 
 For  a $(\delta,q)$-approximate   solution  $\xhb$ to   SAA  \eqref{Eq: SAA} with $\delta\leq \frac{1}{N}$, there exists some universal constant $C_{10}>0$ such that
\begin{align}
\E\left[ \Vert\xbf^*-\xhb\Vert_q^2\right]\leq &\, \vartheta,~~\text{if }N\geq \frac{C_{10} p }{\mu^2}\cdot \frac{\psi_{p}^2}{\vartheta};~~~\text{and}\label{solution bound first}
\\
\Prob\left[\Vert \xbf^*-\xhb\Vert_q^2\leq \vartheta\right]\geq&\, 1-\beta,~~\text{if}~~N\geq \frac{C_{10} p}{\mu^2 }\cdot\frac{\psi_{p}^2}{\vartheta}\cdot \beta^{-\frac{2}{p}},\label{second part thm 1}
\end{align}
%for the same $C$ as the above.
%\item[b.]
%\end{itemize}
for any given  $\vartheta>0$ and $\beta\in(0,1)$.

% Furthermore, if $\kappa(\xi)=0$ for almost every $\xi\in\Xi$, we then have
% \begin{align}
% \Prob\left[\Vert \xbf^*-\xhb\Vert_q^2\leq \vartheta\right]\geq 1-\beta,~~\text{if}~~N\geq \frac{C_8 p}{\mu^2 }\cdot\frac{\psi_{p}^2}{\vartheta}\cdot \beta^{-\frac{2}{p}},\label{second part thm 1}
% \end{align}
%\item[(b).] If, in addition, Assumption \ref{L-smoothness} holds w.r.t. the $q$-norm and $\kappa(\xi)=0$ for all $\xi\in\Xi$ in Assumption \ref{d}, then
%\end{itemize}
% for any  given $\vartheta>0$ and $\beta\in(0,1)$.
\end{theorem}
\begin{proof} Introduce a short-hand that $\Upsilon_j(\xbf):=G(\xbf,\xi_j)-\E[G(\xbf,\xi_j)]$. To show  \eqref{solution bound first}, we invoke  the definition of solution $\xhb$ in solving SAA \eqref{Eq: SAA} and Assumption \ref{SC condition constant}.(a) w.r.t. the $q$-norm to obtain
\begin{align}
&\delta\Vert\xbf^*-\xhb\Vert_q\geq F_N(\xhb)- F_N(\xbf^*)
\geq \left\langle N^{-1}\sum_{j=1}^N G(\xbf^*,\xi_j),\,\xhb-\xbf^* \right\rangle+\frac{\mu}{2}\cdot \Vert\xbf^*-\xhb\Vert_{q}^2,~~~a.s.\nonumber
\\=&\left\langle N^{-1}\sum_{j=1}^N \Upsilon_j(\xbf^*),\,\xhb-\xbf^* \right\rangle+\left\langle N^{-1}\sum_{j=1}^N \E[G(\xbf^*,\xi_j)],\,\xhb-\xbf^* \right\rangle+\frac{\mu}{2}\cdot \Vert\xbf^*-\xhb\Vert_{q}^2.\nonumber
\end{align}
Under Assumptions  \ref{SC condition constant}, observe that $F(\xbf)-F(\xbf^*)=\E[f(\xbf,\xi)-f(\xbf^*,\xi)]\geq  \langle \E[G(\xbf^*,\xi)],\,\xbf-\xbf^*\rangle= \langle \nabla F(\xbf^*),\,\xbf-\xbf^*\rangle $ for all $\xbf\in \X$.  
By the fact that $\xbf^*$ minimizes $F$ on $\X$,  we have   $\left\langle N^{-1}\sum_{j=1}^N \E[G(\xbf^*,\xi_j)],\,\xhb-\xbf^* \right\rangle\geq 0$.  Further in view of $\delta \Vert\xhb-\xbf^*\Vert_q\leq \frac{\delta^2}{\mu}+\frac{\mu}{4}\Vert\xhb-\xbf^*\Vert_q^2$, we   may continue from the above to obtain:
\begin{align}
\frac{\mu}{4}\cdot \Vert\xbf^*-\xhb\Vert_{q}^2\leq  -\left\langle N^{-1}\sum_{j=1}^N \Upsilon_j(\xbf^*),\,\xhb-\xbf^* \right\rangle +\frac{\delta^2}{\mu},~~~~a.s.\label{to obtain from here tail bound}
%\\+\frac{1}{2}\left(N^{-1}\sum_{j=1}^N \kappa_1(\xi_j)-\mu\right)\cdot \Vert\xhb-\xbf^*\Vert^2
%\\\leq & \frac{1}{2\alpha\cdot \mu}\Vert \nabla F_N(\xbf^*)-\nabla F(\xbf^*)\Vert_*^2+\frac{\alpha \mu}{2}\cdot \Vert  \xhb-\xbf^*\Vert^2+\Gamma.
\end{align}
Taking expectations on both sides, we have
\begin{align}
\frac{\mu}{4}\cdot \E[\Vert\xbf^*-\xhb\Vert_{q}^2]\leq  -\E\left[\left\langle N^{-1}\sum_{j=1}^N \Upsilon_j(\xbf^*),\,\xhb-\xbf^* \right\rangle\right]+\frac{\delta^2}{\mu}.\label{first to use temporary}
%\\\leq & \frac{1}{2\alpha\cdot \mu}\Vert \nabla F_N(\xbf^*)-\nabla F(\xbf^*)\Vert_*^2+\frac{\alpha \mu}{2}\cdot \Vert  \xhb-\xbf^*\Vert^2+\Gamma.
\end{align}
Now we observe the following general relationships  for any pair of $d$-dimensional random vectors $\ubf=(u_i)\in\R^d$ and $\vbf\in\R^d$. Let $\eta$ be a scalar such that $\eta^{-1}+p^{-1}=1, \eta\in [q,\,2]$ and recall that $p\geq 2,\,q\in[1,\,2]$. Then,  by H\"older's and Young's inequalities, for any given scalar $b>0$,  
\begin{align}\Big\vert\E[\langle \ubf,\vbf\rangle]\Big\vert\leq~& \Vert \ubf\Vert_{L^\eta}\cdot \Vert \vbf\Vert_{L^p}\leq  \frac{{b}}{2} \Vert \ubf\Vert_{L^\eta}^{{2}}+\frac{1}{2{b} } \Vert \vbf\Vert^2_{L^p} =\frac{{b}}{2}\left(\sum_{i=1}^d\E\left[\vert u_i\vert^\eta\right]\right)^{2/\eta}+\frac{1}{2b }\Vert \vbf\Vert^2_{L^p}\nonumber
\\\stackrel{\eta\leq 2}{\leq} \,&   \frac{b}{2} \E\left[\left(\sum_{i=1}^d\vert u_i\vert^\eta\right)^{2/\eta}\right]+\frac{1}{2b } \Vert \vbf\Vert^2_{L^p}=\frac{b}{2} \E\left[\Vert \ubf\Vert_\eta^2\right]+\frac{1}{2b } \Vert \vbf\Vert^2_{L^p}\label{Actually very useful previous}
 \stackrel{q\leq\eta}{\leq}\,  \frac{b}{2} \E\left[\Vert \ubf\Vert_q^2\right]+\frac{1}{2b } \Vert \vbf\Vert^2_{L^p}  
\end{align}
 where the  first inequality in \eqref{Actually very useful previous} is due to the fact that $(\cdot )^{2/\eta}$ is convex as  $2/\eta\geq 1$. Combining this observation with \eqref{first to use temporary}  immediately leads to:
 \begin{align}
\frac{\mu}{4}\cdot \E[\Vert\xbf^*-\xhb\Vert_{q}^2] \,{\leq} &\,\E\left[\frac{\mu}{8}\Vert \xhb-\xbf^*\Vert_{q}^2\right]+ \frac{2}{\mu}\E\left[\left\Vert N^{-1}\sum_{j=1}^N \Upsilon_j(\xbf^*) \right\Vert_{L^p}^2\right]+ \frac{\delta^2}{\mu}\nonumber
\\\Longrightarrow&\,  \E[\Vert\xbf^*-\xhb\Vert_{q}^2]\leq   \frac{16}{\mu^2}\E\left[\left\Vert N^{-1}\sum_{j=1}^N\Upsilon_j(\xbf^*) \right\Vert_{L^p}^2\right]+\frac{8\delta^2}{\mu^2}.\label{to continue from here 1}
%\\\leq & \frac{1}{2\alpha\cdot \mu}\Vert \nabla F_N(\xbf^*)-\nabla F(\xbf^*)\Vert_*^2+\frac{\alpha \mu}{2}\cdot \Vert  \xhb-\xbf^*\Vert^2+\Gamma.
\end{align}
By Lemma \ref{useful lemma 2} in Section \ref{sec: useful lemma 2}, since ${2}\leq p<\infty$, if we denote by $\Upsilon_{j,i}$ the $i$th component of $\Upsilon_j$,% we immediately have% 
\begin{align}
&\left\Vert N^{-1}\sum_{j=1}^N \Upsilon_j(\xbf^*) \right\Vert_{L^p}^2 = %\left(\sum_{i=1}^d\E[\vert \nabla_i F_N(\xbf^*)-\nabla_i F(\xbf^*)\vert^p]\right)^{2/p}=
\left(\sum_{i=1}^d\left\Vert N^{-1}\sum_{j=1}^N \Upsilon_{j,i}(\xbf^*)\right\Vert_{L^p}^{p}\vphantom{V^V_{V_{V_V}}}\right)^{2/p}\nonumber
\leq  \left( \left(\tilde C \sqrt{p\cdot N^{-1}}\right)^{p}\sum_{i=1}^d\left[\Vert \Upsilon_{1,i}(\xbf^*)\Vert_{L^p}\vphantom{V_{V_{V_V}}}\right]^p\right)^{2/p}
\leq \frac{\tilde C^2\cdot p}{N} \psi_p^2.\nonumber
\end{align}
for some universal constant $\tilde C>0$.  
We may then continue from \eqref{to continue from here 1} to obtain
$
\E[\Vert\xbf^*-\xhb\Vert_{q}^2]\leq   \frac{C  p}{\mu^2N}\psi_p^2,$
for some universal constant $C>0$, which  immediately leads to the desired result in \eqref{solution bound first} (and thus the first part of the theorem).

To show \eqref{second part thm 1}, we   may continue from \eqref{to obtain from here tail bound}. By invoking   H\"older's and Young's inequalities, we  obtain:
\begin{align}
  \Vert\xbf^*-\xhb\Vert_{q}^2\leq  \frac{16}{\mu^2}\left\Vert N^{-1}\sum_{j=1}^N\Upsilon_j(\xbf^*)\right\Vert^2_p+\frac{8\delta^2}{\mu^2},~~~~a.s.\label{to obtain from here tail bound second result}
\end{align}
Now, we can invoke  Lemma \ref{useful lemma 2 tail bound} and Assumption \ref{SC condition constant}.(b)  w.r.t. the $p$-norm to obtain, for any $t>0$, it holds that 
$\Prob[\Vert N^{-1}\sum_{j=1}^N \Upsilon_{j}(\xbf^*)\Vert_p^2\geq t]\leq \left(\widetilde C\psi_p\sqrt{\frac{p}{Nt}}\right)^{p}$
for some universal constant $\widetilde C>0$.  This combined with \eqref{to obtain from here tail bound second result} and $\delta\leq \frac{1}{N}\leq \frac{\mu\sqrt{\vartheta}}{3}$ (where the second inequality holds   given $N\geq 3$ and the satisfaction of the sample requirement in either \eqref{solution bound first} or \eqref{second part thm 1}; if $\mu\sqrt{\vartheta}\geq 1$ then $\mu\sqrt{\vartheta}\geq 1\geq  3/N$, and $\mu\sqrt{\vartheta}\leq 1$ then either one of \eqref{solution bound first} and \eqref{second part thm 1} alone leads to $\frac{1}{N}\leq \frac{\mu\sqrt{\vartheta}}{3}$) implies that
$
\Prob\left[\Vert\xbf^*-\xhb\Vert_q^2\leq  {\frac{16t}{\mu^2}}+\frac{8}{9}\cdot \vartheta \right]\geq 1-\left(\widetilde C\psi_p\sqrt{\frac{p}{Nt}}\right)^{p},$
which evidently leads to the desired result in \eqref{second part thm 1} after some simple re-organization.
\end{proof}

\smallskip

% \begin{remark}\label{compare with benchmark nonlip} This remark complements the related comments  for the case of convex SP problem in Remark \ref{comparison results to be added} earlier.   Under  $\mu$-strong convexity (a stronger version than Assumption \ref{SC condition constant}.(a)) and the assumption of finite $p$th moment of the underlying randomness, the state-of-the-art benchmark  by  \cite{oliveira2023sample} implies a strengthened variation of the sample complexity  \eqref{Eq: history bound further specialized} in application to our settings:% can be reduced to
% \begin{multline}
% \Prob[\Vert\xhb-\xbf^*\Vert^2_q\leq \vartheta]\geq 1-\beta,
% \\~~\text{if }N\geq
%  O\left(\frac{(\mathcal M^2+\varsigma_M^2)\cdot \left(d +\ln \frac{1}{\beta}\right)}{\mu^2\vartheta} +\frac{\upsilon_{\xbf^*}^2}{\mu\vartheta^2}+  p\cdot \left(\frac{\psi_{M}}{(\mathcal M^2+\varsigma_M^2)\cdot \beta}\right)^{2/p}\right),
%  \label{Eq: history bound further specialized pth moment bounded strongly convex}
% \end{multline}
% where  $\varsigma_M^2$ is  the variance of $M(\xi)$. In contrast, our result in \eqref{second part thm 1}  may present improvement over \eqref{Eq: history bound further specialized pth moment bounded strongly convex} in terms of the dependence on multiple quantities, such as $\vartheta$, $d$, and $\mathcal M$.  
% \end{remark}

\begin{remark}\label{remark: no sub gap}It is worth mentioning that  our result here only ensures a controlled $q$-norm distance from the SAA solution to a genuine optimal solution $\xbf^*$ --- at a sample complexity rate that is free from metric entropy terms again.  While it is  unknown whether suboptimality gaps can also be controlled from our results, we argue that the said distance from $\xbf^*$ is a reasonable metric of optimization performance and has been considered in related literature \citep[e.g., by][]{milz2023sample,rakhlin2011making}.
\end{remark}

{\color{black}
\begin{remark}\label{remark: dependence on dimensions} Our sample bounds in  Theorem \ref{thm: first main theorem} depend on $\psi_p^2$, which grows with $d$ in general. More specifically, suppose that, for some $p\geq 2$ and $\phi_p\geq 0$, the component-wise $p$th  central moment of $G(\xbf,\xi)$ is bounded by $\phi_p^p$ everywhere; namely, for $p\geq 2$ and    $\phi_p\geq 0$,  it holds that $\left\Vert G_i(\xbf,\xi)-\E[G_i(\xbf,\xi)]\right\Vert_{L^p}\leq \phi_p$ for all $\xbf\in\X$ and every $i=1,...,d$, where $G_i(\xbf,\xi)$ is  the $i$th entry of $G(\xbf,\xi)$.  Then, it is straightforward to obtain from \eqref{Eq: bounding dependence on d}   that $\psi_p^2\leq d^{2/p}\phi_p^2 $, which becomes dimension-independence  when it is admissible to let $p\geq c\ln d$ for some constant $c>0$.
\end{remark}
}

{\label{comment artifact 2}\color{black}When $p=2$, the complexity rate with the significance level $\beta$ becomes $O(1/\beta)$. As is shown in Proposition \ref{lower bound}.(iii)  of Appendix \ref{beta subsection lower bound}, the said dependence on $\beta$ is intrinsic to the SAA, rather than a proof artifact. 
}

Based on Theorem \ref{thm: first main theorem}, one may further obtain the results for   SAA \eqref{Eq: SAA-ell2} in  solving a convex SP problem  as below.% which is equivalent to the assumption on the underlying randomness considered by \cite{oliveira2023sample}.

\begin{theorem}\label{thm: first main theorem convex}
Let $N\geq 3$, $\epsilon>0$, and $\lambda_0:=\frac{\epsilon}{2R^*}$ for an arbitrary choice of $R^*\geq \max\{1,\, \frac{1}{2}\cdot V_{q'}(\xbf^*)\}$ for some $q'\in(1,2]$. Denote by $\xhb$ a $(\delta,q)$-approximate solution to  \eqref{Eq: SAA-ell2} with  some $q:q'\leq q\leq 2$ and $\delta\leq \frac{1}{N}$.  Under Assumption \ref{GC condition constant all} and Assumption \ref{assumption: Variance suboptimality} w.r.t. the $L^p$-norm for some $p:2\leq p\leq q'/(q'-1)$,   there exists  $\xbf^{*,\epsilon}\in \X^{*,\epsilon}$  that satisfies the following   for   any $\vartheta>0$ and   $\beta\in(0,1)$:
\begin{align}
&\E[\Vert\xbf^{*,\epsilon}-\xhb\Vert_q^2] \leq \vartheta,~~\text{if }N\geq \frac{C_{11}\cdot  p\cdot\psi_{p}^2\cdot (R^*)^2}{ (q'-1)^2\cdot   \epsilon^2\cdot \vartheta};~~~~\text{and}\label{solution bound first convex}
\\
&\Prob\left[\Vert\xbf^{*,\epsilon}-\xhb\Vert_q^2 \leq \vartheta \right]\geq 1- \beta,~~\text{if }N\geq \frac{C_{11} \cdot  p\cdot\psi_{p}^2\cdot (R^*)^2}{ (q'-1)^2\cdot   \epsilon^2\cdot \vartheta}\cdot \beta^{-\frac{2}{p}},\label{large deviation bound first convex}
%N\geq \frac{C_1 p}{\mu^2 }\cdot\frac{\psi_{p}^2}{\delta}\cdot 
\end{align}
where  $C_{11}>0$ is some universal constant.
\end{theorem}

\begin{proof} Observe that SAA  \eqref{Eq: SAA-ell2} can be viewed as the SAA \eqref{Eq: SAA} formulation to an SP problem of the below:
\begin{align}
\min_{\xbf\in\X}F(\xbf)+\lambda_0 V_{q'}(\xbf).\label{suboptimal SP}
\end{align}
Denote by $\ybf$ the optimal solution to this new SP problem. It must hold that% 
\begin{align}
F(\ybf)+\lambda_0 V_{q'}(\ybf)\leq F(\xbf^*)+\lambda_0 V_{q'}(\xbf^*)=F(\xbf^*)+\frac{\epsilon}{2R^*}\cdot V_{q'}(\xbf^*)\leq F(\xbf^*)+\epsilon,\label{suboptimal inequality bound}
\end{align}
where we recall that $\xbf^*$ denotes the optimal solution to the SP problem \eqref{Eq: SP problem statement}.
Therefore, $\ybf$ is an $\epsilon$-suboptimal   solution to the original SP problem in \eqref{Eq: SP problem statement}. As a result, Assumption \ref{assumption: Variance suboptimality} implies that Assumption \ref{SC condition constant}.(b) holds at $\ybf$, an optimal solution of \eqref{suboptimal SP}, with the same constant $\psi_p$. Here,  ``$\xbf^*$'' in Assumption \ref{SC condition constant}.(b) should be equal to  $\ybf$ for this current proof.

Meanwhile, by Assumption \ref{GC condition constant all} and the $(q'-1)$-strong convexity of $V_{q'}$ w.r.t. the $q'$-norm  in the sense of \eqref{SC V} in Appendix \ref{sec: preliminary} \citep{ben2001ordered}, we know that the objective functions of both SAA \eqref{Eq: SAA-ell2} and the SP problem in \eqref{suboptimal SP} must be $\lambda_0(q'-1)$-strongly convex w.r.t.\,the $q'$-norm. Therefore, \eqref{suboptimal SP} admits a unique solution $\ybf$ and  Assumption \ref{SC condition constant}.(a) holds with modulus $\mu:=\lambda_0(q'-1)=\frac{\epsilon}{2R^*} (q'-1)$ w.r.t. the $q'$-norm.  Again, ``$\xbf^*$'' in Assumption \ref{SC condition constant}.(a) should be equal to $\ybf$ for this current proof.%   almost everywhere on $\Xi$. 

In view of the above, we may invoke  \eqref{solution bound first} of Theorem \ref{thm: first main theorem} (where the $q$-norm therein is treated as the $q'$-norm for this current theorem), considering \eqref{suboptimal SP} as the target SP problem, and viewing SAA \eqref{Eq: SAA-ell2} as the corresponding SAA \eqref{Eq: SAA} formulation, whose   solution in the sense of \eqref{define solution} is denoted by $\xhb$. The result of Theorem \ref{thm: first main theorem} then implies that
$\E[\Vert \xhb-\ybf\Vert_q^2]\leq \vartheta,$ $\text{if~}N\geq \frac{\widehat Cp\psi^2_p}{\lambda_0^2 (q'-1)^2 \vartheta}=\frac{\widehat C'p(R^*)^2\cdot \psi^2_p}{\epsilon^2(q'-1)^2 \vartheta},$
for some universal constants $\widehat C,\,\widehat C'>0$ and any given $\vartheta>0$.
Combining this with \eqref{suboptimal inequality bound}---that is, $\ybf$ must be an $\epsilon$-suboptimal solution to the original SP problem in \eqref{Eq: SP problem statement}--- we immediately have the desired result in Eq.\,\eqref{solution bound first convex}. 

Similarly, the result    in  \eqref{large deviation bound first convex} holds as a result of \eqref{second part thm 1} from Theorem \ref{thm: first main theorem}.
\end{proof}

\smallskip

%\subsection{SA-comparable sample complexity of }

\begin{remark}Theorems \ref{thm: first main theorem} and \ref{thm: first main theorem convex} (when one select $R^*$ to be comparable to $V(\xbf^*)$)    formally state the promised sample complexity as in \eqref{strong convex case nonlip} and \eqref{general convex case nonlip} of Section \ref{sec: intro}, respectively. Again, no metric entropy term is in presence.  

Another potentially desirable feature of these complexity bounds is that they do not depend on any Lipschitz constants of $F$ and $f(\cdot,\xi)$, nor those of their (sub)gradients. This result predicts SAA to be effective even if the Lipschitz constants are undesirably large and even potentially unbounded (due to the possibly unbounded feasible region). In contrast, most existing SAA's sample complexity bounds, e.g., by \cite{shapiro2021lectures,hu2020sample,shalev2010learnability}, and \cite{oliveira2023sample}, grow polynomially with the some Lipschitz constants.   To our knowledge, the only SAA result under similar conditions is provided by \cite{milz2023sample}, whose findings imply the same error bound of Theorem \ref{thm: first main theorem} in the $2$-norm setting. Nonetheless, our analysis in Theorem \ref{thm: first main theorem} 
   presents an alternative proof and generalizes from the $2$-norm setting to more general $q$-norm ($1\leq q\leq 2$) settings. Further, Theorem \ref{thm: first main theorem convex} applies to   SP problems  not necessarily strongly convex, which are not discussed by \cite{milz2023sample}. 
   
    To  our knowledge,  there currently is limited theory for   SMD's effectiveness   when none of the  Lipschitz constants of    $F$, $f(\cdot,\xi)$, or their (sub)gradients admits a known upper bound. Comparing with  SMD, our results show  SAA's potential for better applicability  to the  SP problems in less desirable settings. 
\end{remark}

% \begin{remark}\label{comparison discussion SAA 2}
%  The error bounds in Theorem \ref{thm: first main theorem} hold for all the admissible $q$-norms simultaneously without the need to perform any additional tailoring to the SAA formulation. 
% \end{remark}

\begin{remark}\label{light tailed conversion}
It is also worth noting that both Theorems \ref{thm: first main theorem} and \ref{thm: first main theorem convex} explicate the evolution of the complexity rate w.r.t. $\beta$, as the underlying distribution gradually admits more and more bounded central moments (and thus the tail becomes lighter and lighter). Once it is admissible to let $p\geq c\ln (1/\beta)$ for some constant $c >0$, the complexity then becomes logarithmic in $1/\beta$.%, resembling the typical SAA's results in the light-tailed settings, e.g., by \cite{shapiro2021lectures}.% and by our results in Proposition \ref{thm: large deviation} [as in Parts (b) and (c)].% under typical light tails conditions,  the $p$th central moments are finite for all $p\geq 1$.
\end{remark}

% \begin{assumption}\label{everywhere smooth}
% There exist a deterministic and measurable function $L:\Xi\rightarrow\R_+$ such that, for all $\xbf,\ybf\in\X$ and almost every $\xi\in\Xi$,
% \begin{align}
% \Vert \nabla f(\xbf,\xi)-\nabla f(\ybf,\xi)\Vert_2\leq M(\xi)\cdot \Vert\xbf-\ybf\Vert_2.
% \end{align}
% Furthermore, $\E[M(\xi)]=\mathcal L>0$.
% \end{assumption}

Similar to Remark \ref{remark: no sub gap}, Theorem \ref{thm: first main theorem convex} does not provide a guarantee on the solution's  suboptimality gap. Yet, such a guarantee is, as we suspect, hardly available due to the limited regularities in the non-Lipschitzian settings of consideration.  %As we replace these assumptions into their  ``global'' counterparts and impose additional regularities on $F$, a bound on the suboptimality gap becomes available   in Theorem \ref{thm: suboptimality} below.

{\color{black}\begin{remark}\label{remark: dimension remark last} By the same argument as in Remark \ref{remark: dependence on dimensions}, the sample bounds in  Theorem \ref{thm: first main theorem convex} depend on $\psi_p^2$, which grows with $d$ in general. When, for some $p\geq 2$ and $\phi_p\geq 0$, the component-wise $p$th  central moment of $G(\xbf,\xi)$ is bounded by $\phi_p^p$ everywhere, we may show that  $\psi_p^2\leq d^{2/p}\phi_p^2 $. The dependence on $d$ diminishes when it is admissible to let $p\geq c\ln d$ for some constant $c>0$.
\end{remark}
}

\section{Numerical Experiments}\label{sec: simulated data}
%To verify our theoretical predictions numerically, we conducted two sets of experiments and their results are presented in Sections \ref{sec: simulated data} and \ref{sec: real data}, respectively.  In particular, Section \ref{sec: simulated data} presents a   a linear regression problem based on simulated data and Section \ref{sec: real data} reports efficacy of SAA in a novel application to facilitate the dose calculation process of radiotherapy treatment planning. %The second   evaluate the effectiveness of SAA, we consider the 
%\subsection{Simulation results}\label{sec: simulated data}
Numerous empirical results on SAA, as well as its comparison with SMD in the light-tailed setting, have been previously reported \citep[e.g., by][]{nemirovski2009robust,Dvinskikh03092022}. The  experiments presented in this section were designed to complement existing findings by examining  differences in performance  among the different SAA variations of discussion {\color{black} in both light-tailed and heavy-tailed scenarios as well as their comparisons with SMD in heavy-tailed settings.  \citep[Comparisons between SMD and SAA in light-tailed scenarios have been reported by][]{nemirovski2009robust}. To that end, we present two sets of numerical experiments, one on a light-tailed SP problem as   in Section \ref{sec: experiment on light-tailed problem} and one on a   heavy-tailed SP problem as  in Section \ref{sec: experiment on heavy-tailed problem}.} All experiments were implemented in Matlab R2024b and run sequentially on the HiPerGator supercomputing cluster at the University of Florida with AMD EPYC 75F3 32-Core Processor and 32 GB memory. 

\subsection{Experiment on solving a light-tailed SP problem}\label{sec: experiment on light-tailed problem}
For simulation setups, we followed  \cite{fan2014strong} and \cite{liu2019sample} in constructing a stochastic quadratic program---which is equivalent to a population-level linear regression problem---of the form:
\begin{align}
\textcolor{black}{\min_{\xbf\in\R^d} \E_{\boldsymbol a,b}\left[( \boldsymbol a^\top\xbf-b)^2\right]},\label{original test problem 1}
\end{align}
with random parameters specified as $\xi:=(\boldsymbol a,b)\in\R^{d+1}$ and governed by the
data generating process of $b = \boldsymbol a^\top\xbf^{*}+w$. Here, we let $\boldsymbol a\in\R^d$ be a centered Gaussian distributed random vector. The corresponding  covariance matrix $\Sigma=(\varsigma_{i_1,i_2})\in\R^{d\times d}$ was calculated as $\varsigma_{i_1,i_2}=0.5^{\vert i_1-i_2\vert}$. The white noise term $w$ followed the standard Gaussian distribution, independent of $\boldsymbol a$. 
We simulated an i.i.d. sample of $(\boldsymbol a, b)$, denoted by $(\boldsymbol a_j,  b_j)$ for $j=1,...,N$, with  the specific values of $d$ and $N$ to be provided subsequently. The corresponding optimal solution to this SP problem is verifiably $\xbf^*$, which was also randomly generated according to 
$
\xbf^*_{S}=1.5\ubf/\Vert \ubf\Vert_{1.8}$ and $\xbf^*_{S^c}=1.5\vbf/\Vert \vbf\Vert_{1.8}$,
where the entries of $\ubf$ and $\vbf$ were specified as  independent realizations of the standard Gaussian distribution. We let $S$ be the index set of $r$-many uniformly randomly selected dimensions of $\xbf^*$ (with $r:=\min\{d,\,200\}$) and $S^c$ be the complement of $S$. The setups above ensured that $\Vert\xbf^*\Vert_{1.8}$ should remain small in spite of increasing dimensionality, yet $\xbf^*$ was unlikely to be sparse.

We evaluated the effectiveness of SAA in the following several configurations:
\begin{description}
\item[(A). SAA$_r$:] The  SAA   as in \eqref{Eq: SAA} solved  \textcolor{black}{with high accuracy} by \textcolor{black}{the standard gradient descent algorithm} with an initial solution generated uniformly randomly on $[-0.5,\,0.5]^d$.
\item[(B). SAA$_0$:] The  SAA    as in \eqref{Eq: SAA} solved \textcolor{black}{with high accuracy} by \textcolor{black}{\label{the standard gradient}the standard gradient descent algorithm} initialized at the all-zero solution;
\item[(C). SAA-L$_{q'}$:] The SAA  as in \eqref{Eq: SAA-ell2} with $q'\in\{1.01,\,1.5,\,2\}$ solved \textcolor{black}{with high accuracy} by \textcolor{black}{the  standard gradient descent algorithm} with the same initial solution as SAA$_r$;
%\item[(D). SAA-L$_{1.5}$:] The SAA formulation as in \eqref{Eq: SAA-ell2} with $q'=1.5$ solved  with high accuracy by the standard gradient
%descent algorithm with the same initial solution as SAA$_r$;
%\item[(E). SAA-L$_2$:] The SAA formulation as in \eqref{Eq: SAA-ell2} with $q'=2$ solved with high accuracy by the standard gradient descent algorithm with the same initial solution as SAA$_r$;
\item[(D). LASSO:] The LASSO formulation \citep{tibshirani1996regression} solved with high accuracy by  the iterative soft thresholding algorithm \citep{chambolle1998nonlinear} with the same initial solution as SAA$_r$. Since LASSO is a well-studied and widely applied scheme particularly suitable for high-dimensional linear regression, we regard LASSO as a reasonable benchmark scheme. The LASSO formulation incorporates a $1$-norm sparsity-inducing penalty, whose coefficient is  also denoted by $\lambda_0$ hereafter.
\end{description}
\begin{figure}
    \centering
    \small 
    \begin{tabular}{ccc}
% --- First two rows: (a), (c), (e) --------------------------------
\includegraphics[width=0.4\textwidth]{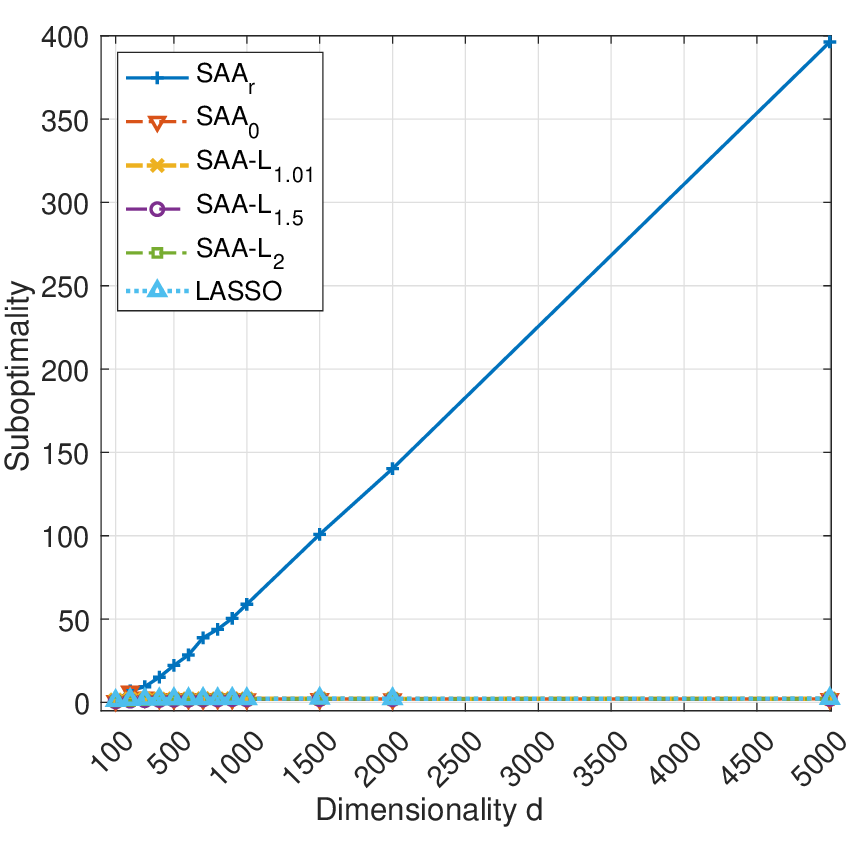} &
\includegraphics[width=0.4\textwidth]{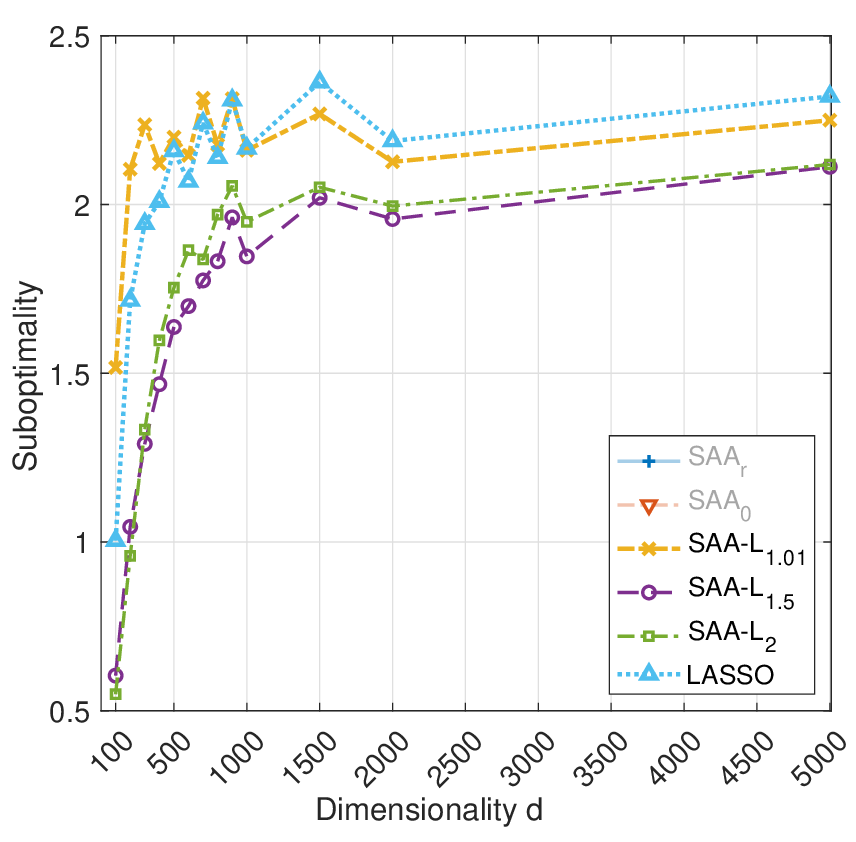}  \\
(a). &
(b).    \\
\includegraphics[width=0.4\textwidth]{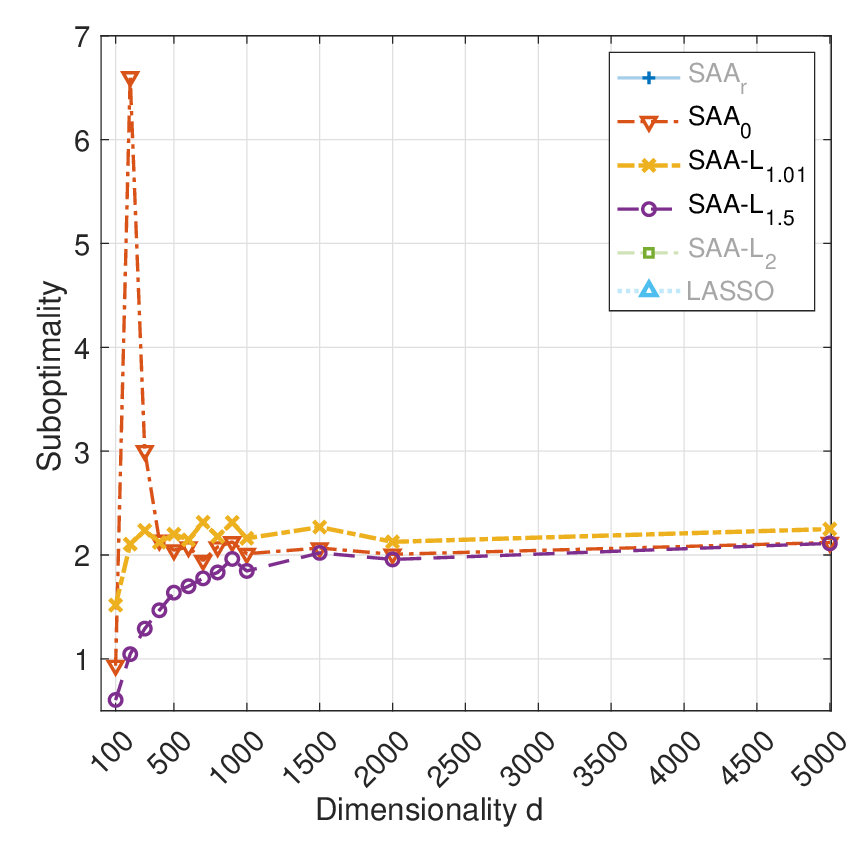} &
\includegraphics[width=0.4\textwidth]{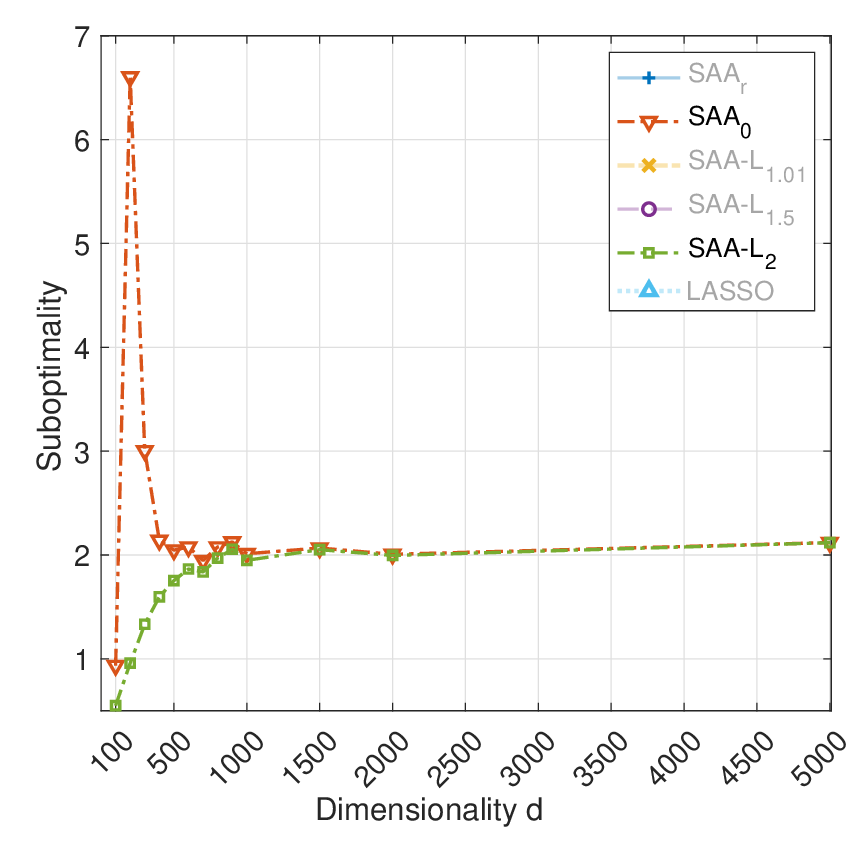}  \\
(c). &
(d).     \\
%\multicolumn{2}{c}{\includegraphics[width=0.8\textwidth]{compare_2_linear.eps}}
%\\
%\multicolumn{2}{c}{(e).}
\end{tabular}
    % \begin{tabular}{cc}
    %     \includegraphics[width=0.3\textwidth]{N200.eps} & \includegraphics[width=0.3\textwidth]{N200Zoom.eps}  \\
    %     (a). $N=200$& (b). $N=200$ (magnified)
    %     \\
    %     \includegraphics[width=0.3\textwidth]{N400.eps} & \includegraphics[width=0.3\textwidth]{N400Zoom.eps}
    %     \\
    %     (c). $N=400$& (d). $N=400$ (magnified)
    %     \\ 
    %     \includegraphics[width=0.3\textwidth]{N600.eps} & \includegraphics[width=0.3\textwidth]{N600Zoom.eps} 
    %     \\
    %     (e). $N=600$& (f). $N=600$ (magnified)\\ 
    % \end{tabular}
%\vspace{1mm}
    \caption{Average suboptimality gaps (over  five replications) achieved by different variations of SAA and their comparison with LASSO when dimensionality increases and $N=200$. Subplot (a) presents all the methods, while subplots (b)-(d) present  magnified views for different  subsets of the methods. }
    \label{fig: synthetized}
\end{figure}
\begin{figure}[H]
    \centering
    \small   \includegraphics[width=0.55\textwidth]{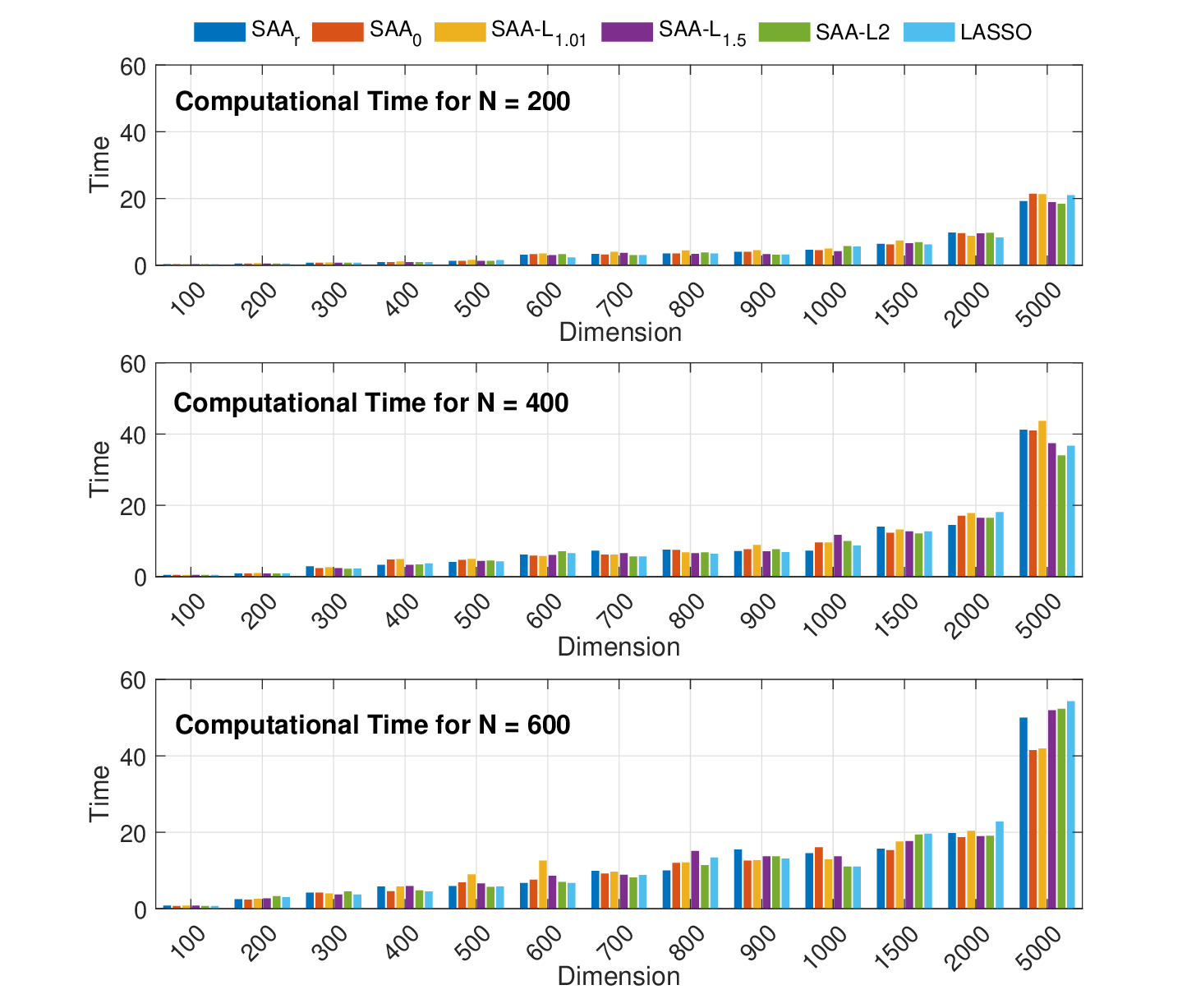}
%\vspace{1mm}
    \caption{Average  computational time (s) over  five replications incurred by different variations of SAA and their comparison with LASSO when dimensionality increases. The subplots in the top, middle, and bottom report results for $N=200$, $400$, and $600$, respectively.}
    \label{fig: synthetized bars}
\end{figure}

For all variations of SAA-L$_{q'}$ (namely, SAA \eqref{Eq: SAA-ell2}), we let $\xbf^0=\mathbf 0$. To choose the value for $\lambda_0$, we performed cross-validation. In particular, we first specified     $N=200$, $d=1000$ and simulated two  i.i.d.\,random samples of $\xi$---one was used as  the training set, and the other, the validation set.   For each of the SAA variations in (C) and (D), we constructed its formulation based on the training set, specifying   $\lambda_0$ to be one candidate value within the set \{0.01, 0.05, 0.1, 0.15, 0.2, 0.25, 0.3, 0.35, 0.4, 0.45, 0.5\}. The resulting  solution $\boldsymbol x_{\lambda_0}$ to each   SAA variation   was then evaluated on the validation set, denoted by $\widehat\xi_j$, $j=1,...,200$,  in terms of the average cost value calculated as $\frac{1}{200}\sum_{j=1}^{200}f( \boldsymbol x_{\lambda_0},\widehat\xi_j)$. Fixing   $\lambda_0$ to  each of its candidate values, we repeated the process above for five independent runs. The  $\lambda_0$-value that led an SAA variation to generate the best average validation performance out of the five repetitions was  then selected for the corresponding SAA variation. % formulation with 

With the setups above, our  experiment was to evaluate how the performance, in terms of the resulting suboptimality gaps and $\ell_2$-loss, of the aforementioned schemes  evolved as $d$ increased from 100 to 5000 for different sample sizes $N\in\{200,\,400,\,600\}$. Here, the suboptimality gap for a solution $\xbf$ can be calculated through a closed-form formula: $(\xbf-\xbf^*)^\top\Sigma(\xbf-\xbf^*)$ and the corresponding $\ell_2$-loss value is calculated as $\Vert\xbf-\xbf^*\Vert_2$. For each ($N$,\,$d$)-combination, we repeated the experiments five times.   Figure \ref{fig: synthetized} presents the resulting solution quality suboptimality gaps for $N=200$.   Meanwhile,  Figure \ref{fig: synthetized bars}  reports the corresponding computational time. Results on more scenarios (e.g., when $N=400$ and $N=600$) as well as additional details are shown in Tables S1, S2, and S3 of the electronic supplement   for the suboptimality gaps, $\ell_2$-loss values, and computational time, respectively.  % Table \ref{table: simulation time} in the electronic supplement 
 From these results, we had  several  observations   as summarized below.

\begin{itemize}
\item First, the solution quality of SAA$_r$ deteriorated   (at a roughly linear rate) as dimensionality $d$ increased. In contrast, all other methods preserved   better performance despite the growth in dimensionality. Among them, as particularly showcased in Subplots  (b) and (c) of Figure \ref{fig: synthetized}, SAA-L$_{1.5}$ consistently achieved the best performance. Meanwhile, SAA-L$_{2}$ also exhibited better performance than  most other methods.  To our analysis, this was because the configurations of the $q'$-norm  (1.5-norm and 2-norm) for the regularization term in SAA-L$_{1.5}$ and SAA-L$_{2}$  were more consistent  with the underlying problem structure that  $\Vert\xbf^*\Vert_{1.8}$ was small. Furthermore, SAA-L$_{1.5}$ and SAA-L$_{2}$ improved over the performance of LASSO, a well-studied and widely applied high-dimensional scheme.  

\item Second, as was also shown in Figure \ref{fig: synthetized}, in many higher-dimensional cases, SAA$_0$ (namely, non-regularized SAA \eqref{Eq: SAA} given an initial solution closer to $\xbf^*$) often achieved comparable performance, if not better, than   LASSO, exhibiting a drastically different pattern than the solution to SAA$_r$---the same formulation of SAA \eqref{Eq: SAA} computed with different initialization. This was an indication that the original  non-regularized SAA  \eqref{Eq: SAA} could admit solutions that are less sensitive to the growth in dimensionality if properly initialized. Nonetheless,  SAA$_0$ still generated  significantly worse performance than SAA-L$_{q'}$ for the cases such as when (i) $N=200$ and $d=200$ or (ii) $N=200$ and $d=300$. These observations suggested the potential benefits of incorporating the regularization terms of discussion into the conventional SAA method.

\item {\color{black}\label{rev double descent}Third, related to the previous point, in  subplots (c) and (d), the curves corresponding to $\text{SAA}_0$ exhibit  peaks in suboptimality when the problem dimensionality $d$ is close to the sample size $N=200$. Then, the suboptimality gap  decreases when $d$ grows beyond $N$, approaching levels comparable to or even below those observed for regularized SAA variations such as SAA-L$_2$. 
We conjecture that this pattern reflects the recently documented phenomenon of  {``double descent''} observed in various machine learning models. Specifically, double descent refers to the behavior where the excess risk of a model (captured here by the suboptimality gap) peaks when the number of features --- that is, the dimensionality $d$ --- approaches the sample size $N$, and then decreases again as $d$  continues to grow \citep{belkin2020two}. While a deeper investigation is beyond the scope of this work, we intend to explore potential connections between double descent and our sample complexity bounds in future research.}
%\item 
\item Fourth, as  shown in Figure \ref{fig: synthetized bars},  despite the variations in formulations, the computational effort in solving  {SAA-L$_{q'}$} (namely, SAA  in \eqref{Eq: SAA-ell2}) did not differ significantly from solving SAA$_0$ or SAA$_r$ (that is, \eqref{Eq: SAA}).  Indeed, none of the formulations involved in the tests were  dominantly more efficient to solve  across most problem scenarios. 

\end{itemize}

Closely similar patterns as mentioned above can be also observed from Tables S1
%\ref{table: simulation quality}
 and S2 %\ref{tab:solution-distance-means}
 for other values of $N$ (namely, $N\in\{400,600\}$) as well as under different quality metric (when $\ell_2$-loss is used to evaluate the solution quality in Table S2 %\ref{tab:solution-distance-means} 
 instead of the suboptimality gap in Table S1).%\ref{table: simulation quality}). 

{\color{black}
\subsection{Experiment on solving a heavy-tailed SP problem}\label{sec: experiment on heavy-tailed problem}

 Based on a stochastic utility problem in the experiments of  \cite{nemirovski2009robust}, we constructed a stochastic nonsmooth convex problem of the form:%with quadratic terms  which can be viewed as   penalty functions to handle box constraints $\{\xbf\in \R^d, -1\leq\xbf\leq1\}$}.
\begin{align} \label{stochastic nonsmooth convex problem}
\min_{\xbf\in\R^d} \E\left[ f(\xbf, \xi)\right],~~\text{where}~ f(\xbf, \xi)=\phi\left(\sum\limits_{i=1}^d\left(i/d+r_i\right)x_i\right)+\frac{M}{2}\sum\limits_{i=1}^d\left(x_i-1\right)_+^2+\frac{M}{2}\sum\limits_{i=1}^d\left(-x_i-1\right)_+^2,
\end{align}
where $\phi:\,\R\rightarrow\R$  in the first term of $f(\cdot,\xi)$ is a piecewise convex linear function given by
$\phi(t) = \max\{v_k+s_k t:\,k=1,...,10\},$
with $v_k$ and $s_k$ being fixed parameters and $\xi:=(r_i)$ being the vector of random parameters with additional details provided in the sequel.   The second and third terms of \eqref{stochastic nonsmooth convex problem} can be viewed as penalty functions to handle box constraints $\{\xbf\in \R^d, -1\leq\xbf\leq1\}$ with notation $(\cdot)_+:=\max\{0,\cdot\}$ and $M=1000$. 

The values of the deterministic parameters $v_k$ and $s_k$ were generated in the beginning of the experiment following the standard Gaussian distribution. After generation, these values were fixed throughout the rest of our experiment. To simulate heavy-tailed uncertainty, we constructed $r_i$, $i=1,...,d$, as $r_i=\nu_i-\E[\nu_i]$. Here, $\nu_i$ is an i.i.d. sequence of  random variables following the Type I Pareto distribution with shape parameter $3.01$ and scale parameter $1$. By this design, the subgradient of \eqref{stochastic nonsmooth convex problem} admits a bounded second moment but is heavy-tailed.   To provide a reference in evaluating solution quality, we solved a high-fidelity SAA problem \eqref{Eq: SAA} for the SP in \eqref{stochastic nonsmooth convex problem} with $N=50,000$  and considered the corresponding solution to be   a (closely near-)optimal solution. Throughout this section, this solution is denoted by  $\xbf^*$. %Hereafter in this section, this solution is then considered the 

 In this experiment setting, we evaluated the effectiveness of  SAA-L$_{q'}$ variants and LASSO, following the same configurations as discussed in Section \ref{sec: experiment on light-tailed problem}. In addition, we also included the following two settings of SMD in numerical comparison.
 
 %In addition to all the configurations reported in our linear regression experiment (SAA$_r$, SAA$_0$, SAA-L$_{q'}$, LASSO), we also report the following SMD configurations:

 \begin{description}
%\item[(E). SMD$_r$:] The SP formulation as in \eqref{Eq: SP problem statement} solved by stochastic gradient descent with the same initial solution as SAA$_r$;
\item[(A). SMD-L$_{1}$:] The SP formulation as in \eqref{Eq: SP problem statement} solved by SMD in $1$-norm setting; namely, the entropic stochastic mirror descent method discussed by \cite{nemirovski2009robust}. 
Based on the same paper, this method applies to SP problems subject to a simplex. To be consistent with this stipulation, we first observe that   \eqref{stochastic nonsmooth convex problem} is equivalent to $\min\left\{F( \xbf):\, \Vert \xbf\Vert_1\leq R_{\ell_1}\right\}$ for a given $R_{\ell_1}\geq \Vert\xbf^*\Vert_1$. This problem is then further transformed into a simplex-constrained problem as follows:
\begin{align}
\min\left\{F\big(R_{\ell_1}\cdot (\ybf_+-\ybf_-)\big):\, \mathbf 1^\top\ybf_++\mathbf 1^\top \ybf_-+s= 1,\,\ybf_+\geq 0,\,\ybf_-\geq 0,\,s\geq 0\right\}.\label{converted problem}
\end{align}
The SMD-L$_{1}$ then refers to the outcomes generated by the entropic stochastic mirror descent in  solving \eqref{converted problem} with  {\it a-priori} knowledge that $R_{\ell_1}:=2.5\cdot \Vert\xbf^*\Vert_1$. We additionally adopted  initial solution  and step size specified as $\theta\cdot  {\sqrt{2\cdot\ln d}}\cdot (\widetilde M_\infty\sqrt{N})^{-1}$, both according to \cite{nemirovski2009robust} with $\widetilde M_\infty$ and $\theta>0$ determined via the approaches  in Appendix \ref{details SMD stepsize}.% namely, an overestimate of the $\Vert\xbf^*\Vert_1$ is ass

%
%initialized at $(R/d,\dots, R/d)\in \R^d$, where $R$ is the 1-norm of the optimal solution to \eqref{Eq: SP problem statement}
\item[(B). SMD-L$_{2}$:] The SP formulation as in \eqref{Eq: SP problem statement} was solved by SMD in the 2-norm setting. More specifically, \eqref{Eq: SP problem statement}  admits an equivalent reformulation as $\min\left\{F( \xbf):\, \Vert \xbf\Vert_2\leq R_{\ell_2}\right\}$ for any given $R_{\ell_2}\geq \Vert\xbf^*\Vert_2$. Here, we assumed the {\it a-priori} knowledge that $R_{\ell_2}=2.5\cdot \Vert\xbf^*\Vert_2$. ``SMD-L$_{2}$'' then refers to the solution generated by the robust stochastic approximation method (in the 2-norm setting) as discussed by \cite{nemirovski2009robust} in solving the said reformulation. In implementing this method, the initial solution was selected to be the origin and the step size was selected as $\gamma_t=\theta\cdot  {\Vert \xbf^*\Vert_2}\cdot ({\widetilde M_2\cdot \sqrt{N}})^{-1}$ according to \cite{nemirovski2009robust} with  $\widetilde M_2$ and $\theta>0$ determined in the approaches discussed in Appendix \ref{details SMD stepsize}.  %was invoked to solve% initialized at the all-zero solution , with $0.5\Vert\cdot\Vert_{q'}^2$ as distance generating functions, $q'\in\{1.5,\,2\}$;
\end{description}

Our experiment on the performance of both SMD and SAA was performed on all combinations of $N\in\{200,\,400,\,600\}$ and $d\in\{100,\,200,\,...,\,900,\,1000,\,1500,\,2000,\,5000\}$. For each $(N,d)$-combination, five rounds of independent runs were performed.  To evaluate the quality of a given solution $\xbf$ from each run, two metrics are used. The first metric was the (approximate) suboptimality gap, calculated by $Gap(\xbf):=\frac{1}{1\times 10^4}\sum_{j=1}^{1\times 10^4} [f(\xbf,\widetilde\xi_j)-f(\xbf^*,\widetilde\xi_j)]$ for an independently generated i.i.d. sequence $\widetilde\xi_j$, $j=1,...,1\times 10^4$.  The second metric was the (approximate) $\ell_2$-loss, calculated as $\Vert\xbf-\xbf^*\Vert_2$.

We first repeated the same set of experiment as in Section \ref{sec: experiment on light-tailed problem} to compare different SAA variations and reported mean suboptimality gaps, mean $\ell_2$-loss values, and mean computational time out of five replications in Tables S4, S5, and S6 in the electronic supplement, respectively. Our observations were generally consistent with Section \ref{sec: experiment on light-tailed problem}. In particular, the below is a summary of some key patterns. 
\begin{itemize}
\item First, the solution quality of SAA$_r$ degraded (approximately linearly) as the dimension $d$ increased, whereas all other methods largely maintained better performance under growing dimensionality.  

\item Second, in many higher-dimensional settings, SAA$_0$ (i.e., the non-regularized SAA \eqref{Eq: SAA} initialized at the all-zero solution) often performed on par with LASSO, in  non-trivial contrast with SAA$_r$ --- the same SAA formulation \eqref{Eq: SAA} but started from a different initial point. This suggests again that the original non-regularized SAA \eqref{Eq: SAA} can yield solutions that are less sensitive to dimensionality when it is suitably initialized. 

\item Third, as shown in Figure \ref{fig: synthetized bars}, despite the differences in formulations, the computational effort required to solve {SAA-L$_{q'}$} (i.e., SAA in \eqref{Eq: SAA-ell2}) was comparable to that of solving SAA$_0$ or SAA$_r$ (i.e., \eqref{Eq: SAA}). In fact, none of the SAA variants tested was uniformly more efficient to solve across the range of problem instances. 
\end{itemize}

We further compared the performance of SAA with SMD to numerically validate the theoretical prediction that the two methods are comparable in solution quality. To that end, we compared (i) SAA-L$_{1.01}$ with SMD-L$_{1}$  and (ii) SAA-L$_2$ with SMD-L$_2$. These pairs of methods for comparison were selected due to their comparable  choices of distance metric. Tables S7 %\ref{SMD vs SAA performance} 
and S8 %\ref{SMD vs SAA time} 
report the results in terms of mean suboptimality gaps and computational time of SMD-L$_{1}$  (and SMD-L$_{2}$) as well as their relative differences compared to SAA-L$_{1.01}$ (and SAA-L$_{2}$, respectively), calculated as per the following formula for all $(N,p)$ combinations:
\begin{align}
\text{\uline{Relative difference in suboptimality}}:=&\,\frac{Gap(\xbf_{\text{SMD}})-Gap(\xbf_{\text{SAA}})}{\frac{1}{1\times 10^4}\sum_{j=1}^{1\times 10^4} f(\xbf^*,\widetilde\xi_j)}\times 100\%,\quad\text{and}\label{formula for relative suboptimality}
\\ \text{\uline{Relative difference in runtime}}:=&\,\frac{\text{Runtime of SMD}-\text{Runtime of SAA}}{\text{Runtime of SAA}}\times 100\%.\label{formula for relative time}\end{align}
Figure \ref{fig: SMD vs SAA} plots the relative difference in suboptimality   for   $N=600$. Meanwhile, the following patterns were observed:
\begin{itemize}
\item \color{black}The solution quality between  SMD and SAA was comparable with   $-0.97\%\sim 0.82\%$ being the range of relative differences between  SMD-L$_1$ and SAA-L$_{1.01}$  and with $-3.81\%\sim 3.80\%$ the range of relative differences between SMD-L$_2$ and SAA-L$_{2}$. This observed similarity in solution quality, given the same sample size $N$, is consistent with our theoretical prediction that SAA and SMD entail comparable sample efficiency.
\item \color{black} The computational time incurred by SMD was   significantly lower than SAA with all relative differences   ranging from $-99.93\%$ to $-99.53\%$. 
\end{itemize}
{\color{black}These observations were all consistent with the experiment results for several light-tailed SP problems reported by \cite{nemirovski2009robust}.}

\begin{figure}
    \centering
    \small 
    \setlength{\tabcolsep}{4pt}
    \begin{tabular}{ccc}
% --- First two rows: (a), (c), (e) --------------------------------
\includegraphics[width=0.4\textwidth]{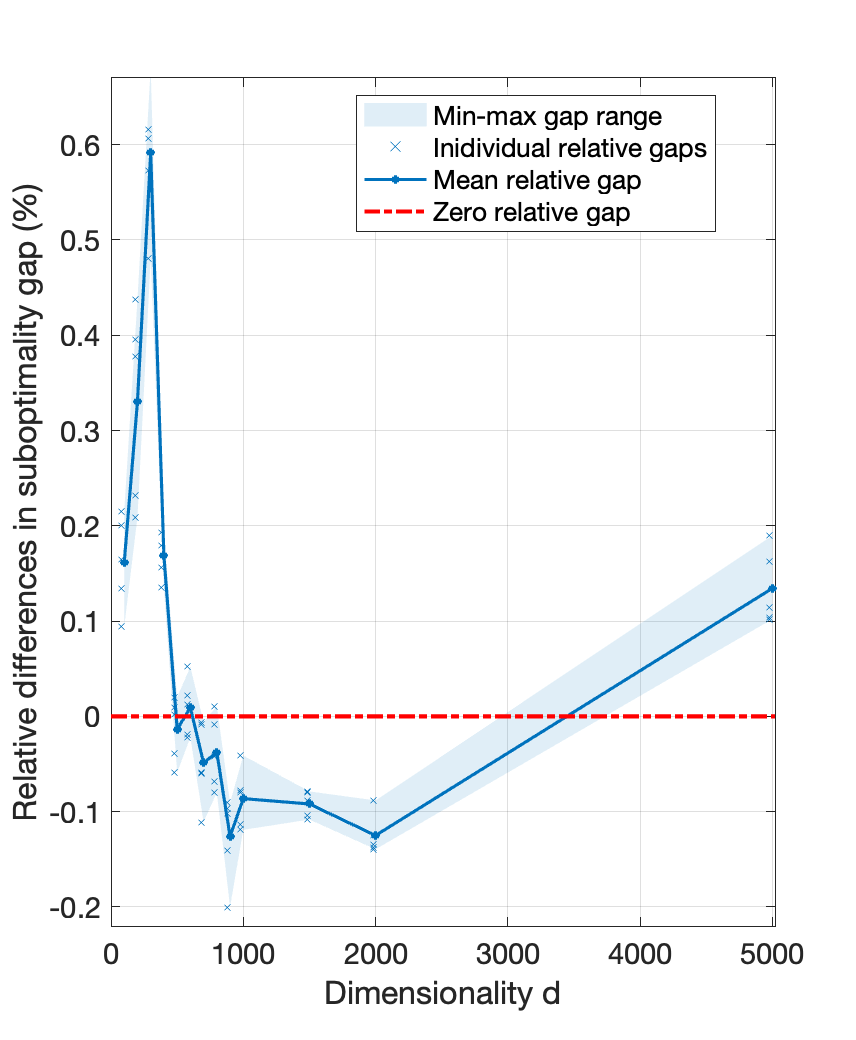} &
\includegraphics[width=0.4\textwidth]{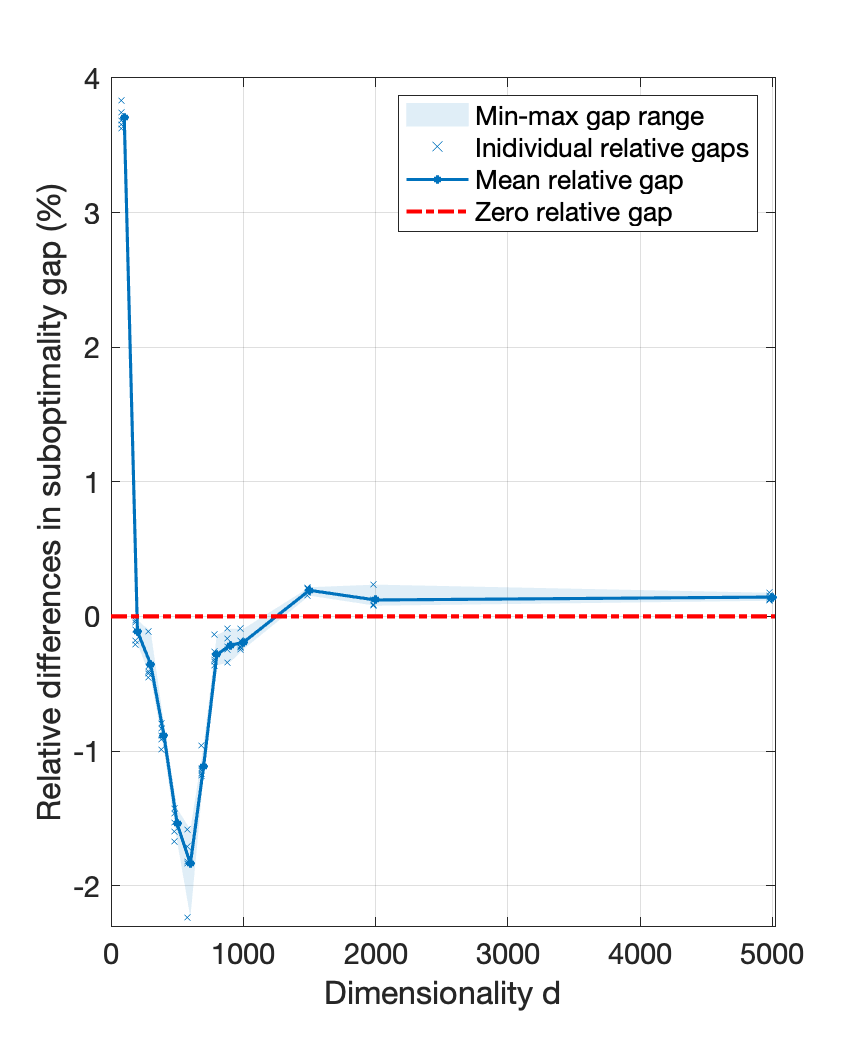} \\ %\includegraphics[width=0.33\textwidth]{Norm_2_random_initial.eps}\\
(a). SAA-L$_{1.01}$ vs.  SMD-L$_{1}$ &
(b).   SAA-L$_{2}$ vs. SMD-L$_{2}$ %& (c). SAA-L$_{2}$ vs. SMD$_{r}$

\end{tabular}

    % \begin{tabular}{cc}
    %     \includegraphics[width=0.3\textwidth]{N200.eps} & \includegraphics[width=0.3\textwidth]{N200Zoom.eps}  \\
    %     (a). $N=200$& (b). $N=200$ (magnified)
    %     \\
    %     \includegraphics[width=0.3\textwidth]{N400.eps} & \includegraphics[width=0.3\textwidth]{N400Zoom.eps}
    %     \\
    %     (c). $N=400$& (d). $N=400$ (magnified)
    %     \\ 
    %     \includegraphics[width=0.3\textwidth]{N600.eps} & \includegraphics[width=0.3\textwidth]{N600Zoom.eps} 
    %     \\
    %     (e). $N=600$& (f). $N=600$ (magnified)\\ 
    % \end{tabular}
\vspace{1mm}
    \caption{\color{black}Relative differences in suboptimality gap calculated as per \eqref{formula for relative suboptimality} for the case with $N=600$. Subplot (a) compares SAA-L$_{1.01}$ and SMD-L$_1$. Meanwhile, Subplot (b) compares  SAA-L$_{2}$ and SMD-L$_2$.}\label{fig: SMD vs SAA}
    %\label{fig: piecewise linear saa smd}
\end{figure}
   
}
   
\section{Conclusion}\label{sec: ccln}
This paper revisits the  sample complexity bounds for  SAA in both strongly convex and (non-strongly) convex SP problems. Under regularity conditions that are typical  to   the SP literature (and beyond the uniform Lipschitz condition), our findings show, perhaps for the first time, that  SAA  achieves sample complexity rates  completely free from any metric entropy terms. This represents a  deviation from the benchmark results where the inclusion of those metric entropy terms has been seemingly inevitable. Generally, as these terms elevate the dependence of the sample requirement on the problem dimensionality, our newly established sample complexity bounds are less sensitive to the increase of dimensionality compared to the state-of-the-art results in many scenarios. In particular, under  assumptions comparable to those in the existing literature on the SMD methods, a mainstream alternative solution approach to SP, part of our results provides  the first revelation that  SAA and the canonical SMD entail identical sample efficiency, closing a seemingly persistent theoretical gap of the order $O(d)$ between these two methods. Furthermore, we also identify some non-Lipschitzian cases where SAA can be shown to retain  effectiveness but, in contrast, the  results for  SMD are currently unavailable to  our knowledge. The paper also reports preliminary numerical experiments, in which our observations were consistent with our theoretical predictions.

\section*{Acknowledgement}
We would like to thank Professor Alexander Shapiro for insightful comments and fruitful discussions. We would also like to express our gratitude to the team of reviewers and editors for their efforts in evaluating and helping improve this paper.  The work is partially supported by NSF CMMI 2016571 and NSF CMMI 2213459.

%\section{Comparison with an existing lower performance bound.}

% \section*{Statement of potential broader impact}

% This paper presents work whose goal is to advance the field of Machine Learning. There are many potential societal consequences of our work, none of which we feel must be specifically highlighted here.
% % In the unusual situation where you want a paper to appear in the
% \bibliography{references}
%\bibliography{example_paper}
 \bibliographystyle{abbrvnat}
 \bibliography{references}

%%%%%%%%%%%%%%%%%%%%%%%%%%%%%%%%%%%%%%%%%%%%%%%%%%%%%%%%%%%%%%%%%%%%%%%%%%%%%%%
%%%%%%%%%%%%%%%%%%%%%%%%%%%%%%%%%%%%%%%%%%%%%%%%%%%%%%%%%%%%%%%%%%%%%%%%%%%%%%%
% APPENDIX
%%%%%%%%%%%%%%%%%%%%%%%%%%%%%%%%%%%%%%%%%%%%%%%%%%%%%%%%%%%%%%%%%%%%%%%%%%%%%%%
%%%%%%%%%%%%%%%%%%%%%%%%%%%%%%%%%%%%%%%%%%%%%%%%%%%%%%%%%%%%%%%%%%%%%%%%%%%%%%%

\begin{appendices} 

% \section{Additional proofs} \label{sec: add proofs}

% \subsection{Proof of Theorem \ref{thm: second main theorem convex optimal rate}.} \label{sec: Proof of MT2}
% % \qed

% \subsection{Proof of Theorem \ref{Most explicit result here}.}\label{sec: proof of Most explicit result here} 
%  \qed

% \subsection{Proof of Theorem \ref{thm: first main theorem}.}\label{sec: proof non-Lipschitz sc} \qed

% \subsection{Proof of Theorem \ref{thm: first main theorem convex}.}\label{sec: non-Lips gc}
%   \qed

\section{Auxiliary results}\label{sec: useful lemma 2}
This section presents useful results for our main proofs. 
{\color{black}
\begin{lemma}\label{important lemma: approximation error}
Let $q\in(1,2]$, $\epsilon\in(0,1]$, and $\beta\in(0,1)$ be any fixed scalars. Denote by $\widetilde{\xbf}(\cdot)$  a $(\delta,q)$-approximate solution to  SAA    \eqref{Eq: SAA-ell2} with its parameters ($q'$, $R^*$, $\lambda_0$, and $\delta$) specified as in \eqref{parameter choices new}. Assume that $\X$ is bounded with  $\mathcal D_{q'}$ being the $q'$-norm diameter of  $\X$ and that $f(\,\cdot\,,\xi)$ is convex on $B$ for every $\xi\in\Xi$. 
%Let $q\in($ $\widetilde{\xbf}(\cdot)$ be a  $(\delta,q)$-approximate solution  to SAA  \eqref{Eq: SAA-ell2} as in \eqref{define solution}. 
Under  Assumption   \ref{assumption f lipschitz}.(a), function $g(\cdot,\cdot)$ as defined in \eqref{constructed new function} satisfies the following inequality
\begin{multline}
-\varkappa_\delta\leq g(\boldsymbol\zeta_{1,N},\zeta)-\mathds{1}(\zeta\in\mathcal E_t)\cdot \left[f_{\lambda_0}(\widetilde\xbf(\boldsymbol\zeta_{1,N}),\zeta)+\mathcal D_{q'} \cdot (\lambda_0\mathcal D_{q'}+t)-f_{\lambda_0}(\xbf^*,\zeta)\right]\leq 0,\\~~\text{for all $\boldsymbol\zeta_{1,N}\in\mathcal E_t^N$ and all $\zeta\in\Xi$},
\label{reformed g} %\textcolor{black}{\widetilde\xbf(\boldsymbol\zeta_{1,N})}
\end{multline}
where $\varkappa_\delta$ is  the same as  in \eqref{eq: define gamma} and   $\mathcal E_t^N$ is defined as in \eqref{event total consideration}.
\end{lemma}
\begin{proof} For  $\gamma>0$ given as in \eqref{eq: define gamma},  let $d_{\gamma}:\Xi^{2N}\rightarrow \R_+$ be the Hamming distance defined as in \eqref{hamming}.
We first observe the below for all $j=1,...,N$, and all $ (\boldsymbol \zeta_{1,N}, \boldsymbol\zeta_{1,N}')\in\mathcal E_t^N\times \mathcal E_t^N$:%,  it must hold that
\begin{align}
&\ \frac{1}{N}\sum_{j=1}^Nf_{\lambda_0}\left(\widetilde{\xbf}(\boldsymbol\zeta'_{1,N}),\zeta_j\right)-\frac{1}{N}\sum_{j=1}^Nf_{\lambda_0}\left(\widetilde{\xbf}(\boldsymbol\zeta_{1,N}),\zeta_j\right) %\textcolor{black}{do we use $\xi$ or $\zeta$}
\label{initial gap to bound and repli}
\\=& \sum_{j:\,\zeta_j\neq \zeta_j'}\frac{f_{\lambda_0}\left(\widetilde{\xbf}(\boldsymbol\zeta'_{1,N}),\zeta_{j}\right)-f_{\lambda_0}\left(\widetilde{\xbf}(\boldsymbol\zeta_{1,N}),\zeta_j\right)}{N} 
 {+}\sum_{{j:\,\zeta_j=\zeta_j'}}\frac{f_{\lambda_0}\left(\widetilde{\xbf}(\boldsymbol\zeta'_{1,N}),\zeta_{j}\right)-f_{\lambda_0}\left(\widetilde{\xbf}(\boldsymbol\zeta_{1,N}),\zeta_j\right)}{N}\nonumber
\\= &  {\sum_{j:\,\zeta_j\neq \zeta_j'}}\frac{f_{\lambda_0}\left(\widetilde{\xbf}(\boldsymbol\zeta'_{1,N}),\zeta_{j}\right)-f_{\lambda_0}\left(\widetilde{\xbf}(\boldsymbol\zeta_{1,N}),\zeta_j\right)}{N} - {\sum_{j:\,\zeta_j\neq \zeta_j'}}\frac{f_{\lambda_0}\left(\widetilde{\xbf}(\boldsymbol\zeta'_{1,N}),\zeta_{j}'\right)-f_{\lambda_0}\left(\widetilde{\xbf}(\boldsymbol\zeta_{1,N}),\zeta_j'\right)}{N}\nonumber
\\&+{\sum_{j:\,\zeta_j\neq \zeta_j'}}\frac{f_{\lambda_0}\left(\widetilde{\xbf}(\boldsymbol\zeta'_{1,N}),\zeta_{j}'\right)-f_{\lambda_0}\left(\widetilde{\xbf}(\boldsymbol\zeta_{1,N}),\zeta_j'\right)}{N}+{\sum_{j:\,\zeta_j= \zeta_j'}}\frac{f_{\lambda_0}\left(\widetilde{\xbf}(\boldsymbol\zeta'_{1,N}),\zeta_{j}\right)-f_{\lambda_0}\left(\widetilde{\xbf}(\boldsymbol\zeta_{1,N}),\zeta_j\right)}{N}\nonumber
\\
= &  {\sum_{j:\,\zeta_j\neq \zeta_j'}}\frac{f_{\lambda_0}\left(\widetilde{\xbf}(\boldsymbol\zeta'_{1,N}),\zeta_{j}\right)-f_{\lambda_0}\left(\widetilde{\xbf}(\boldsymbol\zeta_{1,N}),\zeta_j\right)}{N} - {\sum_{j:\,\zeta_j\neq \zeta_j'}}\frac{f_{\lambda_0}\left(\widetilde{\xbf}(\boldsymbol\zeta'_{1,N}),\zeta_{j}'\right)-f_{\lambda_0}\left(\widetilde{\xbf}(\boldsymbol\zeta_{1,N}),\zeta_j'\right)}{N}\nonumber
\\&+{\sum_{j:\,\zeta_j\neq \zeta_j'}}\frac{f_{\lambda_0}\left(\widetilde{\xbf}(\boldsymbol\zeta'_{1,N}),\zeta_{j}'\right)-f_{\lambda_0}\left(\widetilde{\xbf}(\boldsymbol\zeta_{1,N}),\zeta_j'\right)}{N}+{\sum_{j:\,\zeta_j= \zeta_j'}}\frac{f_{\lambda_0}\left(\widetilde{\xbf}(\boldsymbol\zeta'_{1,N}),\zeta_{j}'\right)-f_{\lambda_0}\left(\widetilde{\xbf}(\boldsymbol\zeta_{1,N}),\zeta_j'\right)}{N}\nonumber
\\= &  {\sum_{j:\,\zeta_j\neq \zeta_j'}}\frac{f_{\lambda_0}\left(\widetilde{\xbf}(\boldsymbol\zeta'_{1,N}),\zeta_{j}\right)-f_{\lambda_0}\left(\widetilde{\xbf}(\boldsymbol\zeta_{1,N}),\zeta_j\right)}{N} - {\sum_{j:\,\zeta_j\neq \zeta_j'}}\frac{f_{\lambda_0}\left(\widetilde{\xbf}(\boldsymbol\zeta'_{1,N}),\zeta_{j}'\right)-f_{\lambda_0}\left(\widetilde{\xbf}(\boldsymbol\zeta_{1,N}),\zeta_j'\right)}{N}\nonumber
\\&+\frac{1}{N}{\sum_{j=1}^N}f_{\lambda_0}\left(\widetilde{\xbf}(\boldsymbol\zeta'_{1,N}),\zeta_{j}'\right)-\frac{1}{N}{\sum_{j=1}^N} f_{\lambda_0}\left(\widetilde{\xbf}(\boldsymbol\zeta_{1,N}),\zeta_j'\right) \nonumber
\\\leq &  {\sum_{j:\,\zeta_j\neq \zeta_j'}}\frac{f_{\lambda_0}\left(\widetilde{\xbf}(\boldsymbol\zeta'_{1,N}),\zeta_{j}\right)-f_{\lambda_0}\left(\widetilde{\xbf}(\boldsymbol\zeta_{1,N}),\zeta_j\right)}{N} - {\sum_{j:\,\zeta_j\neq \zeta_j'}}\frac{f_{\lambda_0}\left(\widetilde{\xbf}(\boldsymbol\zeta'_{1,N}),\zeta_{j}'\right)-f_{\lambda_0}\left(\widetilde{\xbf}(\boldsymbol\zeta_{1,N}),\zeta_j'\right)}{N}\nonumber
\\&+\delta\Vert \widetilde{\mathbf x}(\boldsymbol\zeta_{1,N}')-\widetilde{\mathbf x}(\boldsymbol\zeta_{1,N})\Vert_{q'},\label{second row new M}
\end{align}
where  \eqref{second row new M} is due to the fact that $\widetilde{\xbf}(\boldsymbol\zeta'_{1,N})$ satisfies \eqref{suboptimality in computing SAA} and $q'\leq q$.
Invoking   
  Assumption \ref{assumption f lipschitz}.(a) (and noting that $q'\leq q$ again) and the Lipschitz condition of $V_{q'}$ (as shown in \eqref{Lipschitz inequality of Vq} of Appendix \ref{sec: preliminary} for completeness), we   further obtain
\begin{align}
\text{Eq.\,}\eqref{initial gap to bound and repli}\nonumber
 \leq &  \left[{\sum_{j=1}^N }\frac{  M(\zeta_j)+M(\zeta_j')+2\lambda_0 {\mathcal D_{q'}}   }{N }    \mathds{1}(\zeta_{j}'\neq \zeta_{j}) +\delta\right] \Vert \widetilde{\mathbf x}(\boldsymbol\zeta_{1,N}')-\widetilde{\mathbf x}(\boldsymbol\zeta_{1,N}) \Vert_{q'}\nonumber
\\\leq &\left[ \frac{ 2\cdot \left(t+\lambda_0 {\mathcal D_{q'}}\right) }{N }\cdot \sum_{j=1}^N   \mathds{1}(\zeta_{j}'\neq \zeta_{j})+\delta\right]\cdot \Vert\widetilde{\mathbf x}(\boldsymbol\zeta_{1,N}')-\widetilde{\mathbf x}(\boldsymbol\zeta_{1,N}) \Vert_{q'},
 \label{vital step new}
\end{align}
where the last inequality is due to the assumption that $ (\boldsymbol \zeta_{1,N}, \boldsymbol\zeta_{1,N}')\in\mathcal E_t^N\times \mathcal E_t^N$ with $\mathcal E_t^N$ defined as in \eqref{event total consideration}.

Note that $V_{q'}(\xbf):=\frac{1}{2} \Vert\xbf-\xbf^0\Vert_{q'}^2$  with $q'\in(1,2]$ is $(q'-1)$-strongly convex w.r.t. the $q'$-norm \citep{ben2001ordered} in the sense of \eqref{SC V} in Appendix \ref{sec: preliminary}.
Thus, if we additionally invoke Assumption \ref{GC condition constant all} and the definition of $\widetilde{\xbf}(\cdot)$ as in \eqref{define solution} (as well as the fact that $q'\leq q$), we have 
\begin{align}
& \frac{1}{N}\sum_{j=1}^Nf_{\lambda_0}\left(\widetilde{\xbf}(\boldsymbol\zeta'_{1,N}),\zeta_j\right)-\frac{1}{N}\sum_{j=1}^Nf_{\lambda_0}\left(\widetilde{\xbf}(\boldsymbol\zeta_{1,N}),\zeta_j\right)\geq  \lambda_0\cdot \frac{q'-1}{2}\left\Vert\widetilde{\mathbf x}(\boldsymbol\zeta_{1,N}')-\widetilde{\mathbf x}(\boldsymbol\zeta_{1,N}) \right\Vert_{q'}^2-\delta\cdot \Vert \widetilde{\mathbf x}(\boldsymbol\zeta_{1,N}')-\widetilde{\mathbf x}(\boldsymbol\zeta_{1,N}) \Vert_{q'}. \nonumber
\end{align}
Combining this with \eqref{vital step new}, we  obtain  
$
\left\Vert \widetilde{\mathbf x}(\boldsymbol\zeta_{1,N}')-\widetilde{\mathbf x}(\boldsymbol\zeta_{1,N})\right\Vert_{q'}\nonumber
\leq   \frac{4\cdot (t +\lambda_0 {\mathcal D_{q'}})}{N\cdot \lambda_0\cdot (q'-1)} \cdot \sum_{j=1}^N   \mathds{1}(\zeta_{j}'\neq \zeta_{j})+\frac{4\delta}{\lambda_0\cdot (q'-1)}.%\label{solution distance}
$  
Further invoking Assumption \ref{assumption f lipschitz}.(a) as well as \eqref{Lipschitz inequality of Vq}, for $\zeta\in\mathcal E_t$, we have
\begin{align}
&\vert f_{\lambda_0}(\widetilde{\xbf}(\boldsymbol\zeta_{1,N}),\zeta)-f_{\lambda_0}(\widetilde{\xbf}(\boldsymbol\zeta'_{1,N}),\zeta)\vert\leq  (t+\lambda_0\mathcal D_{q'})\cdot \left\Vert \widetilde{\mathbf x}(\boldsymbol\zeta_{1,N}')-\widetilde{\mathbf x}(\boldsymbol\zeta_{1,N})\right\Vert_{q'}\nonumber
\\\leq &  \frac{4\cdot (t+\lambda_0 \mathcal D_{q'})^2}{N\cdot \lambda_0\cdot (q'-1)}\cdot \sum_{j=1}^N\mathds{1}(\zeta_{j}'\neq \zeta_{j}) 
+\frac{4\delta\cdot (t+\lambda_0\mathcal D_{q'})}{\lambda_0\cdot (q'-1)}=\gamma\sum_{j=1}^N\mathds{1}(\zeta_{j}'\neq \zeta_{j})+\varkappa_\delta=d_{\gamma}(\boldsymbol\zeta_{1,N}',\boldsymbol\zeta_{1,N})+\varkappa_\delta.\label{where define gamma}
\end{align}
%as desired with $\gamma$ and $\varkappa_\delta$ given as in \eqref{eq: define gamma} .

In using the above to prove the desired results,  we consider two complementary cases: $\zeta\notin\mathcal E_t$ or $\zeta\in\mathcal E_t$ below:

{\bf Case 1.} We first observe that the desired inequalities in \eqref{reformed g} hold trivially for the case of $\zeta\notin\mathcal E_t$ by invoking the definition of $g$ in \eqref{constructed new function}. 

{\bf Case 2.} We  consider the   case of $\zeta\in\mathcal E_t$ below.
Invoking \eqref{where define gamma} implies that $f_{\lambda_0}(\widetilde{\xbf}(\boldsymbol z),\zeta)+d_{\gamma}(\boldsymbol\zeta_{1,N},\boldsymbol z)\geq f_{\lambda_0}(\widetilde{\xbf}(\boldsymbol\zeta_{1,N}),\zeta)-\varkappa_\delta$ for any $(\boldsymbol\zeta_{1,N},\boldsymbol z)\in \mathcal E_t^N\times \mathcal E_t^N$. Thus, it holds for  all $\boldsymbol\zeta_{1,N} \in\mathcal E_t^N$  that \begin{align}
&f_{\lambda_0}(\widetilde{\xbf}(\boldsymbol \zeta_{1,N}),\zeta)=f_{\lambda_0}(\widetilde{\xbf}(\boldsymbol \zeta_{1,N}),\zeta)+d_\gamma(\boldsymbol \zeta_{1,N},\boldsymbol \zeta_{1,N})
\geq\inf_{\boldsymbol z\in\mathcal E^N_t}\left\{f_{\lambda_0}(\widetilde{\xbf}(\boldsymbol z),\zeta)+d_{\gamma}(\boldsymbol\zeta_{1,N},\boldsymbol z)\right\}\geq f_{\lambda_0}(\widetilde{\xbf}(\boldsymbol \zeta_{1,N}),\zeta)-\varkappa_\delta.\nonumber
\end{align}
As an immediate result,
\begin{align}
&
\mathds{1}(\zeta\in\mathcal E_t)\cdot\left[ f_{\lambda_0}(\widetilde{\xbf}(\boldsymbol \zeta_{1,N}),\zeta)+\mathcal D_{q'}\cdot (\lambda_0\mathcal D_{q'}+t)-f_{\lambda_0}(\xbf^*,\zeta)\right]\nonumber
\\ \geq&  \inf_{\boldsymbol z\in\mathcal E^N_t}\left\{\mathds{1}(\zeta\in\mathcal E_t)\cdot\left[f_{\lambda_0}(\widetilde{\xbf}(\boldsymbol z),\zeta)+d_{\gamma}(\boldsymbol\zeta_{1,N},\boldsymbol z)+\mathcal D_{q'}\cdot (\lambda_0\mathcal D_{q'}+t)-f_{\lambda_0}(\xbf^*,\zeta)\right]\right\} \nonumber
\\ \geq &\mathds{1}(\zeta\in\mathcal E_t)\cdot\left[ f_{\lambda_0}(\widetilde{\xbf}(\boldsymbol \zeta_{1,N}),\zeta)+\mathcal D_{q'}\cdot (\lambda_0\mathcal D_{q'}+t)-f_{\lambda_0}(\xbf^*,\zeta)\right]-\varkappa_\delta,~~\forall \boldsymbol\zeta_{1,N}\in\mathcal E_t^N.\label{convert inf}
\end{align}

Meanwhile, in the same case of $\zeta\in\mathcal E_t$, Assumption \ref{assumption f lipschitz}.(a) together with \eqref{Lipschitz inequality of Vq} implies that,  for every $\xbf\in\X$, we have $\vert f_{\lambda_0}(\xbf,\zeta)-f_{\lambda_0}(\xbf^*,\zeta)\vert \leq (t+\lambda_0\mathcal D_{q'})\cdot \Vert\xbf-\xbf^*\Vert_{q'}\leq (t+\lambda_0\mathcal D_{q'})\cdot\mathcal D_{q'}$. As an immediate result,  $2(t+\lambda_0\mathcal D_{q'})\cdot\mathcal D_{q'}\geq f_{\lambda_0}(\xbf,\zeta)+(t+\lambda_0\mathcal D_{q'})\cdot \mathcal D_{q'}-f_{\lambda_0}(\xbf^*,\zeta)$. This immediately leads to
\begin{align}
2(t+\lambda_0\mathcal D_{q'})\cdot\mathcal D_{q'}\geq \mathds{1}(\zeta\in\mathcal E_t)\cdot\left[f_{\lambda_0}(\xbf,\zeta)+(t+\lambda_0\mathcal D_{q'})\cdot \mathcal D_{q'}-f_{\lambda_0}(\xbf^*,\zeta)\right],~~\forall (\xbf,\zeta)\in\X\times \mathcal E_t.\label{conver min}
\end{align}
Therefore,   we may continue from \eqref{convert inf} to obtain, for any $\boldsymbol\zeta_{1,N}\in\mathcal E_t^N$, 
\begin{align}
&\mathds{1}(\zeta\in\mathcal E_t)\cdot\left[ f_{\lambda_0}(\widetilde{\xbf}(\boldsymbol \zeta_{1,N}),\zeta)+\mathcal D_{q'}\cdot (\lambda_0\mathcal D_{q'}+t)-f_{\lambda_0}(\xbf^*,\zeta)\right]\label{first line here}
\\\stackrel{\eqref{conver min}}{=}&\min\left\{\vphantom{V^{V^{V^V}}}2\mathcal D_{q'}\cdot (\lambda_0\mathcal D_{q'}+t),~\mathds{1}(\zeta\in\mathcal E_t)\cdot\left[ f_{\lambda_0}(\widetilde{\xbf}(\boldsymbol \zeta_{1,N}),\zeta)+\mathcal D_{q'}\cdot (\lambda_0\mathcal D_{q'}+t)-f_{\lambda_0}(\xbf^*,\zeta)\right]\right\}\nonumber
\\
\geq~&\min\left\{2\mathcal D_{q'}\cdot (\lambda_0\mathcal D_{q'}+t),~\vphantom{V^{V^{V^{V^{V^V}}}}_{V_{V_V}}}\right.\nonumber
\\&\left.\vphantom{V^{V^{V^V}}_{V_{V_V}}} \inf_{\boldsymbol z\in\mathcal E^N_t}\left\{\mathds{1}(\zeta\in\mathcal E_t) \cdot\left[f_{\lambda_0}(\widetilde{\xbf}(\boldsymbol z),\zeta)+d_\gamma(\boldsymbol\zeta_{1,N},\boldsymbol z)+\mathcal D_{q'}\cdot (\lambda_0\mathcal D_{q'}+t)-f_{\lambda_0}(\xbf^*,\zeta) \right]\vphantom{V^{V^V}_{V_V}}\right\} \vphantom{V^{V^{V^V}}_{V_{V_V}}}\right\}\stackrel{\eqref{constructed new function}}{=}g(\boldsymbol\zeta_{1,N},\zeta)\label{mid line here}
\\
\stackrel{\eqref{convert inf}}{\geq}&\min\left\{2\mathcal D_{q'}\cdot (\lambda_0\mathcal D_{q'}+t),\vphantom{V^{V^{V^V}}_{V_{V_V}}}~ \mathds{1}(\zeta\in\mathcal E_t) \cdot \left[f_{\lambda_0}(\widetilde{\xbf}(\boldsymbol\zeta_{1,N}),\zeta)+\mathcal D_{q'}\cdot (\lambda_0\mathcal D_{q'}+t)-f_{\lambda_0}(\xbf^*,\zeta) \vphantom{V^{V^V}_{V_V}}\right] \vphantom{V^{V^{V^V}}_{V_{V_V}}}-\varkappa_\delta\right\}\nonumber
\\
 {\geq}~&\min\left\{2\mathcal D_{q'}\cdot (\lambda_0\mathcal D_{q'}+t),\vphantom{V^{V^{V^V}}_{V_{V_V}}}~ \mathds{1}(\zeta\in\mathcal E_t) \cdot \left[f_{\lambda_0}(\widetilde{\xbf}(\boldsymbol\zeta_{1,N}),\zeta)+\mathcal D_{q'}\cdot (\lambda_0\mathcal D_{q'}+t)-f_{\lambda_0}(\xbf^*,\zeta) \vphantom{V^{V^V}_{V_V}}\right] \vphantom{V^{V^{V^V}}_{V_{V_V}}}\right\}-\varkappa_\delta\nonumber
\\\stackrel{\eqref{conver min}}{=}&  \mathds{1}(\zeta\in\mathcal E_t) \cdot \left[f_{\lambda_0}(\widetilde{\xbf}(\boldsymbol\zeta_{1,N}),\zeta)+\mathcal D_{q'}\cdot (\lambda_0\mathcal D_{q'}+t)-f_{\lambda_0}(\xbf^*,\zeta) \vphantom{V^{V^V}_{V_V}}\right]-\varkappa_\delta,\label{last line here}
\end{align}
Joining \eqref{first line here}, \eqref{mid line here}, and \eqref{last line here}, we immediately have, for any $\boldsymbol\zeta_{1,N}\in\mathcal E_t^N$
\begin{align*}
-\varkappa_\delta\leq g(\boldsymbol\zeta_{1,N},\zeta) -\mathds{1}(\zeta\in\mathcal E_t) \cdot \left\{f_{\lambda_0}(\widetilde{\xbf}(\boldsymbol\zeta_{1,N}),\zeta)+\mathcal D_{q'}\cdot (\lambda_0\mathcal D_{q'}+t)-f_{\lambda_0}(\xbf^*,\zeta) \vphantom{V^{V^V}_{V_V}}\right\}\leq 0.
\end{align*}
Combining the said two cases, we have the desired result of this lemma.
\end{proof}
}
\bigskip

Below we recall the notation that  $\xi$ denotes a random vector  with $\Prob[\xi\in\Xi]=1$, the sample $\xi_1,\ldots,\xi_N$ are i.i.d.\ copies of $\xi$, and, for $\xi_j'\in\Xi$ and $j=1,...,N$,
\[
\boldsymbol\xi_{1,N}=(\xi_1,\ldots,\xi_N),\qquad 
\boldsymbol\xi^{(j)}_{1,N}=(\xi_1,\ldots,\xi_{j-1},\xi_j',\xi_{j+1},\ldots,\xi_N).
\]

\begin{proposition}\label{very important generalization proposition}
Let $\beta\in(0,1)$, $U>0$,  and $A:\Xi^N\times\Xi\to[0,U]$ be a deterministic function that is  measurable   w.r.t.\,the $P^{N+1}$-completion of $\mathcal B(\Xi^N\times\Xi)$. Assume that, for some $\gamma\ge0$,
\begin{equation}\label{bounded difference}
\big|A(\boldsymbol\xi_{1,N},\xi)-A(\boldsymbol\xi^{(j)}_{1,N},\xi)\big|\le\gamma,
\quad \forall\,(\boldsymbol\xi_{1,N},\xi_j',\xi)\in\Xi^N\times\Xi\times\Xi,\;\; j=1,\ldots,N.
\end{equation}
Then there exists a universal constant $c>0$ such that
\[
\Prob\!\left[
\left|\,
\E_{\xi}\!\big[A(\boldsymbol\xi_{1,N},\xi)\big]
-\frac{1}{N}\sum_{j=1}^N A(\boldsymbol\xi_{1,N},\xi_j)
\right|
\;\ge\;
c\!\left(\gamma\ln N\cdot\ln\frac{N}{\beta}
+U\sqrt{\frac{\ln(1/\beta)}{N}}\right)
\right]\;\le\;\beta.
\]
%Here   $\Prob_{\boldsymbol\xi_{1,N}\sim P^{\otimes N}}$ denotes probability taken over the sample $\boldsymbol\xi_{1,N}\sim P^{\otimes N}$.  
\end{proposition}
\begin{proof} 
This is a straightforward result of Theorem 1.1 by \citet{feldman2019high}.  More specifically, let $A' :\Xi^N\times \Xi\rightarrow [0,1]$ be defined as $A':= U^{-1}\cdot A$. By assumption, it holds that  
$
\vert A'(\boldsymbol\xi_{1,N},\xi)-A'(\boldsymbol\xi^{(j)}_{1,N},\xi)\vert \leq \gamma/U,~~\forall (\boldsymbol\xi_{1,N},\xi'_j,\xi)\in\Xi^N\times \Xi\times \Xi,~~j=1,...,N.$
Invoking Theorem 1.1 in \citet{feldman2019high} (where  $A$ and $\gamma$ therein are specified as $A'$ and $\gamma/U$, respectively), we have
\begin{align*}
\beta\geq &\,\Prob\left[\left\vert \E_{\xi}[A(\boldsymbol\xi_{1,N},\xi)]-\frac{1}{N}\sum_{j=1}^N  A(\boldsymbol\xi_{1,N},\xi_j)\right\vert\geq  c\left(\gamma\ln N  \cdot \ln \frac{N}{\beta}+ U\cdot\sqrt{\frac{\ln(1/\beta)}{N} }\right)\right]
\\=&\,\Prob\left[U\cdot \left\vert \E_{\xi}[A'(\boldsymbol\xi_{1,N},\xi)]-\frac{1}{N}\sum_{j=1}^N  A'(\boldsymbol\xi_{1,N},\xi_j)\right\vert\geq  c\cdot U\left(\frac{\gamma}{U}\ln N  \cdot \ln \frac{N}{\beta}+ \sqrt{\frac{\ln(1/\beta)}{N} }\right)\right],
\end{align*}
which immediately leads to the desired result. 
\end{proof}

\begin{lemma}\label{useful lemma 2}
 Let $p \in[2,\infty)$. Denote by $\xi_1,...,\xi_N {\in\R}$ an i.i.d. sequence of random variables with $\E[\xi_1]=0$. Then     $\left\Vert N^{-1}\sum_{j=1}^N \xi_j\right\Vert_{L^p}\leq C\cdot \sqrt{p N^{-1} }\cdot \left\Vert \xi_1 \right\Vert_{L^{p}}$ for some universal constant $C>0$.
 \end{lemma}

  \begin{proof}
 This lemma is largely based on the approach used in the proof of Proposition 1 by \cite{oliveira2023sample}. By Marcinkiewicz' inequality \citep{boucheron2013concentration}, for some universal constant $C>0$, it holds that
$
\left\Vert N^{-1}\sum_{j=1}^N \xi_j\right\Vert_{L^p}\leq  N^{-1}C\cdot \sqrt{p}\left\Vert \sum_{j=1}^N \xi_j^2\right\Vert_{L^{p/2}}^{1/2}\leq N^{-1}C\cdot \sqrt{p}\sqrt{\sum_{j=1}^N \left\Vert  \xi_j^2\right\Vert_{L^{p/2}}}\nonumber=C\cdot \sqrt{p\cdot N^{-1} }\cdot \left\Vert \xi_1 \right\Vert_{L^{p}},\nonumber
$
as desired.  
 \end{proof}
 \bigskip
 
\begin{lemma}\label{useful lemma 2 tail bound}
 Let $p \in[2,\infty)$. Denote by $\vbf_1,...,\vbf_N {\in\R^d}$ an i.i.d. sequence of $d$-dimensional random vectors with $\mathbb{E}[\vbf_1] = \mathbf 0$. For any $t>0$, it holds that
 $
 \Prob\left[\left\Vert N^{-1}\sum_{j=1}^N \vbf_j\right\Vert_{p}^2\geq t\right]\leq \left(\widetilde C\cdot \left\Vert \vbf_1 \right\Vert_{L^{p}}\cdot \sqrt{\frac{p }{Nt}} \right)^p,
 $
 for some universal constant $\widetilde C>0$.
 \end{lemma}
\begin{proof} 
The proof slightly strengthens that of Proposition 1 by \cite{oliveira2023sample} and is a quick result of the Markov's inequality. Let $v_{ij}$ be the $i$th component of $\vbf_j$. Then, it holds, for any $t>0$, that
\begin{align}
&\,\Prob\left[\left\Vert N^{-1}\sum_{j=1}^N \vbf_j\right\Vert_{p}^2\geq t\right]=\Prob\left[\sum_{i=1}^d \left\vert N^{-1}\sum_{j=1}^N v_{ij}\right\vert^p \geq t^{p/2}\right]\nonumber
\\\stackrel{\text{Markov's}}\leq &\, \frac{\E\left[\sum_{i=1}^d \left\vert N^{-1}\sum_{j=1}^N v_{ij}\right\vert^p\right]}{t^{p/2}}=\frac{\sum_{i=1}^d\left(\left\Vert N^{-1}\sum_{j=1}^N v_{ij}\right\Vert_{L^p}\right)^p}{t^{p/2}}
\stackrel{\text{Lemma \ref{useful lemma 2}}}\leq \,\sum_{i=1}^d\left(\widetilde C\cdot \sqrt{\frac{p}{Nt}}\Vert v_{i1}\Vert_{L^p}\right)^p\nonumber
\\=&\,\left(\widetilde C\cdot \sqrt{\frac{p}{Nt}}\right)^p\cdot \sum_{i=1}^d\Vert v_{i1}\Vert_{L^p}^p=\left(\widetilde C\cdot \sqrt{\frac{p}{Nt}}\right)^p\cdot \Vert \vbf_{1}\Vert_{L^p}^p, \nonumber
\end{align}
for some universal constant $\widetilde C>0$, which  immediately leads to the desired result.   
\end{proof}

\section{Proof of Corollary \ref{most explicit}}\label{proof of corollary 1}
{\it [To show Part (a)]:} Observe that, by convexity, $f(\xbf_1,\xi)-f(\xbf_2,\xi)\leq \langle \mathbf g,\xbf_1-\xbf_2\rangle\leq  \Vert \mathbf g\Vert_p\cdot \Vert \xbf_1-\xbf_2\Vert_{q'}$ for every $\mathbf g\in\partial_{\xbf}f(\xbf_1,\xi)$ for   $p:\,2\leq p\leq q'/(q'-1)$ and all $\xbf_1,\xbf_2\in\X$. Furthermore, in view of the assumption that, for all $\xbf\in\X$ and every $\xi\in\Xi$, it holds that $ \Vert\gbf\Vert_{p}\leq \widetilde M_{p}(\xi)$ for some $\gbf\in\partial_{\xbf} f(\xbf,\xi)$, we obtain $f(\xbf_1,\xi)-f(\xbf_2,\xi)\leq  \widetilde M_{p}(\xi) \cdot \Vert \xbf_1-\xbf_2\Vert_{q'}$. Further invoking the assumption that $\Prob\Big[\widetilde M_{p}(\xi)\leq t\Big]\geq 1-2\exp(-t^2/\varphi_{f,p}^2)$ for all $t\geq 0$, we know that Assumption \ref{assumption f lipschitz v3} holds with the $q'$-norm and  $M:=\widetilde M_{p}$. We may invoke Theorem \ref{Most explicit result here}(b)  to obtain the desired result in Part (a).% of this corollary. 
\smallskip
\smallskip
\sloppy

\noindent{\it [To show Part (b)]:} We would like to first show that,  for a universal constant $C'\geq 2$,  it holds that $\Prob\Big[\max_{i=1,...,d} \widetilde M_{i}(\xi) \leq t\Big]\geq 1-2\exp(-t^2/(C'\varphi_{f,\infty}^2\ln d)\big)$ for all $t\geq 0$. To that end,  consider two cases. First,  when $t\leq \varphi_{f,\infty}\sqrt{\ln (2d)}$,   it holds that $1-2\exp(-t^2/(C'\varphi_{f,\infty}^2\ln d)\big)\leq 0\leq \Prob\Big[\max_{i=1,...,d} \widetilde M_{i}(\xi) \leq t\Big]$. 

\sloppy
Second, in the complement case of $t> \varphi_{f,\infty}\sqrt{\ln (2d)}$, it holds that $t^2/\varphi^2_{f,\infty}\geq \ln 2d\geq \ln 2$. Examine a function $h(s):=(s-\ln d)-\frac{s}{C'\ln d}$. Since $C'\ln d>1$, we have $h'(s)=1-(C'\ln d)^{-1}>0$ and thus $h(s)$ is increasing in $s$. Now that $h\big(\ln(2d)\big)=\ln(2d)-\ln d-\frac{\ln(2d)}{C'\ln d}=\ln(2)-\frac{\ln(2d)}{C'\ln d}\geq 0$, we have $h(s)\geq 0$ for all  $s>\ln (2d)$.  
Therefore, $h(t^2/\varphi^2_{f,\infty})\geq 0$, which implies that 
\begin{align}t^2/\varphi^2_{f,\infty}-\ln d\geq t^2/(C'\cdot \varphi^2_{f,\infty}\ln d).\label{useful bridge}
\end{align} Under the assumption  that   $\Prob\Big[ \widetilde M_i(\xi)\leq t\Big]\geq 1-2\exp(-t^2/\varphi_{f,\infty}^2)$ for every $i=1,...,d$ and all $t\geq 0$, we may invoke the union bound to obtain $\Prob\Big[\max_{i=1,...,d}\widetilde M_i(\xi) \leq t\Big]\geq 1-2d\exp(-t^2/ \varphi_{f,\infty}^2 \big)=1-2\exp(-t^2/ \varphi_{f,\infty}^2 +\ln d\big)$. Joining this with \eqref{useful bridge} then implies that  $\Prob\Big[\max_{i=1,...,d}\widetilde M_i(\xi) \leq t\Big]\geq  1-2\exp\big(-t^2/ \varphi_{f,\infty}^2 +\ln d\big)\geq 1-2\exp(-t^2/(C'\varphi_{f,\infty}^2\ln d)\big)$ for all $t> \varphi_{f,\infty}\sqrt{\ln (2d)}$. Let $\widetilde{\mathbf M}(\xi)=\big(\widetilde M_i(\xi):\,i=1,...,d\big)$. Combining the two cases above, we have that $\Prob\Big[ \Vert \widetilde{\mathbf M}(\xi) \Vert_{\ln d} \leq  e t\Big]\geq \Prob\Big[\max_{i=1,...,d}\widetilde M_i(\xi) \leq t\Big]\geq 1-2\exp(-t^2/(C'\varphi_{f,\infty}^2\ln d)\big)$ for all $t\geq 0$. Observe that by convexity and by our assumption, $ f(\xbf_1,\xi)-f(\xbf_1+r\hat{\textnormal{e}}_i,\xi) \leq g_i\cdot r\leq \vert g_i\vert \cdot \vert r\vert\leq \widetilde M_i(\xi) \cdot \vert r\vert$ for any $g_i\in\partial_{x_i} f(\xbf_1,\xi)$ and $r\in \R$, for every $\xbf_1\in\X$ and $\xi\in\Xi$, where $\hat{\textnormal{e}}_i$ is the $i$th standard basis. Thus, $ \vert f(\xbf_1,\xi)-f(\xbf_2,\xi)\vert\leq  \max_{i=1,...,d} \widetilde M_i(\xi)\cdot \Vert\xbf_1-\xbf_2\Vert_1\leq  e \cdot \Vert \widetilde{\mathbf M}(\xi)\Vert_{\ln d}\cdot \Vert\xbf_1-\xbf_2\Vert_{q'} $, for all $\xbf_1,\xbf_2\in\X$, with $q'=1+(\ln d-1)^{-1}$. 
 Then,  we may invoke Theorem \ref{Most explicit result here}(b) with $M(\xi):=e\Vert \widetilde{\mathbf M}(\xi)\Vert_{\ln d}$,  $p=\ln d$, $q'=1+(\ln d-1)^{-1}$ (which indicates that $(q'-1)^{-1}\leq \ln d$),  $\varphi_{f,p}=\varphi_{f,\ln d}\leq \sqrt{C'\varphi_{f,\infty}^2\ln d}$, as well as   $R^*=\max\{1,\mathcal D_{q'}^2\}\leq    \mathcal D_1^2 $, implies the desired result in Part (b).\hfill\scalebox{0.7}{$\blacksquare$}

 \section{Some useful properties of $V_{q'}$ in Eq \eqref{Eq: SAA-ell2}}\label{sec: preliminary} We   discuss some properties of $V_{q'}$ for $q'\in(1,2]$  in this section. Note that this function
is   differentiable and $({q'}-1)$-strongly convex      w.r.t.\ the $q'$-norm, according to  \cite{ben2001ordered}.  Therefore, 
\begin{align}
V_{q'}(\xbf_1)-V_{q'}(\xbf_2)-\langle\nabla V_{q'}(\xbf_2),\, \xbf_1-\xbf_2\rangle
\geq \frac{{q'}-1}{2}\Vert \xbf_1-\xbf_2\Vert_{q'}^2,~~\forall\,\xbf_1,\,\xbf_2\in\X.\label{SC V}
\end{align}
 As a result of \eqref{SC V}, when   $F_N(\cdot)$ is  convex,   $F_{\lambda_0,N}(\cdot)$ is also  $[\lambda_0\cdot (q'-1)]$-strongly convex w.r.t.\,the same norm.  
% Then, if we recall   that $\xhb$ denotes the minimizer of $F_{\lambda_0,N}$ on $\X$, we have  the  following inequality:  
%  $$
%  F_{\lambda_0,N}(\xbf)-F_{\lambda_0,N}(\xhb)\geq \frac{q'-1}{2}\Vert \xbf-\xhb\Vert_{q'}^2,~~\forall \xbf\in\X.
%  $$

Another important property of  $V_{q'}$ is its Lipschitz continuity under mild conditions. To see this, one may observe that, for any $(\xbf,\,\xbf^0)\in\X^2$,   the below holds with    $\varrho=q'/(q'-1)$:
\begin{align}\Vert \nabla V_{q'}(\xbf)\Vert_\varrho=&\,\Vert \xbf-\xbf^0\Vert_{q'}^{2-q'}\left(\sum_{i=1}^d \left( {\vert x_i-x_i^0\vert}\right)^{(q'-1)\varrho}\right)^{1/\varrho} 
%=\,\Vert \xbf-\xbf^0\Vert_{q'}^{2-q'}\cdot \left(\sum_{i=1}^d \left( {\vert x_i-x_i^0\vert}\right)^{q'}\right)^{(q'-1)/q'}\nonumber
% \\
=\,\Vert \xbf-\xbf^0\Vert_{q'},\label{property of Vq}
\end{align}
Thus, when the $q'$-norm diameter of the feasible region is bounded by $\mathcal D_{q'}$, we can tell that $V_{q'}$ is Lipschitz continuous in the following sense:   
\begin{align}
\vert V_{q'}(\xbf)-V_{q'}(\ybf)\vert\leq \mathcal D_{q'}\cdot \Vert \xbf-\ybf\Vert_{q'},~~\forall \xbf,\ybf\in\X.\label{Lipschitz inequality of Vq}
\end{align}

\section{Comparison with   lower bounds on sample complexity }\label{sec: lower bound}
\subsection{Lower bound on the rate with $d$}
Below we restate   Example 5.21 of \cite{shapiro2021lectures}, which provides a lower sample complexity for SAA. Consider a special SP problem \eqref{Eq: SP problem statement} with $f(\xbf,\xi):=\Vert \xbf\Vert_2^{2m}-2m\xi^\top\xbf$, where $m\geq 1$  and $\X:=\{\xbf\in\R^d:\Vert\xbf\Vert_2\leq 1\}$.  Suppose, further, that random vector $\xi$ has normal distribution $\mathcal N(0,\sigma^2I_d)$, where $\sigma^2\geq 1$ and $I_d$ is the $d$-by-$d$ identity matrix.  For convenience, we let  $d\geq 8$. We use $O(1)$ to denote universal constants.

As per \cite{shapiro2021lectures}, the event that an exact optimal solution to SAA \eqref{Eq: SAA} ensures $\epsilon$-suboptimality with $\epsilon\in(0,1)$ in solving the population-level SP problem  coincides with the event that $\{\chi_{d}^2\leq N\epsilon^{2/v}/\sigma^2\}$, where $v:=\frac{2m}{2m-1}$ and  $\chi_{d}^2$ denotes the chi-square distribution with $d$ degrees of freedom.     Thus, if we let $m=\ln(1/\epsilon)$, one can derive the following lower bound---the sample size is required to satisfy:
\begin{align}
N>\frac{d\sigma^2}{\epsilon^{2/v}}=\frac{d\sigma^2}{\epsilon^{2-\frac{1}{\ln(1/\epsilon)}}}=\frac{ d\cdot \sigma^2}{e\cdot\epsilon^{2}}\label{testnew},
\end{align}
in order to ensure that the SAA solution is an $\epsilon$-suboptimal solution (for $\epsilon\in(0,1)$) to the true problem with probability at least $1-\beta$ (for $\beta\in(0,\,0.3)$).  The above then identifies an ``adversarial'' case of   the sample complexity, showing that SAA's sample complexity  should grow no slower than the rate of $O(d)$ in general.

The lower bound in \eqref{testnew} is not at odds with our results.  For instance, when invoking Part (b) of Theorem \ref{Most explicit result here}, one may let $q'=2$ and $R^*=1$. It then can be verified that $c_l\cdot m^2d\sigma^2\leq \varphi^2\leq  c_u \cdot m^2d\sigma^2$ for some constants $c_u\geq c_l>0$. As a result, the  sample complexity   becomes:
\begin{align}
O(1)\cdot\frac{   m^2d\sigma^2}{ \epsilon^2} \cdot   \ln \left(\Phi_1\ln\frac{e}{\beta}\right)\cdot\ln^2\frac{\Phi_1}{\beta},\label{test bound}
\end{align}
where $\Phi_1\leq \frac{O(1) m  \sigma\sqrt{d}}{ \epsilon}$. Here \eqref{test bound} 
%\textcolor{black}{this becomes 115 since double label} 
is consistent with the   sample complexity lower bound in \eqref{testnew}.  

The analysis above does not rule out the possibility of a better-than-polynomial growth rate with $d$ under some additional problem structural assumption.  Indeed,   we are able to identify an important special case where our new sample bounds can be dimension insensitive and thus surpass \eqref{testnew}, as presented below. 

Consider a modification to the example above---we assume  instead that $\X:=\{\xbf\in\R^d:\Vert\xbf\Vert_{1}\leq 1\}$ and keep all other settings the same. In evaluating the conditions for  Part (b) of Corollary \ref{most explicit}  for this modified example,  we let $R^*=e$ and $q'= 1+( \ln d-1)^{-1}$. One may also verify that
$
 Y_i:= 2m\cdot (1+\vert \xi_i\vert)\geq \sup_{\xbf\in \X}\vert \nabla_i f(\xbf,\xi)\vert=\sup_{\xbf\in \X}\left\vert 2m\Vert\xbf\Vert_2^{2m-2}x_i-2m\xi_i\vphantom{V^{V^V}}\right\vert.$
Hence, for all $t\ge 0$ and all $i=1,...,d$,
\begin{equation}\label{eq:Yi-all-t}
\Prob\{\,Y_i \le t\,\}
\;=\;
\Prob\!\left\{\,2m\bigl(1+|\xi_i|\bigr)\le t\,\right\}
\;=\;
\Prob\!\left\{\,|\xi_i|\le \frac{(t-2m)_+}{2m}\right\}
\;\ge\;
1-2\exp\!\left(-\,\frac{((t-2m)_+)^{2}}{8\,m^{2}\sigma^{2}}\right),
\end{equation}
where $(x)_+=\max\{x,0\}$ and the last inequality applies the Gaussian tail bound
$\Prob\{Z\ge u\}\le e^{-u^{2}/2}$ for $Z\sim\mathcal N(0,1)$, $u\ge 0$.  Consider the case with  $t\ge 4m    \iff t-2m\ge t/2$. Plugging this into
\eqref{eq:Yi-all-t} yields
$
\Prob\{Y_i>t\}
\;\le\; 2\exp\!\left(-\,\frac{(t/2)^{2}}{8\,m^{2}\sigma^{2}}\right)
\;=\; 2\exp\!\left(-\,\frac{t^{2}}{32\,m^{2}\sigma^{2}}\right).$
Meanwhile, consider the case with $0\leq t< 4m$ (in view of $\sigma\geq 1$). We have $\Prob\{Y_i>t\}\leq 1\leq \; 2\exp\!\left(-\,\frac{t^{2}}{32\,m^{2}\sigma^{2}}\right)$. Aggregating the two cases, we then have $2m\cdot (1+\vert \xi_i\vert)$, which overestimates $\sup_{\xbf\in \X}\vert \nabla_i f(\xbf,\xi)\vert$, to satisfy $\Prob\big[2m\cdot (1+\vert \xi_i\vert)\leq t\big]\geq 1-2\exp\!\left(-\,\frac{t^{2}}{32\,m^{2}\sigma^{2}}\right)$ for all $i=1,...,d$.  {\color{black}Invoking   Corollary \ref{most explicit}.(b) with $\varphi_{f,\infty}:= O(1)\cdot  m \sigma$,
 we have that the sample complexity bound  becomes }
\begin{align}
O(1)\cdot\frac{   m^2\sigma^2\ln^{\color{black}{2}} d }{ \epsilon^2} \cdot   \ln \left(\Phi_2\ln\frac{e}{\beta}\right)\cdot\ln^2\frac{\Phi_2}{\beta},\label{test bound v2}
\end{align}
where $\Phi_2\leq \frac{O(1) m\cdot   \sigma\cdot  {\ln d}}{ \epsilon}$. The sample complexity result in \eqref{test bound v2} surpasses the lower bound in \eqref{testnew} with respect to the dependence on $d$. In further contrast, the current SAA's sample complexity benchmark in \eqref{reduced rate} with $\varphi:= O(1)\cdot m\cdot \sigma\cdot \sqrt{\ln d}$ yields the following:
\begin{align}
O(1) \cdot \frac{m^2\sigma^2\ln d}{\epsilon^2} \left[   d\cdot  \ln\left(\frac{O(1)\cdot  m\sigma\ln d}{\epsilon}\right)+\ln\frac{1}{\beta}\right].\label{test conventional bound}
\end{align}
One can see that \eqref{test bound v2} is   more efficient than \eqref{test conventional bound}   for sufficiently large $d$ (with other parameters fixed). 
A closely similar  argument as above can show that our findings are not at odds with, and can sometimes even surpass, the lower sample complexity bounds by \cite{guigues2017non}.
{\color{black}
\subsection{A lower bound on the rate with $\beta$}\label{beta subsection lower bound} 
In heavy-tailed settings when only the second moment of the underlying randomness is bounded, it is possible to show that the rate of $O(1/\beta)$ in Theorems \ref{thm: suboptimality}  and \ref{thm: first main theorem} (with $p=2$) are at least asymptotically optimal. More specifically, we have the following proposition, where we note that an exact solution to SAA refers  to a $(0,\,q)$-approximate solution in the sense of \eqref{suboptimality in computing SAA}. The choice of $q\in(1,2]$ is then arbitrary.
   \begin{proposition}\label{lower bound}
 The following three statements hold:
   \begin{itemize}
   \item[(i)] Consider an exact solution $\widehat \xbf$ to  SAA \eqref{Eq: SAA}. For given   $\epsilon>0$ and   $N\geq 1$, under the same set of assumptions as in Theorem \ref{thm: suboptimality}, there exists  an instance of   \eqref{Eq: SP problem statement} where $ \Prob\left[ F(\widehat \xbf)-F(\xbf^*)\geq \epsilon\right]  \geq \frac{1}{4}\cdot \min\left\{1,\,   \frac{\sigma_p^{2}}{ \mu\cdot\epsilon \cdot  N}\right\}.$
   
    \item[(ii)] Consider an exact solution $\widehat \xbf$ to  SAA \eqref{Eq: SAA-ell2} with hyperparameters specified  as per \eqref{parameter setting}. For a given   $\epsilon>0$  and  $N\geq 1$, under the same set of assumptions as in Theorem \ref{thm: second main theorem convex optimal rate},   there exists an instance of \eqref{Eq: SP problem statement} where $ \Prob\left[ F(\widehat \xbf)-F(\xbf^*)\geq \epsilon\right]  \geq \frac{1}{4}\cdot \min\left\{1,\,   \frac{4}{9}\cdot\frac{\sigma_p^{2}}{ \epsilon^2 \cdot  N}\right\}.$
    
   \item[(iii)] Consider an exact solution $\widehat \xbf$ to  SAA \eqref{Eq: SAA}. For a given   $\vartheta>0$  and   $N\geq 1$, under the same set of assumptions as in Theorem \ref{thm: first main theorem} with $p=2$, there exists an instance of \eqref{Eq: SP problem statement} where $  \Prob\left[ \Vert \widehat \xbf- \xbf^*\Vert_q^2\geq \vartheta\right]  \geq \frac{1}{4}\cdot \min\left\{1,\,   \frac{2\psi_2^{2}}{ \mu^2\cdot\vartheta \cdot  N}\right\}.$
   \end{itemize}
     \end{proposition}
     
     \begin{proof}  Consider $f:\,\R\times \R\rightarrow\R$ defined as $f(\xbf,\xi):=\frac{\mu}{2}(x-\xi)^2$, where  $\xi$ is  a discrete random variable with probability mass function: 
\[\text{$\Prob[\xi=y]=\varpi/2$ for $y=\pm\sqrt{2}\cdot(1+\lambda_0/\mu)\cdot  N\sqrt{\epsilon/ \mu}$ and $\Prob[\xi=y]=1-\varpi$ for $y=0$}.\] 
Since $\E[\xi]=0$, the SP problem in \eqref{Eq: SP problem statement} is then substantiated into:
     \begin{align}
     \min_{\xbf\in\R}F(\xbf):=\E\Big[ \frac{\mu}{2}(\xbf-\xi)^2\Big]=\E\Big[ \frac{\mu}{2}(\xbf^2+\xi^2)\Big]\label{adversarial SP}
     \end{align}
  We further let $\varpi:=\min\left\{0.5/N,\,{\sigma_p^2}\cdot [2\mu\cdot (1+\lambda_0/\mu)^2 N^2\epsilon]^{-1}\right\}$  for $p=q/(q-1)$.  It is easy to verify that, by construction, $f(\cdot,\xi)$ is $\mu$-strongly convex w.r.t. the $q$-norm for any choices of $q\in(1,2]$. Hence, Assumption \ref{SC condition constant all} holds and the optimal solution to \eqref{adversarial SP} is $\xbf^*=0$. Meanwhile, one may easily derive that   $\nabla F(\xbf):=\mu \xbf$ and $\nabla f(\xbf,\xi)={\mu}\cdot (\xbf-\xi)$. As a result, Assumption \ref{assumption: Variance everywhere} holds with
 \begin{align}\sup_{\xbf\in\R}\E\left[\Vert \nabla F(\xbf)-\nabla f(\xbf,\xi)\Vert^2_p\right]=\mu^2\E[(\xi-\E[\xi])^2]\leq\sigma_p^2.\label{sigma result}
 \end{align}
 
 % Then, we may calculate variance as $E[(\xi-\E[\xi])^2]=\sigma^2$.

     Let $\xi_1,...,\xi_N$  be i.i.d. copies   of $\xi$.  For any given $\lambda_0\in[0,0.5 \epsilon]$, consider    
     \begin{align}
 \min_{\xbf\in\R}  N^{-1} \sum_{j=1}^N \frac{\mu}{2}(\xbf-\xi_j)^2+\frac{\lambda_0}{2}\cdot\xbf^2.\label{template SAA}
 \end{align}
 This problem subsumes both   SAAs \eqref{Eq: SAA} and \eqref{Eq: SAA-ell2} for the following configurations:
 \begin{itemize}
 \item \eqref{template SAA} becomes  SAA \eqref{Eq: SAA} when $\lambda_0=0$;
 \item  \eqref{template SAA}  becomes SAA  \eqref{Eq: SAA-ell2}  when $\lambda_0=0.5\epsilon/R^*$. Here,  w.l.o.g., we let  $\xbf^0=0$ for the Tikhonov-like regularization. 
 \end{itemize}
 A closed-form optimal solution to \eqref{template SAA} writes as: $\widehat \xbf=(1+\lambda_0/\mu)^{-1}\cdot N^{-1} \sum_{j=1}^N  \xi_j$. We know that 
    $
     F(\widehat \xbf)-F(\xbf^*)=\frac{\mu}{2}\cdot\widehat \xbf^2=\frac{\mu}{2\cdot(1+\lambda_0/\mu)^2}\cdot \left(N^{-1} \sum_{j=1}^N  \xi_j\right)^2$.
  Thus, the event that $\left\{\left\vert N^{-1} \sum_{j=1}^N  \xi_j\right\vert\geq \sqrt{2}\cdot(1+\lambda_0/\mu)\cdot   \sqrt{\epsilon/ \mu}\right\}$ is equivalent to  $\{F(\widehat \xbf)-F(\xbf^*)\geq \epsilon\}$. Denote by $K\sim Binomial (N,\varpi)$. Note that
     \begin{align}
   &  \Prob[F(\widehat \xbf)-F(\xbf^*)\geq \epsilon]=\Prob\left[\frac{\mu}{2} \xhb^2\geq  \epsilon\right]\nonumber
     \\=& \Prob\left[\left\vert N^{-1} \sum_{j=1}^N  \xi_j\right\vert\geq \sqrt{2}\cdot(1+\lambda_0/\mu)\cdot   \sqrt{\epsilon/ \mu}\right] \nonumber
     \\=& \Prob\left[\left\vert \{i:\,\xi_i=\sqrt{2}\cdot(1+\lambda_0/\mu)\cdot  N\sqrt{\epsilon/ \mu}\}\right\vert-\left\vert \{i:\,\xi_i=-\sqrt{2}\cdot(1+\lambda_0/\mu)\cdot  N\sqrt{\epsilon/ \mu}\}\right\vert\neq 0\right]  \nonumber
     \\\geq& \Prob\left[K=1\right]  =1-\Prob[K=0]-\Prob[K\geq 2] \nonumber
      \\\geq & 1-(1-\varpi)^N-{N \choose 2}\cdot \varpi^2 \overset{\substack{\text{\tiny Second-order}\\\text{\tiny Bonferroni}}}{\geq} N\varpi-2{N \choose 2}\cdot \varpi^2. \label{second last step lower bound}
      \end{align}
  By the construction of $\varpi$, we have $N\varpi=\min\left\{0.5,\,\frac{\sigma_p^{2}}{(2\mu\cdot\epsilon)\cdot (1+\lambda_0/\mu)^{2} N}\right\}$ and $2{N \choose 2}\cdot \varpi^2=N(N-1)\cdot \left(\min\left\{0.5/N,\,\frac{\sigma_p^{2}}{(2\mu\cdot\epsilon)\cdot (1+\lambda_0/\mu)^{2} N^2}\right\}\right)^2\leq 0.5\cdot    \min\left\{0.5,\,\frac{\sigma_p^{2}}{(2\mu\cdot\epsilon)\cdot (1+\lambda_0/\mu)^{2} N}\right\}$. This combined with Eq.\,\eqref{second last step lower bound} then implies that 
 \begin{align}\Prob\left[ F(\widehat \xbf)-F(\xbf^*)\geq \epsilon\right] =\Prob\left[\left\vert N^{-1} \sum_{j=1}^N  \xi_j\right\vert\geq \sqrt{2}\cdot(1+\lambda_0/\mu)\cdot   \sqrt{\epsilon/ \mu}\right]  \geq \frac{1}{4}  \min\left\{1,\,  \frac{\sigma_p^{2}}{\mu \epsilon  (1+\lambda_0/\mu)^{2} N}\right\}\label{general result lower bound}
 \end{align}

\noindent [{\it For Part (i)}]:
This part of the proposition considers SAA \eqref{Eq: SAA}; namely,  $\lambda_0=0$. Then, \eqref{general result lower bound} immediately implies the desired result.
\smallskip

\noindent [{\it For Part (ii)}]: This part of the proposition considers SAA \eqref{Eq: SAA-ell2} for $\lambda_0$ specified as in \eqref{parameter setting}, which is consistent with  selecting $R^*\geq 1$ and $\lambda_0\in(0,\, \epsilon/2]$. Let $\mu=\epsilon$; and thus, verifiably,  $f(\cdot,\xi)$  satisfies Assumption \ref{GC condition constant all}.
Then invoking \eqref{general result lower bound}  with $\frac{1}{4}  \min\left\{1,\,  \frac{\sigma_p^{2}}{\mu \epsilon  (1+\lambda_0/\mu)^{2} N}\right\}\geq \frac{1}{4}  \min\left\{1,\, (1+0.5)^{-2}\cdot \frac{\sigma_p^{2}}{\mu \epsilon   N}\right\}$ implies   Part (ii) of this proposition.

 %according to \eqref{general result lower bound} 
\smallskip

\noindent [{\it For Part (iii)}]:
Let $\lambda_0=0$. One may easily verify that Assumption \ref{SC condition constant}.(a) holds by construction. Eq.\,\eqref{sigma result} implies Assumption \ref{SC condition constant}.(b) (with   $p:=2$) and $\psi_2:=\sigma_p$. Since $ \Prob\left[ (\widehat \xbf-\xbf^*)^2\geq 2\epsilon/\mu\right] =\Prob\left[\left\vert N^{-1} \sum_{j=1}^N  \xi_j\right\vert\geq \sqrt{2 \epsilon/ \mu}\right],$ we  let $\epsilon:=\vartheta\mu/2$ in \eqref{general result lower bound} to show the claimed.
\end{proof}

    \bigskip
      
 The proposition above shows that  the rate of $O(1/\beta)$   in our sample complexity  bounds  in Theorems \ref{thm: suboptimality}, \ref{thm: second main theorem convex optimal rate}, and \ref{thm: first main theorem} (for $p=2$) is intrinsic to the SAA,  rather than a proof artifact. 
}
\vspace{-3mm}
{
\color{black}
\section{Additional details on hyperparameter selection for SMD}\label{details SMD stepsize}
\sloppy\raggedbottom
Recall that $G(\xbf,\xi)$ denotes a subgradient of $f(\cdot,\xi)$ at $\xbf$. In determining $\widetilde M_\infty$ for SMD-L$_1$, we approximated $\E[\sup_{\Vert \xbf \Vert_1\leq R_{\ell_1}}\Vert G(\xbf,\xi)\Vert_\infty]$ by $\frac{1}{500}\sum_{j=1}^{500} \sup_{  \Vert \xbf \Vert_1\leq R_{\ell_1}}\Vert  G(\xbf,\xi_j)\Vert_\infty$, where $\xi_j$, $j=1,...,500$, is an i.i.d. sample of $\xi$. Likewise, $\widetilde M_2$ was estimated by $\frac{1}{500}\sum_{j=1}^{500} \sup_{  \Vert \xbf \Vert_2\leq R_{\ell_2}}\Vert G(\xbf,\xi_j)\Vert_2$. Meanwhile, we performed cross-validation similar to Section \ref{sec: experiment on light-tailed problem} in  selecting  $\theta$ from a set of candidate  values $\big\{a\cdot b:\,a\in\{1,2,...,9\},\,b\in\{0.1,\,1,\,10,\,100,\, 1000\}\big\}$. More specifically, we set    $d=1000$  and $N=600$ and simulated two i.i.d. random samples of $\xi$---one for optimization of \eqref{stochastic nonsmooth convex problem} and one for ``validation''. An SMD variant was invoked to operate using the ``optimization''  set. The generated solution, denoted by $\boldsymbol x_{\theta}$,    was then evaluated on the validation set, denoted by $\widehat\xi_j$, $j=1,...,600$,  in terms of the average cost value calculated as $\frac{1}{600}\sum_{j=1}^{600}f( \boldsymbol x_{\theta},\widehat\xi_j)$. Fixing   $\theta$ to  each of its candidate values, we repeated the process above for five independent runs. The  $\theta$-value that led an SMD variation to generate the best (smallest) average validation performance was  then selected for the corresponding SMD variation. % for some $\xi_j$

%Given the selected $\theta$ and the step size rules as the above, %For all configurations, we solved the corresponding formulation with the same sample size $N$. We kept the pipeline for choosing $\lambda_0$ as our linear regression experiment in Section \ref{subsec: linear regression}. With the setups above, our  experiment was to evaluate how the performance, in terms of the resulting suboptimality gaps and solution distances, of the aforementioned schemes  evolved as $d$ increased from 100 to 5000 for different sample sizes $N\in\{200,\,400,\,600\}$. Here, the suboptimality gap for a solution $\xbf$ can be validated through SAA formulation \eqref{Eq: SAA} with very large sample size $N_v$, as $1/N_v\sum_{j=1}^{N_v}(f(\xbf,\xi_j)-f(\xbf^*,\xi_j))$. We solve for the corresponding optimal solution $\xbf^*$ to \eqref{Eq: SP problem statement} using the SAA formulation in \eqref{Eq: SAA}, with $N=20,000$ and initial solution generated uniformly randomly on $[-0.5,0.5]^d$.
}

\end{appendices} 

\end{document}